\documentclass[UKenglish]{fundam}

\publyear{25}
\papernumber{3}
\volume{194}
\issue{2}
\theDOI{10.46298/fi.12773}

\versionForARXIV

\usepackage[utf8]{inputenc}
\usepackage{babel,csquotes}
\usepackage{xparse}			%enable better macro definitions
\usepackage{amssymb}
\usepackage{thmtools,thm-restate}
\usepackage{mathrsfs}
\usepackage{enumitem}
\setlist[description]{style=sameline}
\newlist{pcases}{description}1
\setlist[pcases]{nosep,labelindent=\parindent,labelwidth=8pt,leftmargin=0pt,style=sameline}
\usepackage{mathtools}
\usepackage{stackrel}
\usepackage{booktabs}

\usepackage{ebproof}
\ebproofset{
	left label template={\normalfont\small (\inserttext)},
	right label template={\small \inserttext},
	rule margin=.5ex,% default .7ex
	separation=1.5em,% default 1.5em
	label separation=.5ex,% default .5ex
}
\ebproofnewstyle{small}{
  template = \small$\inserttext$,
  left label template = \footnotesize{(\inserttext)},
  right label template = \footnotesize{\inserttext} }
\ebproofnewstyle{smaller}{
  template = \footnotesize$\inserttext$,
  left label template = \scriptsize{(\inserttext)},
  right label template = \scriptsize{\inserttext} }
\ebproofnewstyle{compact}{
  separation = 1.5ex, rule margin = .25ex,}
\ebproofnewstyle{tight}{
  separation = 1em, rule margin = .4ex,
  }
\newcommand{\hypod}{\hypo{\dotsm}}
\newcommand{\axiom}[2][]{\hypo{}\infer1[#1]{#2}}
\newcommand{\bud}[2][]{\hypo{\vphantom{x}#1}\infer1[\hyp]{#2}}
\NewDocumentCommand{\subproof}{sO{}m}{%
	\IfBooleanTF{#1}{%
		\hypo{#2}\infer[no rule]1{#3}
	}{%
	\hypo{}\ellipsis{$#2$}{#3}
	}
	}
\providecommand{\subproof}[2][]{\hypo{#1}\infer[no rule]1{#2}}

\usepackage[bookmarks, colorlinks = true, linkcolor = black, urlcolor = black, citecolor = black, anchorcolor = black]{hyperref}
\let\phi\varphi
\let\epsilon\varepsilon

\newcommand{\card}[1]{{\lvert{#1}\rvert}} % cardinality (of set)
\newcommand{\lh}[1]{{\lvert{#1}\rvert}} % length (of sequence)
\newcommand{\rk}[1]{{\lvert{#1}\rvert}} % rank (of formula/construction)
\newcommand{\Rk}[1]{{\bigl\vert{#1}\bigr\vert}} % rank : big
\NewDocumentCommand{\setof}{ m o }{ \{ #1 \IfValueT{#2}{ \mid #2}  \} }
\NewDocumentCommand{\setoff}{ m o }{ \{ \, #1 \,\} \IfValueT{#2}{ _{#2} } }
\NewDocumentCommand{\Setof}{ m o }{\bigl\{ #1 \IfValueT{#2}{ \bigm| #2} \bigr\} }
\NewDocumentCommand{\SetOf}{ m o }{\Bigl\{ \, #1 \IfValueT{#2}{ \Bigm| #2} \, \Bigr\} }
\NewDocumentCommand{\SETOF}{ m o }{\Biggl\{\, #1 \IfValueT{#2}{ \Biggm| #2} \,\Biggr\} }
\providecommand{\setof}[1]{\{ #1 \}}
\providecommand{\Setof}[1]{\{ #1 \}}

\newcommand{\tuple}[1]{\langle #1 \rangle}
\newcommand*{\Child}[2]{\mathsf{child}_{#1}(#2)}
\newcommand*{\Suc}[2][]{\mathsf{child}_{#1}(#2)}
\newcommand{\res}{\mathord{\uparrow}}
\newcommand{\Lit}{\mathsf{Lit}}
\newcommand{\Var}{\mathsf{Var}}

\newcommand{\form}{\mathsf{L}_\mu}

\providecommand{\D}{}
\renewcommand{\D}{\mathbb{D}}
\newcommand{\PF}[1]{\Pi({#1})}
\newcommand{\PD}{\PF\D}

\NewDocumentCommand{\bo}{m}{[ ]}%{\oldsquare_{#1}}%{[\act #1]}
\NewDocumentCommand{\di}{m}{\langle \rangle}%{\langle\act #1\rangle}

\let\diamond\diamondsuit
\newcommand{\conj}{{\textstyle\bigwedge}}
\newcommand{\disj}{{\textstyle\bigvee}}
\NewDocumentCommand{\supp}{O{}}{\circ_{#1}}%{\nabla}
\NewDocumentCommand{\mf}{o}{\mu\IfValueT{#1}{^{#1}}}% #3
\NewDocumentCommand{\nf}{o}{\nu\IfValueT{#1}{^{#1}}}% #3
\NewDocumentCommand{\mnf}{o}{\sigma\IfValueT{#1}{^{#1}}}% #3

\providecommand{\sf}{}
\NewDocumentCommand{\oldsf}{s}{\triangleright\IfBooleanT{#1}{^1}}
\NewDocumentCommand{\nsf}{s}{\mathrel{\not\vartriangleright}\IfBooleanT{#1}{^1}}
\RenewDocumentCommand{\sf}{s}{\vartriangleright\IfBooleanT{#1}{^*}}

\NewDocumentCommand{\sft}{}{\vartriangleright^*}

\NewDocumentCommand{\var}{m}{#1} % was \mathsf{#1}

\NewDocumentCommand{\subform}{om}{{\mathsf{Sub}\IfValueT{#1}{_{#1}}(#2)}}
\providecommand{\subform}[1]{\mathsf{Sub}(#1)}

\NewDocumentCommand{\oldsubs}{mmo}{{{#1}({#2}\IfNoValueF{#3}{/\var{#3}})}}	% a(b/x) = \subs{a}{b}[x]
\NewDocumentCommand{\subs}{om}{{[{\IfNoValueF{#1}{#1 \mapsto} #2}]}}	% \subs[x]{b} = (b/x)
\providecommand{\subs}[2][x]{[{#2}/{#1}]}

\renewcommand{\neg}[1]{{#1}^\bot}
\providecommand{\denote}[1]{[\!\lvert#1\rvert\!]}
\providecommand{\oper}[1]{\langle\!\rvert{#1}\rvert\!\rangle}
\newcommand{\Mod}{\mathsf{Mod}}
\newcommand{\nms}{\mathsf{N}}
\newcommand{\ann}{\mathsf{A}}
\providecommand{\Tr}[2][]{\mathsf{Thd}_{#1}(#2)}
\providecommand{\nTr}[2][]{\mathsf{Thd}^\bot_{#1}(#2)}
\providecommand{\thd}{\mathsf{Thd}}
\newcommand{\system}[1]{\mathsf{#1}}
\newcommand{\systemc}[1]{\mathsf{c#1}}
\newcommand{\kmu}{\system{K\mu}}
\newcommand{\kmm}{\system{K\bar\mu}}
\newcommand{\kmo}{\system{K\mu_\omega}}
\newcommand{\cmu}{\system{C\mu}}
\newcommand{\kmui}{\systemc{K\mu}}
\newcommand{\kmmi}{\systemc{K\bar\mu}}
\newcommand{\cmui}{\systemc{C\mu}}
\newcommand{\kmoi}{\systemc{K\mu_\omega}}
\NewDocumentCommand{\rul}{m}{\ensuremath{\mathsf{#1}}}
\newcommand{\rules}{\mathscr{R}}
\newcommand{\sequents}{\mathsf{Seq}}
\newcommand{\sequent}{\mathscr{S}}

\newcommand{\exc}{\rul{e}}
\newcommand{\con}{\rul{c}}

\newcommand{\weak}{\rul{w}}
\newcommand{\cut}{\rul{cut}}
\newcommand{\mcut}{\rul{mcut}}
\newcommand{\weakL}{\rul{wL}}
\newcommand{\weakR}{\rul{wR}}
\newcommand{\excL}{\rul{eL}}
\newcommand{\conL}{\rul{cL}}
\newcommand{\excR}{\rul{eR}}
\newcommand{\conR}{\rul{cR}}

\newcommand{\hyp}{\rul{b}}

\newcommand{\idL}{\rul{idL}}
\newcommand{\idR}{\rul{idR}}
\newcommand{\idLR}{\rul{id}}

\newcommand{\id}{\rul{id}}

\newcommand{\sqR}{\rul{{\square}R}}
\newcommand{\diL}{\rul{{\diamond}L}}

\newcommand{\starR}{\rul{{\star}R}}
\newcommand{\starL}{\rul{{\star}L}}

\newcommand{\conjR}{\rul{{\conj}R}}
\newcommand{\conjL}{\rul{{\conj}L}}
\newcommand{\disjR}{\rul{{\disj}R}}
\newcommand{\disjL}{\rul{{\disj}L}}

\newcommand{\muR}{\rul{{\mu}R}}
\newcommand{\muL}[1][]{\ensuremath{{\mu}^{#1}\mathsf{L}}}
\newcommand{\nuR}[1][]{\ensuremath{{\nu}^{#1}\mathsf{R}}}
\newcommand{\nuL}{\rul{{\nu}L}}
\newcommand{\sigR}[1][\sigma]{\rul{{#1}R}}
\newcommand{\sigL}[1][\sigma]{\rul{{#1}L}}

\newcommand{\sqRs}{\rul{{\square^*}R}}
\newcommand{\diLs}{\rul{{\diamond^*}L}}
\newcommand{\conjsL}{\rul{{\conj^{\!*}}L}}
\newcommand{\muhL}[1][]{\ensuremath{{\hat\mu}^{#1}\mathsf{L}}}
\newcommand{\nuhR}[1][]{\ensuremath{{\hat\nu}^{#1}\mathsf{R}}}

\newcommand{\cw}{\rul{cw}} % control weakening
\newcommand{\aw}{\rul{aw}} % ann weakening
\newcommand{\awL}{\rul{aL}} % ann weakening
\newcommand{\awR}{\rul{aR}} % ann weakening
\newcommand{\dup}{\rul{dup}}
\providecommand{\cover}{\rul{cvr}}

\newcommand{\amuL}[1]{\ensuremath{\mu^{#1} \mathsf L}}
\newcommand{\anuR}[1]{\ensuremath{\nu^{#1} \mathsf R}}

\NewDocumentCommand{\seq}{omm}{ {\IfNoValueTF{#1}{}{#1 : }#2\Rightarrow #3} }
\newcommand*{\aseq}[3][c]{\seq[#1]{#2}{#3}}
\newcommand*{\tseq}[2]{\seq[#1]{#2}{}}
\NewDocumentCommand{\Seq}{mm}{ #1\Vdash #2 }

\newcommand*{\rseq}[1]{\seq{}{#1}}
\newcommand*{\lseq}[1]{\seq{#1}{}}

\newcommand{\VdashV}{\mathrel{\dashv\hskip.4pt\Vdash}}

\newcommand{\indL}{\rul{indL}}
\newcommand{\indR}{\rul{indR}}
\newcommand{\sindL}{\rul{sindL}}
\newcommand{\sindR}{\rul{sindR}}
\newcommand{\ind}{\rul{ind}}

\newcommand{\sind}{\rul{sind}}

\NewDocumentCommand{\inv}{m}{{\mathsf{inv}(#1)}}
\renewcommand{\c}[1]{\mathsf{c}_{#1}}

\begin{document}
% ======================
\title{Demystifying $\boldsymbol\mu$\footnotetext{\textbf{ACM CCS} 
\textbullet~Theory of computation $\looparrowright$ Modal and temporal logics
 \textbullet~Theory of computation $\looparrowright$ Proof theory
}
}

\author{Bahareh Afshari\thanks{This work was supported by the Knut and Alice Wallenberg Foundation [2020.0199] and Dutch Research Council [OCENW.M20.048]. The authors are deeply grateful for the substantial feedback and highly constructive comments received from the referees during the reviewing process of the article; special thanks also go to Sebastian Enqvist for feedback on an earlier version.}\\
  Department of Philosophy, Linguistics and Theory of Science \\
  University of Gothenburg
\and Graham E.~Leigh\thanksas{1}\\
 Department of Philosophy, Linguistics and Theory of Science \\
  University of Gothenburg
\and Guillermo Men\'endez Turata\\
Institute for Logic, Language and Computation \\
University of Amsterdam
}
\address{B.~Afshari, University of Gothenburg, Box 200, 405 30 Göteborg, Sweden.}

\maketitle

\runninghead{B.~Afshari, G.E.~Leigh, G.~Men\'endez Turata}{Demystifying $\mu$}

\begin{abstract}
We explore the theory of illfounded and cyclic proofs for the propositional {modal $\mu$-calculus}. 
A fine analysis of {provability} for classical and intuitionistic modal logic provides a novel bridge between finitary, cyclic and illfounded conceptions of proof and re-enforces the importance of two normal form theorems for the logic: guardedness and disjunctiveness.
\end{abstract}

%----------------------------------------
\section{Introduction}
%----------------------------------------
Formal proofs are an important tool in the mathematical study of computational systems, such as certifying the correctness of an abstract model or verifying that an implementation adheres to a specification.
The certification and verification role of proofs boils down to questions of \emph{proof existence} and \emph{proof synthesis}: 
Does a given formula (sequent, judgement, etc.) admit a proof?
Can a proof be generated for each provable formula (sequent, etc.)?
Both questions rest heavily on a third, often unspoken: 
What constitutes a proof?

A \emph{proof} is normally understood as a finite tree with vertices labelled from a class of \emph{deduction elements} (formulas, sequents, judgements, etc.) where each vertex together with its children match the conclusion and premises of one of a fixed collection of \emph{inferences}, the decision of which should be complexity-theoretically simple, at worst polynomial-time decidable.
This definition is computationally sufficient for many logics such as propositional logic, modal logics, even predicate logic. 
In each case, soundness and completeness theorems for the corresponding notion of proof confirm that proof existence agrees with semantic validity, and normal form theorems -- notably admissibility of cut -- give rise to proof synthesis.

Logics and theories incorporating inductive and co-inductive concepts strain the traditional notion of proof, however.
Proof existence relies on more complex completeness theorems which coax the infinitary behaviour of the inductive and co-inductive constructions into finitary induction principles.
Synthesis likewise suffers, becoming a task in discerning induction invariants from mere provability or elaborating the proof calculus to recover desirable normal forms.
The latter is the typical approach of sequent calculi for modal logics where inductive properties -- transitivity, factivity, well-foundedness, etc.\ -- are internalised in the inferences rules.

\emph{Illfounded proofs} offer to alleviate the proof-theorist's burden not by complicating the notion of proofs but by \emph{relaxing} it.
Proofs are no longer constrained to finite trees; they can be infinite trees, or even graphs, at the cost of infinite branches fulfilling some pre-determined correctness condition.
Illfounded proofs provide an alternative to the traditional notion of proof that treats logics and theories of inductive concepts as manifestations of a general infinitary framework rather than diverging extensions of finitary logic.
With a change in perspective comes the need to revisit the encompassing theory.
Indeed, the inductive construction of formal proof plays a non-trivial role in most -- if not all -- fundamental results of well-founded proof theory: proof transformations, cut elimination, computational interpretations, syntactic approaches to interpolation and so on.

This article seeks to further the development of the theory of illfounded proofs. 
We do so in the context of the \emph{modal $\mu$-calculus}, the extension of propositional modal logic by quantifiers for inductive and co-inductive properties.
The logic encompasses the vast majority of modal and temporal logics used in specification and verification. 
 As well as being highly expressive, it enjoys good meta-logical properties, namely, decidability, finite model property and uniform interpolation. 
 It serves as a natural metalanguage for comprehending the dynamics and dependencies between induction and co-induction without obscuring its core fragments.

The first syntactic characterisation of the modal $\mu$-calculus was, in retrospect, a system of illfounded proofs~\cite{JanWal95,NiwWal:96}. 
Completeness of a \emph{finitary} calculus materialised only as a consequence of an intricate analysis of the infinitary system~\cite{Waluk95compl-lics,Waluk00comp-lc}.
Kozen's finitary axiomatisation~\cite{Koz83} may provide the `correct' notion of proof according to the traditional view, yet it is the infinitary notion which has proved more relevant for, and amenable to, investigation. 
The median of these extremes are the \emph{cyclic proofs}, illfounded proofs whose infinite expression is confined to periodic patterns, that is, unfoldings of finite graphs into trees.
Such proofs display many of the advantages of both formalisms.
Cyclic illfounded proofs are known to be complete for the modal \( \mu \)-calculus and can be used to show decidability and the {small model property}~\cite{NiwWal:96}.
With refinement, cyclic proof calculi have been developed for deterministic proof-search and representation~\cite{jungteerapanich:article,stirling2014tableau,AL16-mfo,DKMV23-det,LeiWehr23GTCtoReset} and interpolation~\cite{AL22lyndonInterp,ALM21uniformInter,Shamkanov:2014Circular}.
The first purely proof-theoretic investigation of illfounded, cyclic and wellfounded notions of proof for the modal \( \mu \)-calculus was undertaken by Studer~\cite{Studer:2009Proof-Theoretic}.

Reflecting on the results and techniques mentioned above, we examine the connection between illfounded, cyclic and finitary proof systems for the modal $\mu$-calculus in the context of completeness for each notion of proof.
Two novel ingredients of the study are a focus on the \emph{constructive} modal \( \mu \)-calculus as a framework for proving completeness of the classical logic proof systems and the role of \emph{admissibility}, notably cut admissibility, in tying the role of standard normal form theorems to completeness arguments.

By \emph{constructive modal \( \mu \)-calculus} we mean  Wijesekera's \emph{constructive modal logic}~\cite{Wijesekera:1990Constructive},  conceived by Fitch~\cite{Fitch:1948Intuitionistic}, extended by operators for least and greatest fixed point of all positive propositional functions.
Constructive modal logic is an intuitionistic realisation of classical modal logic \( \system{K} \) with a natural proof-theoretic rendering via a sequent calculus for modal logic restricted to `intuitionistic' sequents \( \seq \Gamma \alpha \).
Indeed, we take this `proof-theoretic' presentation as the definition of the constructive counterpart of the illfounded and cyclic proof systems.
The logic is strictly contained in what is commonly called \emph{intuitionistic modal logic}~\cite{Plotkin:1986A-framework,Simpson:1994The-Proof}.\footnote{
For a comparison of modal logics with an intuitionistic base see, e.g.,~\cite{Das:2023On-Intuitionistic}.}
As our employment of constructive modal logic is as a tool for analysing provability in \emph{classical} logic, we restrict attention to a fragment expressively complete over the classical modal \( \mu \)-calculus, namely,  (intuitionistic) sequents of implication-free formulas allowing negated atoms.
We refer to this as the \emph{constructive fragment} of the classical modal \( \mu \)-calculus and leave open the extent to which the presented results, specifically completeness and normal form theorems, can be lifted to the `full' constructive modal \( \mu \)-calculus.
Our main results concerning the constructive fragment specifically is the admissibility of cut over illfounded proofs and the extension of known normal form theorems from classical modal logic---guardedness and disjunctivity---to the constructive fragment.
Admissibility of cut is proved via a co-inductive argument following established techniques~\cite{ForSan13,BDS16,Savateev:2017Cut-Elimination,Baelde:2022Bouncing,Sau23,Acclavio:2024Infinitary,AfsLei24-weyl}.
The \emph{disjunctive normal form theorem}, first established in~\cite{JanWal95}, is a powerful tool in the study of classical modal fixed point logics.
The result states that every formula is equivalent to a formula in which all uses of conjunction are constrained to specific combinations of the modal operators and atoms.
Satisfiability and validity checking for disjunctive formulas (over classical logic) is linear-time, and completeness of both the illfounded and finitary proof systems is trivial for this fragment.
The classical definition of disjunctive formulas is insufficient as a normal form over intuitionistic logic.
A modest generalisation, however, yields a notion of disjunctive formula which is expressively adequate over the constructive fragment.
As an immediate corollary of the expressive adequacy result is that the constructive modal \( \mu \)-calculus satisfies the uniform (Lyndon) interpolation property.

The main advantage of constructive modal logic and, in particular, its robust sequent calculi lies in an unexpected place, namely the reduction of illfounded provability to cyclic provability, specifically in isolating well-structured finite representations of illfounded proofs.
This way we establish a partial completeness theorem for the system of \emph{constructive} cyclic proofs (sans implication) which, paired with the aforementioned admissibility and normal form theorems, bootstraps the completeness argument for the full, classical, logic.
A relatively straightforward translation confirms that our classical system of cyclic proofs coincides with the finitary one, namely Kozen's axiomatisation of the propositional modal \( \mu \)-calculus.

The reader familiar with proofs of completeness for \( \mu \)-calculi will notice a similarity in our account to other completeness theorems for \( \mu \)-calculi~\cite{Waluk00comp-lc,EnqSV18-dynamics,Kaivola:1995Axiomatising,Doumane:2017Constructive}.
In broad strokes, our argument follows the Kozen--Walukiewicz strategy~(\cite{Waluk00comp-lc}) of establishing completeness for an expressively adequate fragment of the logic (negated disjunctive formulas) and showing that each formula is provably equivalent (over Kozen's axiomatisation) to a formula in the fragment.
Of the two parts of the proof, the second is the more challenging and involves identifying a tight connection between formulas and carefully chosen disjunctive equivalents.
The connection induces a consequence relation, referred to as \emph{tableaux consequence} in~\cite{Waluk00comp-lc}, which refines provable implication in certain contexts via a reduction already isolated by Kozen~\cite{Koz83}.
Our completeness argument forgoes such a consequence relation by showing that the reduction of the full logic to the disjunctive fragment is \emph{constructive} in the sense described above.
Indeed, we suspect that Walukiewicz' tableaux consequence, as well as the closely related \emph{game consequence} relation of  Enqvist et al~\cite{EnqSV18-dynamics}, realise intensional forms of constructive implication. 
Likewise, it is reasonable to ask whether the completeness proofs for linear-time \( \mu  \)-calculus due to Kaivola~\cite{Kaivola:1995Axiomatising} and Doumane~\cite{Doumane:2017Constructive}, which follow strategies similar to Walukiewicz~\cite{Waluk00comp-lc}, can also be reconceptualised as completeness via constructive logics.

% --------------
\subsection*{Outline of paper}
% --------------
Section~\ref{s-prelim} introduces the modal \( \mu \)-calculus, its semantics and the three fragments of importance to this article: guarded formulas, disjunctive formulas and the \( \Pi \)-fragment which comprises all formulas lacking the least fixed point quantifier.
The finitary and illfounded sequent calculus are defined in Section~\ref{s-proof-systems}. 
The finitary system is a sequent calculus rendering of Kozen's axiomatisation of the modal \( \mu \)-calculus.
The system of illfounded proofs is analogous to Niwinski and Walukiewicz' refutation tableaux system from~\cite{NiwWal:96}.

Section~\ref{s-constr-frag} introduces the constructive fragment of the illfounded calculus and establishes preliminary results, including the admissibility of cut and the expressive adequacy of the guarded fragment over intuitionistic logic.
The cyclic calculus is introduced in
Section~\ref{s-cyclic-proofs}, where it is related to the finitary and illfounded proof systems.
It is shown that cyclic proofs for classical logic coincides with Kozen's axiomatisation, and that the (cut-free) constructive fragment coincides with the illfounded proofs on implications \( \alpha \to \beta \) where \( \alpha \) is restricted to the \( \Pi \)-closure of the disjunctive formulas.

Section~\ref{s-DNF} treats the second normal form theorem, the expressive adequacy of the disjunctive formulas. 
We establish a strong form of the result in that every formula is provably equivalent to a disjunctive formula via a constructive illfounded proof.
We remark on two corollaries of the disjunctive normal form theorem.
The first is completeness of the constructive illfounded proofs for negative sequents (sequents \( \seq \alpha \bot \)), which will be utilised later. 
The second corollary is the observation that the method of assigning disjunctive formulas can be relativised to a specific selection of atoms, giving a novel proof of uniform Lyndon interpolation for the constructive fragment.

The three calculi finally converge in Section~\ref{s-completeness} where the normal form theorems and partial completeness results for constructive logics combine to establish completeness of the classical illfounded, cyclic and finitary calculi.

%----------------------------------------
\section{Modal \texorpdfstring{$\boldsymbol\mu$}{mu}-Calculus: Syntax and Semantics}
\label{s-prelim}
%----------------------------------------
Fix a countable set \( \Var \) of \emph{variables}. 
Symbols \( x,y,\dotsc \) serve as meta-variables for variables.
We assume a set \( \neg \Var = \setof{ \lnot x }[x\in \Var] \) of symbols, called \emph{negated variables}, disjoint from \( \Var \)
Elements of \( \Lit = \Var \cup \neg\Var \) are \emph{literals}.
The \emph{formulas} and \emph{predicates} of the modal \( \mu \)-calculus are constructed from literals via the grammars
\[\label{e-the-grammar}
	\begin{aligned}
\text{Formulas:}&&
\alpha &\Coloneqq 
	x \mid 
	\lnot x \mid 
	\square  \alpha \mid 
	\diamond \alpha \mid 
	\conj \Gamma \mid
	\disj \Gamma \mid 
	\mf \phi \mid 
	\nf \phi 
	\\
\text{Finite sets:}
&&
\Gamma &\Coloneqq  
	\emptyset \mid
	\Gamma \cup \setof{ \alpha }
	\\
\text{Predicates:}&&
\phi &\Coloneqq  
	\lambda  x\, \alpha \quad\text{provided \( \alpha \) is \( x \)-positive}
	\end{aligned}
\]
We assume the usual notion of free and bound variable occurrences.
A formula is \( V\! \)\emph{-positive} (or \emph{\( x \)-positive} if \( V = \setof x \)) if \( \lnot x \) does not occur in the formula for every variable \( x \in V\! \).
Predicates and formulas are identified up to \( \alpha \)-equivalence, understood in the usual sense.
We refer to \( \conj \) and \( \disj \) as \emph{connectives}, \( \square \) and \( \diamond \) as \emph{modal} operators, and \( \mu \) and \( \nu \) as \emph{quantifiers}.
Formulas \( \disj \Gamma \) (\( \conj \Gamma \)) are called \emph{disjunctions} (\emph{conjunctions}) and elements of \( \Gamma \) as \emph{disjuncts} (\emph{conjuncts}).
A formula of the kind \( \mu \phi \) or \( \nu \phi \) is called \emph{quantified}.
The set of formulas is denoted \( \form \). 
A set \( L \) of literals is \emph{consistent} if \( x \not\in L \) or \( \lnot x \not\in L \) for every \( x \in \Var \).

We adopt the following naming convention. Symbols \( \alpha, \beta , \dotsc \) range over formulas, \( \phi , \psi, \chi \) over predicates and \( \Gamma , \Delta , \dotsc \) over finite sets of formulas.
The symbol \( \bigcirc \) is used exclusively as a meta-variable for the two connectives \( \disj \) and \( \conj \), the symbol \( \triangle \) as a meta-variable for the modal operators, and \( \sigma \) for the quantifiers.
The following abbreviations for common formulas and constructions are utilised:
\[
\begin{aligned}
	\alpha \wedge \beta &\coloneqq \conj \setof{ \alpha , \beta }
	&
	\square \Gamma &\coloneqq \setof{ \square  \alpha }[\alpha \in \Gamma]
	& 
	\mf x \alpha &\coloneqq \mf (\lambda  x.\, \alpha)
	&
	\top &\coloneqq \conj \emptyset % \nf x\, \square x
	\\
	\alpha \vee \beta &\coloneqq \disj \setof{ \alpha , \beta }
	&
	\diamond \Gamma &\coloneqq \setof{ \diamond  \alpha }[\alpha \in \Gamma]
	& 
	\nf x \alpha &\coloneqq \nf (\lambda  x.\, \alpha)
	&
	\bot &\coloneqq \disj \emptyset %\mf x\, \diamond x
%	\\
%	\lnot \alpha &\coloneqq \alpha \to \bot
\end{aligned}
\]
For the \emph{complexity} of a formula \( \alpha \), in symbols \( \rk \alpha \), we employ the following definition: 
\[
  \begin{aligned}
  	\rk x = \rk {\lnot x}&= 
  	0
  	&
  	\rk { \triangle \alpha } &= \rk {\alpha} + 1
  	\quad
  	&
  	\Rk {\bigcirc \Gamma} &= \max \Setof{\rk {\alpha} + 1 }[ \alpha \in \Gamma]
%  	\rk { \phi \alpha } &= 1 + \max\setof{ \rk \phi , \rk \alpha }
%  	\rk {\alpha \to \beta} &= \max \setof{\rk \alpha + 1 , \rk \beta + 1}
  	\\
  	\rk {\lambda  x.\, \alpha } &= \rk {\alpha} 
  	\quad
  	&
  	\rk { \sigma  \phi } &= \rk {\phi} + 1
  \end{aligned}
\]
Negation is simulated via an involution on formulas and predicates:
\[
\begin{aligned}
	x^{\bot} &= \lnot x
	&
	\neg{ (\bigcirc \Gamma) } &= \neg\bigcirc \neg{\Gamma}
	&
	\neg{\Gamma} &= \setof{ \neg{\alpha} }[\alpha \in \Gamma ]
%	&
%	\neg{(\alpha \to \beta)} &= \conj\setof{\alpha,\neg\beta}
	\\
	\neg{ (\lnot x) } &= x
	&
	\neg{ (\triangle \alpha) } &= \neg{\triangle} \neg{\alpha}
	&
	\neg{ ( \sigma  \phi ) } &= \neg{\sigma } \neg{\phi}
	&
%	\neg{ (\lambda x.\, \alpha ) } &= \lambda \neg x\, \neg{ \alpha }
	\neg{ (\lambda x.\, \alpha ) } &= \lambda x\, \neg{ \alpha\subs[x]{\lnot x} }
\end{aligned}
\]
where \( \setof{ \bigcirc , \neg\bigcirc } = \setof{\conj,\disj} \) and similarly for \( \neg \triangle \) and \( \neg \sigma \), and \( \alpha\subs[x]{\lnot x} \) is the result of substituting \( \lnot x \) for all free occurrences of \( x \) in \( \alpha \). 
Observe that \( \neg{\alpha\subs[x]{\lnot x}} \) is \( x \)-positive iff \( \alpha \) is \( x \)-postive.

With the negation operation in place, the primitive operation of substitution given above can be lifted to a general form in a natural way.
We call a function \( \theta \colon V \to \form \) where \( V \subseteq \Var \) a \emph{substitution}.
Given a formula or predicate \( \eta \) and a substitution \( \theta \), define \( \eta [\theta ] \) to be the result of simultaneously substituting \( \theta x \) for \( x \) in \( \eta \) modulo renaming of bound variables in \( \eta \). The operation commutes with the connectives and quantifiers and is given by the following clauses for literals and predicates:
\[
\begin{gathered}
	x [{\theta}] =
	\begin{cases}
		\theta x, &x \in \mathsf{dom}\, \theta,
		\\
		x, &\text{otherwise,}
	\end{cases}
	\qquad
	\lnot x [\theta] = 
		\neg{(\theta x)},
	\\
	(\lambda y.\, \alpha ) [{\theta}] =
		\lambda y.\, \alpha [{\theta}],  \ 
		{\text{for $ y \not\in \mathsf{dom}\, \theta$ and \( y \) not free in \( \mathsf{Rng}\, \theta \).}}
\end{gathered}
\]
We employ notation \( \phi \beta \) as shorthand for \( \alpha\subs[x]\beta \) where \( \phi = \lambda x\, \alpha \).
\begin{lemma}
	For any predicate or formula \( \eta \), \( \rk \eta \le \rk{ \eta \subs{ \theta } } \le \rk {\eta} + \max\setof{ \rk {\theta(x)} }[x \in \mathsf{dom}\, \theta ]\). 
\end{lemma}

% --------------
\subsection{Subformulas and threads}
% --------------
There are two natural choices for a notion of \emph{subformula} in \( \mu \)-calculi, depending on the treatment of quantified formulas. 
We say that \( \beta \) is a \emph{literal subformula} of \( \alpha \) if \( \beta \) is among the formulas constructed in the generation of \( \alpha \) according to the grammar generating \( \form \).
In particular, \( \beta \) is a literal subformula of \( \sigma x.\, \beta \).
The second view takes its definition from the inference rules used to derive quantified (and other) formulas and specifies \( \phi(\sigma\phi) \) as a subformula of \( \sigma \phi \).
The complexity of a formula always bounds the complexity of its literal subformulas, but not of its derived subformulas.

In the context of proof systems, derived subformula plays a more important role than the literal one and, hence, is the primary notion in this paper and we drop the adjective `derived' in most cases.
The length of a sequence \( s \) is denoted \( \lh s \).
For a sequence of formulas \( \tau = ( \alpha_i)_i \) define \( \neg\tau = ( \neg{\alpha_i} )_i \).
\begin{definition}[Subformula]
	The \emph{formula graph} is the directed graph \( \tuple{\form, \sf } \) where \( \alpha \sf \beta \) iff  
	(i) \( \alpha = \bigcirc \Gamma \) and \( \beta \in \Gamma \),
	(ii) \( \alpha = \triangle \beta \), or
	(iii) \( \alpha = \sigma \phi \) and \( \beta = \phi\alpha \).
	If \( \alpha \sf \beta \) then \( \beta \) is called an \emph{immediate subformula} of \( \alpha \).
	The reflexive transitive closure of \( \sf \) is called the \emph{(derived) subformula relation} and is denoted \( \sft \). 
	A \emph{(derived) subformula} of \( \alpha \) is any \( \beta \) such that \( \alpha \sft \beta \).
	The set of subformulas of \( \alpha \) is denoted \( \subform{\alpha} \).
\end{definition}

Witnesses of the subformula relation, namely sequences \( (\alpha_i)_{i<k} \) such that \( \alpha_i \sf \alpha_{i+1} \) for each \( i \) are called \emph{threads}.
A sequence following the reflexive closure of \( \sf \) is called a \emph{weak thread}.
%The length of a sequence \( s \) will be denoted \( \lh s \).
Formulas may admit both infinite  and infinitely many threads, such as in the formula \( \alpha = \mf x( { \diamond x \vee \square  y }) \) for which there is one infinite thread and infinitely many distinct threads witnessing \( \alpha \sft y \).
A thread that starts at \( \alpha \) is called an \( \alpha \)-thread.
We write \( \tau : \alpha \sft \beta \) to express that \( \tau \) is a finite \( \alpha \)-thread ending in \( \beta \), that is, \( \lh{\tau} > 0 \), \( \tau(0) = \alpha \) and \( \tau({\lh {\tau}-1}) = \beta \). 
If \( \beta = \alpha \) and \( \lh \tau > 1 \) the thread is called \emph{cyclic}.
Writing \( \tau : \alpha \sft_\eta \beta  \) conveys, in addition, that \( \eta \) is a formula of lowest complexity on the thread \( \tau \), namely \( \tau(i) = \eta \) for some \( i < \lh{\tau} \) and \( \rk{\eta} \le \rk{\tau(i)} \) for every \( i \le \lh{\tau} \).
Note that if \( \tau \) a thread then so is \( \neg{\tau} = ( \neg{\tau(i)} )_{i<\lh {\tau} } \).

The following proposition relates threads with substitution.
\begin{restatable}{proposition}{propthdsubs}
	\label{thd-subst}
	Let \( \theta \) be a substitution.
	If \( \tau : \alpha \subs {\theta} \sft \beta \) then either:
	\begin{enumerate}
		\item \( \tau = (\tau'  \subs {\theta}) \tau'' \) for threads \( \tau' : \alpha \sft x \in \mathsf{dom}\, \theta \) and \( \tau'' : \theta(x) \sft \beta \).
		\item \( \tau = \tau' \subs {\theta} \) for some thread \( \tau' : \alpha \sft \beta' \) where \( \beta = \beta' \subs {\theta} \).
	\end{enumerate}
\end{restatable}

The next proposition sums up the central property of threads.
%The proof is deferred to appendix~\ref{app-syntax}.

%
\begin{restatable}{proposition}{propsubsumption}
	\label{thd-progress}
	Let \( \tau \) be an infinite thread. There is a unique formula \( \eta \) satisfying: (i) \( \eta = \tau(i) \) for infinitely many \( i \) and (ii) \( \rk{\eta} \le \rk{ \tau(j) } \) for all but finitely many \( j \).
	Moreover, \( \eta = \sigma \phi \) for some predicate \( \phi \) and quantifier \( \sigma  \).
\end{restatable}

\begin{definition}
	An infinite thread is \( \nu \) if the formula associated to the thread by Proposition~\ref{thd-progress} has the form \( \nf \phi \), and is said to be \( \mu \) otherwise.
\end{definition}
%\tbd[left and right threads?]

% --------------
\subsection{Three fragments of \texorpdfstring{$\boldsymbol{\form}$}{L-mu}}
\label{s-fragments}
% --------------
%
We identify three important fragments of \( \form \).
%
% --------------
%\subsubsection{The guarded fragment}
% --------------
The first fragment is often the focus of studies in modal \( \mu \)-calculi due to its well-structured threads.
This is the \emph{guarded} fragment comprising the formulas for which every cyclic thread among subformulas contains a modal formula.
%Guarded formulas are of especial import for proof-search and maintaining structural properties of illfounded proofs in general.
The guarded fragment is especially relevant in the context of proof-search.
\begin{definition}[Guarded fragment]
	A formula \( \alpha \) is \emph{guarded} if every cyclic thread in the subformula graph of \( \alpha \) passes through a modality: for every \( \beta \in \subform{\alpha} \) and \( \tau : \beta \sft \beta \), if \( \lh{\tau} > 1 \) then \( \tau(i) = \triangle \tau({i+1}) \) for some \( i < \lh{\tau} \) and \( \triangle \in \setof{ \square  , \diamond } \).
\end{definition}
%

% --------------
%\subsubsection{$\Pi$- and \( \Sigma \)-formulas}
% --------------
The second fragment of interest consists of all formulas in which \( \nu \) is the only quantifier.
We call this the \( \Pi \)-fragment.
In practise, however, greater stock will be put on the `closure' of certain sets of formulas under the generating operations of the \( \Pi \)-fragment.
Thus, the \( \Pi \)-fragment is defined as \( \PF \emptyset \) where, for any \( F \subseteq \form \), the set \( \PF F \) is given by
\begin{definition}[\( \PF F \)-fragment]
	Let \( F \subseteq \form \).
	The set \( \PF F \) is the smallest set of formulas satisfying:
	\begin{enumerate}
		\item \( F \cup \Lit \subseteq \PF F \),
		\item If \( \Gamma \subseteq \PF F \), then \( \disj \Gamma , \conj \Gamma \in \PF F \),
		\item If \( \alpha \in \PF F \), then \( \square   \alpha , \diamond \alpha  \in \PF F \),
		\item If \( \phi x \in \PF F \) then \( \nf \phi \in \PF F \),
		\item If \( \alpha \in \PF F \) is \( V \)-positive and \( \theta \colon V \to \PF F \), then \( \alpha \subs \theta \in \PF F \). 
	\end{enumerate}
	A \( \PF {\emptyset} \)-formula is called a \( \Pi \)-formula.
	A \( \Sigma \)-formula is a formula \( \neg \alpha \) for \( \alpha \in \PF \emptyset \).
\end{definition}

The closure operation \( F \mapsto \PF F \) induces a quantifier hierarchy for \( \form \) known as the \emph{alternation depth} or \emph{Niwiński hierarchy}, given by \( \Pi_1 = \PF{\emptyset} \) and \( \Pi_{n+1} = \PF{\Sigma_n} \) where \( \Sigma_n = \neg{\Pi_{n}} \).
The hierarchy exhausts \( \form \).
We will mostly be concerned with \( \PF F \)-formulas for two specific choices of \( F \) however:
the empty set and the set of formulas in the final fragment to be introduced.

%As \( \PF F \) is closed under substitution, there may be formulas \( \mf \phi \in \PF F \) which are not in \( F \) alone.

\begin{lemma}\label{sigma-F-nf}
	If \( \mf \phi \in \PF F \), then \( \phi = \psi \subs { \theta } \) for \( \mf \psi \in F \) and a substitution \( \theta \colon V \to \PF F \).
\end{lemma}
\begin{proof}
	By induction on the generation of \( \PF F \). If \( \mu  \phi \in F \) then we are done.
	Otherwise, \( \mu \phi = \alpha \subs \theta \) for some substitution \( \theta \colon V \to \PF F \) and \( V \)-positive \( \alpha \).
	We may assume that \( \alpha \not\in \Var \) as otherwise \( \mf \phi = \theta \alpha \) and the result follows from the induction hypothesis.
	Therefore, \( \alpha = \mf  \chi \) for some predicate \( \chi \), whence the induction hypothesis provides \( \mu \psi \in F \) and \( \theta' \colon V' \to \PF F \) such that \( \mu \chi = \mu\psi \subs {\theta' } \).
	Without loss of generality we may assume that each variable in \( V' \) occurs free in \( \psi \) and, therefore, \( \theta' x \) is \( V \)-positive (as \( \alpha \) is).
	Define \( \theta'' = \theta' \circ \theta \colon V \cup V' \to \PF F \) by \( \theta'' x = (\theta' x) \subs {\theta} \) for \( x \in V' \) and \( \theta'' x = \theta x \) otherwise.
	Then \( \mf \phi = \mf \psi \subs {\theta''} \).
\end{proof}

% --------------
%\subsubsection{The disjunctive fragment}
%\label{ss-disj}
% --------------
The final fragment is the \emph{disjunctive} fragment introduced by Janin and Walukiewicz~\cite{JanWal95} and employed heavily in Walukiewicz' completeness proof~\cite{Waluk00comp-lc}.
It is the class of formulas wherein use of conjunction is restricted to specific arrays of modal formulas and variables.
One of the many results concerning disjunctive formulas established in~\cite{JanWal95} is that they are every formula is equivalent to a formula in the disjunctive fragment. 
%Moreover, the disjunctive equivalent formula can be computed in a canonical way allowing these formulas to be employed as a `normal form' for \( \mu  \)-calculus.

For our rendition of the disjunctive normal form theorem it is necessary to weaken the definition of disjunctive formula slightly
so that the expressive adequacy of disjunctive formulas holds constructively rather than merely under classical logic. 
The modification also has the advantage of being more amenable to standard techniques from structural proof theory while maintaining the core properties that make the fragment indispensable.
Our modification takes the form of a liberalisation of the modal operator at the heart of the disjunctive fragment.

\begin{definition}[Expanded cover modality]
	Given a set \( \Gamma \cup \Delta \) of formulas, we introduce formulas \( \nabla(\Gamma , \Delta ) \) and, more generally, \( \nabla_L(\Gamma , \Delta ) \) for \( L \subseteq \Lit \) defined by
	\[
		\nabla_L( \Gamma , \Delta ) \coloneqq \conj\bigl( L \cup \diamond  \Gamma \cup \setof{ \square \disj \Delta } \bigr)
		\qquad
		\nabla(\Gamma , \Delta ) \coloneqq 
		\nabla_\emptyset( \Gamma , \Delta ) = \conj\bigl( \diamond  \Gamma \cup \setof{ \square \disj \Delta } \bigr)
	\]
	The \emph{expanded cover modality} comprises the formulas \( \nabla_L(\Gamma , \Delta) \) where \( \Gamma \subseteq \Delta \) and \( L \) is consistent.
\end{definition}
The `standard' cover modality is the special case \( \nabla \Gamma \coloneqq \nabla ( \Gamma , \Gamma ) \).
The two model operators \( \square \) and \( \diamond \) have simple representations via the expanded cover modality: 
\( \square \alpha \) is equivalent to \( \nabla ( \emptyset , \setof \alpha ) \) and \( \diamond \alpha \) to \( \nabla( \setof{\alpha} , \setof{ \alpha , \top } ) \).
Moreover, these equivalences hold over constructive modal logic (cf.~Proposition~\ref{exp-vs-unary-cover} below).
Indeed, the representation of \( \square \alpha \) already demonstrates a difference with the unary cover modality base on which \( \square \alpha \) is expressed as \( \nabla \emptyset \vee \nabla \setof{ \alpha } \), involving an external disjunction.
More generally, the expanded cover modality \( \nabla( \Gamma , \Delta ) \) is equivalent to the disjunction \( \nabla \Gamma \vee \nabla \Gamma' \) where \( \Gamma' = \Gamma \cup \Setof{ \disj ( \Delta \setminus \Gamma) } \), but this equivalence does not hold constructively in general. 

The restriction of consistency on \( L \) does not reduce the expressivity of the cover modality as any formula \( \nabla_L (\Gamma , \Delta) \) where \( L \) is \emph{not} consistent is clearly unsatisfiable.
Nor is it simply a matter of convenience: Theorem~\ref{kmoi-comp-conj} depends specifically on the restriction to show that for a certain class of formulas employing the dual of the cover modality, provability relative to a constructive calculus coincides with validity against the classical semantics. In particular, it is necessary that no instance of the excluded middle can be expressed as the dual of a cover modality.

\begin{definition}[Disjunctive fragment]
	\label{d-conjunctive}%
	A formula is \emph{disjunctive} iff every subformula that is a conjunction is an instance of an expanded cover modality.
	The \emph{disjunctive fragment} is the set \( \D \) of disjunctive formulas.
	A formula \( \alpha \) \emph{conjunctive} iff \( \neg \alpha \) is disjunctive.
\end{definition}
Although \( \bot \) is disjunctive by definition, \( \top \) is not as every conjunction contains at least one formula. In the context of disjunctive formulas, therefore, we take in its place the disjunctive rendition of \( \nu x\, \square x \).
It is with remarking that the disjunctive fragment is more constrained than merely utilising the operator \( \nabla \) in place of conjunction in the inductive definition of formulas.
Specifically, variables occurring alongside the modal operators (the `\( V \)' in \( \nabla_V( \Gamma , \Delta) \)) can not be bound in a disjunctive formula.
The formula \( \nabla_{\setof x} ( \emptyset , \emptyset ) \), for instance, is disjunctive yet \( \alpha = \nf x .\, \nabla_{\setof x}( \emptyset , \emptyset ) \) is not because \( \conj \setof{ \alpha , \square \top } \in \subform \alpha \) can not be realised as an instance of the expanded cover modality.
However, \( \alpha \in \PD \), a fact that will be useful later.

Combining the restriction on conjunctions in disjunctive formulas and on quantifiers in the \( \Pi \) fragments is:
\begin{proposition}
	\label{SC-is-WC}
	Let \( \tau : \conj \Gamma \sft_{\mf \phi} \conj \Gamma \) be a cyclic thread. If \( \conj \Gamma \in \PD \) then \( \conj \Gamma \) is an instance of the expanded cover modality.
\end{proposition}
\begin{proof}
	Suppose  \( \Gamma \subseteq \PD \) and \( \tau : \conj \Gamma \sft_{\mf \phi} \conj \Gamma \).
	In particular, \( \mu \phi \in \PD \), whereby Proposition~\ref{sigma-F-nf} implies that \( \mu  \phi = \mu  \psi \subs {\theta} \) for \( \mu \psi \in \D \) and a substitution \( \theta \colon V \to \PD \).
	Without loss of generality we may assume that all variables in \( V \) occur free in \( \psi \).
	In particular, \( \rk{\theta(x)} < \rk{\mu  \phi } \) for each \( x \in V \).
	As \( \lh{\tau(i)} \ge \lh {\mu  \phi} \) for all \( i < \lh{\tau} \) it follows that \( \Gamma = \Gamma' \subs {\theta } \) for some \( \conj \Gamma' \in \subform{\mf \psi} \) (Proposition~\ref{thd-subst}(2)), whence \( \conj \Gamma = \nabla_V( \Delta , \Delta') \) for some \( \Delta , \Delta' \) and \( V \).
\end{proof}

% --------------
\subsection{Fixed-point semantics}
\label{s-semantics}
% --------------
%\tbd[classical!]

Formulas of the modal \( \mu \)-calculus are evaluated against labelled trees, where
an \emph{\( L \)-labelled tree} (henceforth simply \emph{tree}) is a pair \( \tuple{T, \lambda  } \) where \( T \subseteq \omega^{<\omega} \) is a non-empty and prefix-closed set of finite sequences of natural numbers, called \emph{vertices}, and \( \lambda  \colon T \rightarrow L \) is function assigning each \( v \in T \) a \emph{label} \( \lambda (v) \).
The empty sequence is called the \emph{root}.
A proper prefix of a vertex \( u \in T \) is called the \emph{predecessor} of \( u \); conversely, a \emph{successor} of \( u \) is a vertex that has \( u \) as a predecessor.
Denoting the prefix relation on sequences by \( \le \) and the proper prefix relation by \( < \), we have that \( u \) is a successor of \( v \) iff \( v < u \).
Proper successors of minimal length are called \emph{children}, namely the set of children of \( u \in T \) is the set \( \Suc [T] u \coloneqq T \cap \setof{ ui }[ i \in \omega ] \).

The label function associated to a tree \( T \) is written as \( \lambda _T \). We abuse notation and identify labelled trees with the set of their vertices, writing \( T = \tuple{T,\lambda _T} \).
Given \( \tuple{T,\lambda } \) and \( s \in T \) define \( \tuple{T \res s,\lambda \res s} \) as the subtree of \( T \) at \( s \), given by \( T\res s = \setof{ v \in T }[sv \in T] \) and \( \lambda \res s \colon v \mapsto \lambda (sv) \).
The set of immediate subtrees of \( T \) is denoted \( \tuple{T,\lambda }_* = \Setof{ \tuple{T\res s,\lambda \res s} }[ s \in T \text{ and } \lh{s} = 1 ] \).

A \emph{path} in \( T \) is a finite or infinite sequence \( P = v_0v_1\dotsm \) of vertices such that \( v_{i+1} \in \Suc [T]{v_i} \) for all \( i < \lh P \). 
An infinite path is also called a \emph{branch}.
Given a branch \( B \) we often write \( v \in B \) if \( v \le B(i) \) for some \( i \).
\begin{definition}
	Set \( E = 2^\Var \).
	A \emph{(tree) model} is a \( E \)-labelled tree \( \tuple{T,A} \).
%	Such a label function is called an \emph{assignment}.
	The set of models is denoted \( \Mod \). 
	Given a model \( M = \tuple{T,A} \) and a set of models \( U \subseteq \Mod \), we write \( A \subs[x]{U} \) for the function \( A' \colon T \to E \) such that \( A'( v ) = A(v) \cup \setof{x} \) if \( \tuple{T,A} \res v \in U \) and \( A'(v) = A(v) \setminus \setof {x} \) otherwise.
\end{definition}

% --------------
\subsubsection{Kripke semantics}
% --------------

In Kripke semantics for the modal \( \mu  \)-calculus each formula \( \alpha \) is assigned a set \( \denote{\alpha} \subseteq \Mod \) of models, called the \emph{denotation} of \( \alpha \), and each predicate \( \phi \) associated a function \( \denote{\phi} \colon 2^{\Mod} \rightarrow 2^{\Mod} \), defined by mutual recursion:
\begin{align*}
	\denote x &= \setof{ \tuple{T,A} \in \Mod }[ x \in A(\epsilon)]
%	&
%	\denote {\neg x} &= \Mod \setminus \denote x
	\\
	\denote { \square  \alpha } &= \setof{ M \in \Mod }[ M_* \subseteq \denote{\alpha} ]%\square  \denote {\alpha}
	&
	\denote { \diamond a } &= \setof{ M \in \Mod }[ M_* \cap \denote{\alpha} \neq \emptyset ]%\diamond \denote {\alpha}
	\\
	\denote {\conj \Gamma} &= \bigcap \denote {\Gamma }
	&
	\denote {\mf \phi } &= \bigcap \Setof{ U \subseteq \Mod }[ \denote { \phi } U \subseteq U]
	\\
	\denote {\disj \Gamma} &= \bigcup \denote {\Gamma }
	&
	\denote {\nf \phi } &= \bigcup \Setof{ U \subseteq \Mod }[ U \subseteq \denote { \phi } U ]
\end{align*}
where
\begin{align*}
%	\intertext{where}
	\denote {\Gamma} &= \Setof{ \denote{\gamma} }[\gamma \in \Gamma]
	%&
	&
	\denote { \phi } &\colon {U \mapsto  \Setof{ \tuple{ T , A } }[ \tuple{ T , A \subs[x]U } \in \denote{\phi x } ]}.
	\qquad
\end{align*}
If \( \denote{\alpha} = \Mod \) we call \( \alpha \) \emph{valid}.

The proofs of the following are standard; the reader can consult, for instance, \cite{Arnold:2001Rudiments}.

\begin{proposition}
	Suppose \( \tuple{T,A} \) and \( \tuple{T,A'} \) are models and \( V \subseteq \Var \) is such that for all \( v \in T \), \( A(v) \cap V = A'(v) \cap V \).
	Then for every formula \( \alpha \) whose free variables are among \( V \),  \( \tuple{T,A} \in \denote{\alpha} \) iff \( \tuple{T,A'} \in \denote{\alpha} \).
\end{proposition}

\begin{proposition}
	The denotation of predicates are monotone functions: If \( \phi \) is a predicate and \( U \subseteq V \subseteq \Mod \), then \( \denote \phi U \subseteq \denote \phi V \).
\end{proposition}

\begin{lemma}
	\( \denote{\neg{\alpha}} = \Mod \setminus \denote{\alpha} \) for every formula \( \alpha \).
\end{lemma}

\begin{proposition}
	\label{denot-sem-fxpt}
	\( \denote{\phi}\denote{\sigma  \phi} = \denote{\sigma \phi} \). % and \( \denote{\nu \phi} = \denote{\phi}(\denote{\nu \phi}) \).
\end{proposition}
%

% --------------
\subsubsection{Game semantics}
% --------------

Rather than inductively ascribing a set of models to each formula we can instead view the connectives, modalities and quantifiers as operations that `reduce' the claim `\( \alpha \) is true in \( M \)' to analogous claims about immediate subformulas and submodels.
In the case of the purely modal vocabulary, it is straightforward to read the Kripke semantics in this style.
However, from this local perspective the statement `\( \mu  \phi \) is true at \( M \)' is re-expressed as the just as complex claim `\( \phi(\mu  \phi) \) is true at \( M \)'; the global constraint ensures that this reduction is permitted only finitely often.

This semantics has a natural formulation as a two-player game, one player's aim being to `verify' the claim `\( \alpha \) is true in \( M \)' against moves by their opponent.
For our purposes, we need only the `certificate' of this verification process, i.e., a winning strategy, which we call a \emph{verification}:
\begin{definition}[Verification]
A \emph{pre-verification} is an \( (\form \times \Mod ) \)-labelled tree \( V = \tuple{ T , \lambda  } \) satisfying, for every \( v \in T \):
\begin{enumerate}
	\item If \( \lambda (v) = ( x , M ) \in \Var \times \Mod \) then \( M \in \denote x \).
	\item If \( \lambda (v) = ( \disj \Gamma , M ) \) there exists \( \gamma \in \Gamma \) and \( v' \in \Suc[V] v \) with \( \lambda (v') = ( \gamma , M ) \).
	\item If \( \lambda (v) = ( \diamond\alpha , M ) \) there exists \( M ' \in M_* \) and \( v' \in \Suc[V] v \) with \( \lambda (v') = ( \alpha , M' ) \).
	\item If \( \lambda (v) = ( \conj \Gamma , M ) \) then for all \( \gamma \in \Gamma \) there exists \( v' \in \Suc[V] v \) with \( \lambda (v') = ( \gamma , M ) \).
	\item If \( \lambda (v) = ( \square \alpha , M ) \) then for all \( M' \in M_* \) there exists \( v' \in \Suc[V] v \) with \( \lambda (v') = ( \alpha , M' ) \).
	\item If \( \lambda (v) = ( \sigma  \phi , M ) \) there exists \( v' \in \Suc[V] v \) with \( \lambda (v') = ( \phi(\sigma \phi) , M ) \).
\end{enumerate}
A \emph{verification} is a pre-verification \( V \) such that every branch \( ( \alpha_i , M_i )_i \) of \( V \), the thread \( ( \alpha_i)_i \) is $\nu$.
By \( \oper{\alpha} \) is denoted the set of models \( M \) for which there exists a verification \( V \) with root \( ( \alpha , M ) \).
\end{definition}

Unlike Kripke semantics, each (pre-)verification only references subtrees of the tree modal labelling the root.
But following that semantics, predicates can be interpreted as monotone functions over the set of models:
%\( \oper{\phi} \colon 2^\Mod \rightarrow 2^\Mod \) as 
\[ \oper{\phi} \colon U \mapsto \Setof{ \tuple{ T , A } }[ \tuple{ T , A\subs [x] U } \in \oper{\phi x } ] . \]

It can be shown that the game semantics ascribes the appropriate fixed points to the two quantifiers:

\begin{proposition}
	\label{FST-lemma}
	If \( \oper{\phi} U \subseteq U \) then \( \oper{\mu \phi} \subseteq U \) and if \( U \subseteq \oper{\phi} U \) then \( U \subseteq \oper{\nu \phi} \).
\end{proposition}

The equivalence of game and Kripke semantics, often referred to as the `fundamental semantic theorem', was first established by Streett and Emerson~\cite{StrEmer89}.
{Contrary to some sources, however, the theorem admits an easy proof on the basis of a few elementary lemmas relating the two semantics.}
Indeed, we have already all the necessary ingredients.

\begin{theorem}[Fundamental Semantic Theorem]
	\label{FST}
	For all \( \alpha \), \( \denote{\alpha} = \oper{\alpha} \).
\end{theorem}

\begin{proof}
	By induction on \( \rk{\alpha} \) using Propositions~\ref{denot-sem-fxpt} and \ref{FST-lemma}.
\end{proof}

\begin{corollary}%[Determinacy for game semantics]
	\( \oper{\neg{\alpha}} = \Mod \setminus \oper{{\alpha}} \).
\end{corollary}

% --------------
\section{Finitary and Illfounded Proofs}
\label{s-proof-systems}
% --------------
The basic object of deduction in the finitary and illfounded proof system is the \emph{sequent}, expressions of the form \( \Phi \Rightarrow \Psi \) where \( \Phi \) and \( \Psi \) are finite sequences of formulas.
In section~\ref{s-cyclic-proofs} we introduce a third proof system which operates on a slightly broader notion of sequent in which `formula' are constructed via a richer grammar and the sequent itself carries an additional, non-formula, component that interacts with some inferences of the calculus.
A uniform definition of `proof system' should, therefore, be broad enough to accommodate both choices of deduction elements.
In the interests of space, however, we formulate the concept in terms of `ordinary' sequents hoping that the general intension of the definition is clear.

The set of sequents is denotes \( \sequents \).
We employ symbols \( \Phi \), \( \Psi \), etc.\ as metavariables for finite sequences of formulas.
In a sequent \( \Phi \Rightarrow \Psi \), the terms \emph{antecedent} and \emph{consequent} refer to \( \Phi \) and \( \Psi \) respectively, and the sequent is \emph{intuitionistic} if \( \lh{\Psi} \le 1 \). 
In an intuitionistic sequent we refer to the formula in \( \Psi \) (if it is non-empty) as the \emph{head} and to \( \Phi \) as the \emph{context}.
In the context of illfounded proofs it is important to have a well-defined notion of formula \emph{occurrence}.
Once established, however, the ordering of formulas within sequents will be largely ignored and left implicit.

A \emph{position} is a pair \( p = (s,k) \in \setof{\rul L,\rul R} \times \omega \). If \( s = \rul L \) we refer to \( p \) as a \emph{left} position; otherwise, \( p \) is a \emph{right} position.
A \emph{formula occurrence} in a sequent \( \sequent = \seq\Phi\Psi \) is a position \( p = (s,k) \) such that \( k < \lh \Phi \) if \( p \) is a left position, and \( k < \lh{\Psi} \) otherwise.
The formula at position \( p = ( s , k ) \) in \( \sequent \) is denoted \( \sequent(p) \), where
\[ 
	(\seq\Phi\Psi)(p) =
	\begin{cases}
		\Phi(k), &\text{if } s = \rul L,
		\\
		\Psi(k), &\text{if } s = \rul R.
	\end{cases}
\]

An \emph{inference} is a pair \( ( \sequent , S )  \) where \( \sequent \) is a sequent and \( S \) is a finite set of sequents, usually presented in the form:
\[
\begin{prooftree}
	\hypo{ S }
	\infer1{ \sequent }
\end{prooftree}
\qquad\text{or}\qquad
\begin{prooftree}
	\hypo{ \sequent_0 }
	\hypod
	\hypo{ \sequent_{n} }
	\infer3{ \sequent }
\end{prooftree}
\quad\text{if \( S = \setof{ \sequent_0, \dotsc, \sequent_n } \).}
\]
The sequent \( \sequent \) is referred to as the \emph{conclusion} of the inference and elements of \( S \) as \emph{premises}.
%Inferences with no premises are \emph{axioms}.
%A sequent is \emph{axiomatic} if it is the conclusion of an axiom.
%
A set of inferences constitutes a \emph{rule}. 
Typically, rules are specified by an inference scheme such as the rule (\( \disjR \)) comprising the inference
\[
  \begin{prooftree}
	\hypo{ \seq \Phi {\Psi , \gamma} }
	\infer1{ \seq \Phi {\Psi, \disj \Gamma }}
\end{prooftree}
\]
for all \( \Phi \), \( \Psi \), \( \Gamma \) and \( \gamma \in \Gamma \).
An inference is an \emph{instance} of a rule iff it is an element of that rule.
For our purposes each rule is further identified as being either a \emph{logical} or \emph{structural} rule, the only difference  being how they may be utilised within derivations and proofs.

\begin{definition}[Derivation]
	A \emph{derivation} over a set of rules \( \rules \) is a tree \( d \) with labelings \( \sequent_d \colon d \to \sequents \) and \( \rules_d \colon d \to \rules \) satisfying:
	\begin{enumerate}
		\item for each vertex \( v \in d \),\label{it-derivation-local} %the sequent \( S_v \) and rule \( (\rul r_v) \) labelling \( v \) are such that \[ \bigl( S_v , \Setof{ S_u }[u \in \Suc v] \bigr)  \in (\rul r_v) . \]
		\( ( \sequent _d(v) , \setof{ \sequent_d(u) }[u \in \Suc v] )  \in \rules_d(u)  \).
		\item for every branch \( B \subseteq d \) there are infinitely many \( i < \omega \) such that \( \rules_d(B(i)) \) is a logical rule.\label{it-derivation-global}
	\end{enumerate}
	The \emph{endsequent} of a derivation is the sequent labelling its root.
	A derivation is \emph{constructive} if every sequent occurring in the derivation is intuitionistic.
\end{definition}
The second condition in the definition is primarily of significance for the illfounded proof system introduced in section~\ref{s-ill-founded-proofs} where certain derivations with infinite paths qualify as proofs.
As a consequence of the first condition, leaves of a derivation are labelled by rule instances with no premises.

Not every derivation constitutes a \emph{proof}.
In finitary proof systems, only finite derivations will be permitted as proofs, while in the system of illfounded proofs infinite derivations are recognised as proofs \emph{if} their branches fulfil a syntactic correctness condition.

% --------------
\subsection{Basic fixed-point logic}
\label{s-kmm}
% --------------

The first system of derivations and proofs we introduce forms the core of the other calculi in this paper.
This is the system \( \kmm \) of \emph{basic fixed-point logic}.
The structural and logical rules of \( \kmm \) are listed in fig.~\ref{f-kmm-struc} and \ref{f-kmm-logic} respectively.
Rules denoted (\starL) are called \emph{left} rules and (\starR) are \emph{right} rules.
Other rules, namely (\hyp) and (\idLR), are both left and right rules.
Leaves of a derivation labelled by the structural rule (\( \hyp \)) are called \emph{buds} and serve to allow conditional proofs, i.e., proofs from \emph{assumptions}.
Thus a finite derivation in the rules of \( \kmm \) will be called a (\( \kmm \)-)\emph{proof from assumptions} with a (\( \kmm \)-)\emph{proof} being a finite derivation \emph{without} buds.

The structural rules (\excL) and (\excR) are jointly called the \emph{exchange} rules; (\weakL) and (\weakR) are \emph{weakening} rules, and (\conL) and (\conR) \emph{contraction} rules.
The rules (\idL), (\idR) and (\idLR) are collectively the \emph{identity} rules.
By (\exc) (respectively, (\weak) and (\con)) we denote the rule expressing closure under applications of the exchange (resp., weakening and exchange, contraction and exchange) rules, namely (\exc) denotes the collection of instances \( (\sequent , \setof{\sequent'} ) \) for which there is a (possibly trivial) proof from assumptions using (\excL) and (\excR) only whose conclusion is \( \sequent \) and (only) bud is \( \sequent' \).
When there is no cause for confusion we leave instances of the structural rules implicit.

%Vertices of a derivation that are premises to an instance of \( (\square ) \) are called \emph{modal premises}.

In each of the rules of \( \kmm \) with the exception of (\hyp) and (\idLR), a single formula (occurrence) in the conclusion is designated the \emph{principal} formula (occurrence). 
This is the formula occurrence \( (\rul L , 0) \) for a left rule and \( (\rul R , \lh{\Psi} ) \) for a right rule with \( \Psi \) as in the figures.
In the case of (\hyp) and (\id), all formula occurrences in the conclusion are designated as principal.
The `matching' formula occurrence in the premise(s) of the rule are \emph{minor} formula( occurrence)s. In the logical rules, the minor formula is the formula occurrence at the position of the principal formula; for the structural rules it is the positions inhabited by \( \alpha \) in fig.~\ref{f-kmm-struc}.
Other formula occurrences are \emph{side} formulas. 
With the exception of the modal inferences (\sqR) and (\diL), the sequent of side formulas occurs in both conclusion and each premise in the same order.

\begin{figure*}
	\centering
	\ebproofset{small}
	\begin{tabular}{l@{\qquad}c@{\qquad}c@{\qquad}c}
		\begin{prooftree}
			\hypo{ \Phi , \alpha , \Lambda \Rightarrow \Psi }
			\infer1[\excL]{ \alpha , \Phi , \Lambda \Rightarrow \Psi }
		\end{prooftree}
		&
		\begin{prooftree}
			\hypo{ \alpha , \alpha , \Phi \Rightarrow \Psi }
			\infer1[\conL]{ \alpha , \Phi \Rightarrow \Psi }
		\end{prooftree}
		&
		\begin{prooftree}
			\hypo{ \Phi \Rightarrow \Psi }
			\infer1[\weakL]{ \alpha , \Phi \Rightarrow \Psi }
		\end{prooftree}
		&
		\begin{prooftree}
			\hypo{}
			\infer1[\hyp]{ \Phi \Rightarrow \Psi }
		\end{prooftree}
		\\[1em]
		\begin{prooftree}
			\hypo{ \Phi \Rightarrow \Psi , \alpha , \Lambda }
			\infer1[\excR]{ \Phi \Rightarrow \Psi , \Lambda , \alpha }
		\end{prooftree}
		&
		\begin{prooftree}
			\hypo{\Phi \Rightarrow \Psi , \alpha , \alpha  }
			\infer1[\conR]{ \Phi \Rightarrow \Psi , \alpha }
		\end{prooftree}
		&
		\begin{prooftree}
			\hypo{ \Phi \Rightarrow \Psi }
			\infer1[\weakR]{ \Phi \Rightarrow \Psi , \alpha }
		\end{prooftree}
	\end{tabular}
	\caption{Structural rules of \( \kmm \). Rules denoted \( \starL \) (\( \starR \)) are \emph{left} (resp.\ \emph{right}) rules.}
	\label{f-kmm-struc}
\end{figure*}
% END FIGURE
%
\begin{figure*}
	\centering
	\ebproofset{small}
	\begin{tabular}{c@{\quad}c@{\quad}c@{\quad}r}
		\begin{prooftree}
			\hypo{ \seq {\gamma , \Phi} \Psi }
			\infer[left label={$\gamma \in \Gamma$}]1[\conjL]{ \seq {\conj \Gamma , \Phi} \Psi }
		\end{prooftree}
		&
		\begin{prooftree}
			\hypo{ \Setof{ \seq {\gamma , \Phi} \Psi }[\gamma \in \Gamma] }
			\infer1[\disjL]{ \seq {\disj \Gamma , \Phi} \Psi }
		\end{prooftree}
		&
		\begin{prooftree}
			\hypo{ \seq {\phi ( \sigma  \phi ) , \Phi} \Psi }
			\infer1[\sigL]{ \seq {\sigma \phi , \Phi} \Psi }
		\end{prooftree}
		&
		\begin{prooftree}
			\hypo{ \seq {\alpha , \Phi} \Psi }
			\infer1[\diL]{ \seq {\diamond \alpha , \square \Phi} {\diamond \Psi} }
		\end{prooftree}
		\\[1.5em]
		\begin{prooftree}
			\hypo{ \setof{ \seq \Phi {\Psi , \gamma} }[ \gamma \in \Gamma ] }
			\infer1[\conjR]{ \seq \Phi {\Psi , \conj \Gamma} }
		\end{prooftree}
		&
		\begin{prooftree}
			\hypo{ \seq \Phi {\Psi , \gamma} }
			\infer[left label={$\gamma \in \Gamma$}]1[\disjR]{ \seq \Phi {\Psi , \disj \Gamma } }
		\end{prooftree}
		&
		\begin{prooftree}
			\hypo{ \seq \Phi {\Psi , \phi ( \sigma  \phi ) }}
			\infer1[\sigR]{ \seq \Phi {\Psi , \sigma \phi }}
		\end{prooftree}
		&
		\begin{prooftree}
			\hypo{ \seq \Phi {\Psi , \alpha} }
			\infer1[\sqR]{ \seq {\square \Phi} {\diamond \Psi , \square  \alpha} }
		\end{prooftree}
		\\[1.5em]
		\begin{prooftree}
			\axiom[\idL]{ \lseq {\alpha, \neg \alpha} }
		\end{prooftree}
		&
		\begin{prooftree}
			\axiom[\idR]{ \rseq  {\neg \alpha , \alpha } }
		\end{prooftree}
		&
		\begin{prooftree}
			\axiom[\idLR]{\seq {\alpha } \alpha }
		\end{prooftree}
		&
		%
%		\multicolumn{3}{c}{
%		\begin{prooftree}
%			\hypo{}
%			\infer[left label={$\setof{\alpha,\neg\alpha} \subseteq \Phi \cup \neg \Psi $}]1[$\alpha$]{\seq {\Phi} \Psi }
%		\end{prooftree}
%		}
%		%
%		\begin{prooftree}
%			\axiom[$\alpha$\rul L]{\lseq {\alpha,\neg\alpha} }
%		\end{prooftree}
%		%
%		&
%		%
%		\begin{prooftree}
%			\axiom[$\alpha$\rul M]{\seq {\alpha} \alpha }
%		\end{prooftree}
%		%
%		&
%		%
%		\begin{prooftree}
%			\axiom[$\alpha$\rul R]{\rseq {\alpha,\neg\alpha} }
%		\end{prooftree}
		%
		%
%		&
%		%
%		\begin{prooftree}
%			%
%			\hypo{\Phi \cup \neg \Psi = \setof{\alpha, \neg\alpha}}
%			\infer1[$\alpha$]{ \seq \Phi \Psi }
%			%
%		\end{prooftree}
		%
		%
	\end{tabular}
	\caption{Logical rules of system \( \kmm \). Rules denoted \( \starL \) (\( \starR \)) are \emph{left} (resp.\ \emph{right}) rules.}
	\label{f-kmm-logic}
\end{figure*}
% END FIGURE

The calculus \( \kmm \) is of little significance beyond it being the common core of the other proof systems we introduce.
We collect some observations on the rules of \( \kmm \) for later use.
\begin{proposition}
	\label{kmm-kmmi}
	Any \( \kmm \)-derivation whose endsequent has empty consequent is constructive.
\end{proposition}

\begin{proposition}
	\label{kmm-neg}
	There is a \( \kmm \)-proof from \( \Phi \Rightarrow \Psi \) assumptions \( \setof{\seq{\Phi_i}{\Psi_i}}[i \in A] \)
	 iff there is a \( \kmm \)-proof of \( \neg\Psi , \Phi \Rightarrow \) from assumptions \( \setof{\lseq{\neg{\Psi_i},\Phi_i}}[i \in A] \).
\end{proposition}
\begin{proof}
	Replace each sequent in the proof in the way described. In the `left to right' direction, 
	this causes each instance of a right rule (\starR) to become an instance of the corresponding left rule (\starL) up to  exchange, and each left rule (\starL) to remain as a left rule (up to exchange).
%	Provided the end-sequent has the form \( \lseq \Xi \), the \( \kmm \) derivation will be constructive.
\end{proof}

%\begin{nproposition}
%	Every \( \kmm \)-derivation with endsequent \( \Phi \Rightarrow \) where \( \Phi \subseteq \formc \) is an intuitionistic derivation.
%\end{nproposition}
An important derived rule in the various calculi to follow expresses that predicates induce monotone functions:
\begin{prooftree*}
	\hypo{ {\alpha } \Rightarrow {\beta} }
	\infer1[$\phi$]{ { \phi \alpha } \Rightarrow { \phi \beta }}
\end{prooftree*}
%
%The monotonicity rules ($\phi$) are derivable in \( \kmm \).
%
\begin{proposition}
	\label{kmm-mono}
	The monotonicity rules are derivable in \( \kmm \):
	For each \( \phi \), \( \alpha \) and \( \beta \), 
%	the monotonicity rule
%	\begin{prooftree*}
%		\hypo{ {\alpha } \Rightarrow {\beta} }
%		\infer1[$\phi$]{ { \phi \alpha } \Rightarrow { \phi \beta }}
%	\end{prooftree*}
%	is constructively derivable: 
	there is a \( \kmm \)-proof of \( \seq{\phi\alpha}{\phi\beta} \) from  assumptions \( \seq \alpha \beta \).
\end{proposition}
%
%\begin{proof}
%	Induction on the complexity of \( {\phi} \).
%\end{proof}
%

% --------------
\subsection{Finitary proofs}
\label{s-inductive-proofs}
% --------------
The first extension of \( \kmm \) is a rendition of Kozen's axiomatisation of the modal \( \mu  \)-calculus in the sequent calculus. This calculus, called \( \kmu \), extends basic fixed point logic by \emph{initial} sequents, three forms of sequent expressing \( \seq \alpha \alpha \), the structural rule of cut and rules of induction. The new rules are presented in fig.~\ref{f-kmu}.
%The induction rules (\indL) and (\indR) are
Proofs (with assumptions) in the sense of \( \kmm \) are, again, the finite derivations (with buds).
We refer to proofs in \( \kmu \) as \emph{finitary} proofs.
The constructive fragment of \( \kmu \), i.e., the restriction to rules employing only intuitionistic sequents, is denoted \( \kmui \).

% FIGURE: K\mu 
\begin{figure*}
	\centering
	\ebproofset{small}
	\begin{prooftree}
		\hypo{ \seq \Phi {\Psi, \phi \bigr( \conj ( \Phi \cup \neg \Psi ) \bigl)} }
		\infer1[\indR]{ \seq \Phi {\Psi, \nf \phi} } 
	\end{prooftree}
	\qquad
	\begin{prooftree}
		\hypo{ \seq {\phi \bigl( \disj(\neg \Phi \cup \Psi ) \bigr) , \Phi } \Psi }
		\infer1[\indL]{ \seq {\mf \phi, \Phi}  \Psi } 
	\end{prooftree}
	\qquad
	\begin{prooftree}
		\hypo{ \seq{\Phi} {\Psi , \alpha} }
		\hypo{ \seq{\alpha , \Lambda} \Xi }
		\infer2[\cut]{ \Phi , \Lambda \Rightarrow \Psi , \Xi }
	\end{prooftree}
%	\qquad
%	\begin{prooftree}
%		%
%		\hypo{ \Phi , \alpha , \alpha }
%		\infer1[\con]{ \Phi , \alpha }
%		%
%	\end{prooftree}
	\caption{Induction and cut rules of \( \kmu \). The induction rules (\indL) and (\indR) are denoted jointly by (\ind).}
	\label{f-kmu}
\end{figure*}
% END FIGURE

The transformation of \( \kmm \)-derivations from 
Proposition~\ref{kmm-neg} carries over to \( \kmu \) though Proposition~\ref{kmm-kmmi} does not due to the presence of cut.
\begin{proposition}
	\label{kmu-neg}
	\( \kmu \vdash \Phi \Rightarrow {\Psi} \) iff \( \kmu \vdash \neg\Psi , \Phi \Rightarrow {} \).
\end{proposition}
\begin{proof}
	By Proposition~\ref{kmm-neg}, it suffices to treat the  rules in fig.~\ref{f-kmu}.
	The induction rules commute with the translation by design. 
	The transformed cut rule is replaced by (\disjL) with principal formula \( \disj \setof{ \alpha , \neg\alpha} \) and cut against a \( \kmu \)-proof of the same formula.
\end{proof}

Later we will have need for a more general formulation of the induction rules, introduced in \cite{AL17lics}:
\[
\begin{prooftree}
	\hypo{ \seq{\Phi} {\Psi , \nf x\, \phi \bigl( \conj ( \Phi \cup \neg\Psi ) \vee \var x \bigr) } } 
	\infer1[\sindR]{ \seq{\Phi} {\Psi , \nf \phi} }
\end{prooftree}
\qquad
\begin{prooftree}
	\hypo{ \seq{\mf x\, \phi \bigl( \disj ( \neg\Phi \cup \Psi  ) \wedge \var x \bigr) , \Phi } {\Psi} } 
	\infer1[\sindL]{ \seq{\mf \phi , \Phi} {\Psi} }
\end{prooftree}
\]
Instances of the two rules immediately above are referred to as \emph{strong induction}. By (\sind) we denote the union of (\sindL) and (\sindR ). The reader is referred to~\cite[Thm VI.3]{AL17lics} for a proof of the next proposition.
\begin{proposition}
	\label{kmu-sind}
	The strong induction rules are derivable in \( \kmu \). Instances of strong induction with intuitionistic conclusion are provable in \( \kmui \).
\end{proposition}

Soundness of \( \kmu \) with respect to Kripke semantics is straightforward:
\begin{theorem}
	\label{kmu-soundness}
	If \( \kmu \vdash \Phi \Rightarrow \Psi \) then \( \conj \Phi \to \disj \Psi \) is valid.
\end{theorem}
%

%----------------------------------------
\subsection{Illfounded proofs}
\label{s-ill-founded-proofs}
%----------------------------------------

Rather of expanding basic fixed point logic by induction and cut rules, a sound and complete calculus can be obtained by taking inspiration from the game semantics and admitting infinite derivations as proofs that express winning strategies.
These will be the derivations for which every branch carries an infinite thread of a suitable kind.
The system \( \kmo \) of \emph{illfounded proofs} -- \emph{co-proofs} for short, standing for \emph{co-inductive proofs} -- is a refinement of Niwiński and Walukiewicz's refutation calculus from~\cite{NiwWal:96}.

\begin{definition}\label{d-ancestor-thread}%[trace]
Let \( B \) be a branch of a \( \kmm \)-derivation \( d \), and let \( \sequent_i \) be the sequent at \( B(i) \).
A finite {trace} through \( B \) is a sequence of positions \( ( p_i )_{i \le K} \) such that \( p_0 \) is a position in \( \sequent _0 \) and for all \( i < K \),
\begin{enumerate}
%	\item \( p_i \) is a position in \( \sequent_{i} \),
	\item if \( \rules_d(B(i)) \) is not structural, then \( p_{i+1} = p_i \),
	\item if \( \rules_d(B(i)) \) is structural, then \( p_{i+1} \) is the formula occurrence corresponding to \( p_i \).
%	\begin{enumerate}
%		\item if \( p_i \) is principal, then \( p_{i+1} \) is minor,
%		\item if \( p_i \) is not principal 
%	\end{enumerate}
\end{enumerate}
We refer to \( p_i \) as a \emph{descendent} of \( p_0 \) at \( B(i) \).
An infinite trace is a sequence all of whose finite prefixes are traces.
A (finite or infinite) trace \( t = (p_i)_{i<K} \) through \( B \) induces a weak thread \( w = ( \sequent_i(p_i) )_{i<K} \) selecting the formula at position \( p_i \) in each \( \sequent _i \).
The thread \emph{supported} by \( t \) (along \( B \)) is the unique thread \( \tau = w \circ g \) where \( g \colon L \to K \) is the strictly increasing function such that \( j \in g[L] \) iff \( p_j \) is principal at \( B(j) \).
We say that \( B \) \emph{carries} \( \tau \) if \( \tau \) is an infinite thread that is support by some trace in \( B \).
\end{definition}

For a set of positions \( P \), set \( \thd_P(d,B) \) to be the set of infinite threads supported by a trace \( t \) (along \( B \)) satisfying \( t(0) \in P \).
If \( t(0) \) is a left (right) position we refer to the thread supported by \( t \) as a \emph{left} \emph{(right)} thread.
Note, a thread can be both left and right, even relative to the same branch.
We write \( \thd_p(d,B) \) in place of \( \thd_{\setof p}(d,B) \), and \( \thd ({d,B}) \) for \( \thd_P(d,B) \) where \( P \) is the set of all positions in the root sequent.
Given a subsequent \( \sequent \) of the endsequent of \( d \),
we may write \( \thd_{\sequent}(d,B) \) in place of \( \thd_P(d,B) \) where \( P \) is the set of positions corresponding to the formulas in~\( \sequent \).
\begin{proposition}
	Let \( d \) be a derivation in \( \kmm \) with a branch \( B \).
	If \( B \) carries an infinite thread starting in a guarded formula, then every infinite trace through \( B \) supports an infinite thread.
\end{proposition}
\begin{proof}
%	By inserting an application of \( (\disj) \) at the root we may assume the endsequent consists of a single formula.
	Fix a branch \( B \) and let \( \tau \in \thd(d,B) \).
%	Let \( P \) be the set of positions in the sequent \( B(0) \).
%	As each sequent along \( B \) is non-empty, Proposition~\ref{kmm-thread-unique} implies that \( F \) contains finite traces of arbitrary length.
%	Hence, \( F \) is infinite and there exists \( p_0 \in P \) such that \( F_0 \coloneqq \bigcup_{i} \thd_p(d,B(i)) \) is an infinite tree of finite branching width. 
%	By weak König's lemma \( F_0 \) contains an infinite path, i.e., an infinite trace \( t \in F \).
	As \( \tau(0) \) is guarded yet infinite, it follows that \( B \) passes through a modal rule infinitely often.
	Thus, for any trace \( t \) through \( B \), the weak thread supported by \( t \) steps through the subformula relation \( \sf \) infinitely often, i.e., \( t \) supports an infinite thread.
\end{proof}

\begin{definition}\label{d-coproofs}
	The calculus \( \kmo \) comprises the rules of \( \kmm \) (fig.~\ref{f-kmm-struc} and \ref{f-kmm-logic}) with the identity rules restricted to variables.
	A proof in \( \kmo \) is a derivation such that every branch carries either a right \( \nu \)-thread or a left \( \mu \)-thread.
	Proofs in \( \kmo \) are called \emph{illfounded proofs} or \emph{co-proofs}.
\end{definition}
We commonly refer to a branch satisfying the condition above as \emph{good}. A branch carrying neither a left \( \mu \)-, nor right \( \mu \)-thread may be called \emph{bad}.

Unlike the finitary calculus, the rule of cut is not permitted in co-proofs. 
The rule can be added, though the condition on proofhood must be weakened to allow for the good threads to initiate from cut formulas: simply replace `every branch carries' by `every branch has a suffix which carries'.
Our choice of cut-free illfounded proofs is motivated by the reduction of co-proofs to cyclic proofs in Theorem~\ref{quasi-completeness} which requires a bound on the number of distinct formulas occurring in the initial co-proof that need not exist for illfounded proofs with cut.
For the same reason, although we prove admissibility of the cut rule for \( \kmo \) (in the next section), this fact does not provide a method of \emph{eliminating} arbitrary cuts in the illfounded context.
%Nevertheless, admissibility of cut does imply a partial cut elimination theorem in the form of Theorem~\ref{pce} below which is sufficient for a `reduction' of finitary proofs into co-proofs (via cyclic proofs).
%
%Of course, soundness of the calculus \emph{with} cut and completeness of co-proofs \emph{without} cut will confirm that the two calculi are equal in terms of provable sequents, but does not establish any informative relationship between the two notions of co-proof.

Our first result concerning co-proofs is soundness.
\begin{theorem}[Soundness of Co-inductive Proofs]
	\label{kmo-snd}
	If \( \kmo \vdash \seq \Phi \Psi \) then \( \seq\Phi\Psi \) is valid.
\end{theorem}
\begin{proof}
	Without loss of generality we may assume that \( \Psi \) is empty and \( \Phi \) comprises a single formula \( \alpha \).
%	Proposition~\ref{kmm-neg} allows us to restrict to the case that \( \Psi \) is empty.
	Let \( \pi \) be a co-proof of \( \lseq \alpha \) and let \( M \) be any model.
	We show that there is no verification of \( (\alpha, M) \).
	Assume the contrary, namely that \( V \) is a verification with root \( ( \alpha , M ) \).
	The verification induces a branch \( B \) of \( \pi \) and associated sequence of vertices \( (v_i)_i \colon \omega \to M \) with the property that for every weak thread \( \tau \) induced by a trace along \( B \), the sequence \( ( \tau(j) , v_j )_{j<\lh {\tau}} \) is a weak path in \( V\! \). 
	As \( B \) carries a left \( \mu \)-thread, this contradicts that \( V \) is a verification.
\end{proof}
The proof above easily generalises to co-proofs with cut.
% --------------
\section{The Constructive Fragment}
\label{s-constr-frag}
% --------------

By restricting the rules of \( \kmo \) to intuitionistic sequents, a co-inductive sequent calculus can be given which respects intuitionistic logic.
As with the previous calculi, we refer to the calculus as the \emph{constructive fragment} of \( \kmo \) and denoted it \( \kmoi \).
The robustness of the constructive fragment is central to our proof of completeness of (classical) \( \kmu \), and is the topic of much of the remainder of this article.

As remarked in the introduction, the term `fragment' is something of a misnomer as the classical calculus \( \kmo \) can be identified with the negative fragment of the constructive \( \mu \)-calculus under the transformation described in Proposition~\ref{kmm-neg}.
That is, we may reinterpret classical co-proofs as the constructive co-proofs with endsequents \( \seq \Phi {\emptyset} \).
	
The following observation sums up the syntactic properties of constructive co-proofs.
\begin{proposition}\label{kmui-basic}
	Every branch of a constructive co-proof carries at most one right thread and every co-proof with empty consequent is constructive.
\end{proposition}
\begin{definition}
	By \( \Phi \Vdash \alpha \) we express that there exists a constructive co-proof of the sequent \( \Phi \Rightarrow \alpha \) in which the identity rules have a variable as a principal formula. 
	For formulas \( \alpha, \beta \), define \( \alpha \VdashV \beta \) as \( \alpha \Vdash \beta \) and \( \beta \Vdash \alpha \).
\end{definition}

We begin with some simple closure properties on \( \Vdash \).
\begin{proposition}
	\label{kmoi-mono}
	Let \( \alpha \), \( \beta \) be formulas, \( \phi \), \( \psi \) be predicates and \( x \) a fresh variable.
	\begin{enumerate}
		\item If \( \alpha \Vdash \beta \) and \( \phi x \Vdash \psi x \) then \( \phi \alpha \Vdash \psi \beta \).
		\item If \( \phi x \Vdash \psi x \) then \( \sigma  \phi \Vdash \sigma  \psi \).
	\end{enumerate}
\end{proposition}
\begin{proof}
	1.\ Throughout a constructive co-proof of \( \seq {\phi x} {\psi x} \), uniformly substitute \( \alpha \) for \( x \) occurring in the context and \( \beta \) for \( x \) in the head.
	The positivity constraints on abstraction ensure that the result is a constructive derivation of \( \phi\alpha \Rightarrow \psi\beta \) in which buds are of the form \( \alpha \Rightarrow \beta \).
	So \( \phi\alpha \Vdash \psi\beta \).
	
	2.\ Using the same construction as in part 1, from a constructive co-proof of \( \phi x \Rightarrow \psi x \) a constructive derivation of \( \sigma \phi \Rightarrow \sigma \psi \) can be obtained, wherein all buds are identical to the root sequent \( \sigma \phi \Rightarrow \sigma \psi \). Identifying such buds with the root and unravelling to an illfounded derivation witnesses \( \Seq {\sigma \phi} {\sigma \psi} \) for both \( \sigma = \nu \) and \( \sigma = \mu \).
\end{proof}

The next result offers our first example of a constructive co-proof.
\begin{proposition}
	\label{kmoi-id}
%	For every \( \alpha \) and \( \Phi \) there exists a tight co-proof of \( \alpha \Seq \alpha , \Phi \).
	\( \alpha \Vdash \alpha \).
\end{proposition}
\begin{proof}
	By induction on the syntactic complexity of \( \alpha \).
	Suppose \( \alpha = \sigma \phi \). 
	The induction hypothesis implies \( \phi x \Vdash \phi x \), whence \( \Seq \alpha \alpha \) by Proposition~\ref{kmoi-mono}.
	The other cases are straightforward.
\end{proof}

It was remarked in section~\ref{s-prelim} that the expanded cover modality allows for an intuitionistic treatment of disjunctive formulas not offered by unary form of the modality.
The natural disjunctive formulas arising from the normal-form construction (detailed in Section~\ref{s-DNF} and implicit in \cite{Waluk00comp-lc,ALM21uniformInter}) are intuitionistically equivalent to the starting formula when the expanded cover modality is utilised but not for the non-expanded modality.
The distinction occurs at certain stages in the construction of the disjunctive formula where in place of \( \nabla(\Gamma,\Delta) \) a formula \( \nabla \Gamma \vee \nabla \Gamma' \) is utilised (for \( \Gamma' \) dependent on \( \Gamma \) and \( \Delta \)).
The equivalence of the two forms holds classically but is not derivable in the constructive fragment.

\begin{proposition}
	\label{exp-vs-unary-cover}
	There exist formulas not equivalent over \( \kmoi \) to any disjunctive formula employing only the non-expanded (unary) cover modality.
%
%	For the disjunctive formulas \( \phi = \nabla(\emptyset, \setof{x} ) \) and \( \psi = \nabla \emptyset \vee \nabla \setof{x} \) we have that \( \kmo \vdash \seq \phi \psi \) but \( \nSeq{\phi }{\psi} \).
\end{proposition}
\begin{proof}
	Suppose \( \beta \) is a disjunctive formula satisfying \( \square x \Vdash \beta \). 
	Inspecting the witnessing co-proof we can, without loss of generality,  assume \( \beta = \nabla_L(\Gamma,\Delta) \).
%	Then, 
%	\( \square x \Vdash l \) for each \( l \in L \), \( \square x \Vdash \diamond \gamma \) for each \( \gamma \in \Gamma \), and \( \square x \Vdash \square \disj \Delta \).
	Soundness implies that \( L = \Gamma = \emptyset \) and \( \denote x \subseteq \denote{ \disj \Delta} \), i.e., \( \Gamma \subsetneq  \Delta \).
	In particular, \( \beta \) cannot be formed by use of the unary cover modality only.
%	By the previous All that remains is to observe that \( \square x \VdashV \) is whose disjunctive form presentation is the \( \VdashV \)-equivalent formula \( \conj \Setof{ \square \disj \setof{x}} \).
%
%	The claim of provability is straightforward and omitted.
%	\ad{The claim} of provability is witnessed as follows, noting that \( \phi = \square x \) and \( \psi = \square \bot \vee ({ \diamond x \wedge \square x}) \):
%	\[
%	\begin{prooftree}
%		\axiom{\seq x x }
%		\infer1[\weakR]{ \seq x {\bot , x} }
%		\infer1[\sqR]{ \seq {\phi} {\square \bot , \diamond x} }
%		
%		\axiom{\seq x x }
%		\infer1[\sqR]{ \seq {\phi} {\square x } }
%		\infer1[\weakR]{ \seq {\phi} {\square \bot , \square x} }
%		
%		\infer 2[\conjR] { \seq {\phi} {\square \bot , { \diamond x \wedge \square x} } }
%		
%		\infer 1[\disjR,\disjR] { \seq {\phi} {\psi , \psi } }
%		\infer 1[\conR] { \seq {\phi} { \psi } }
%	\end{prooftree}
%	\]
%	To observe the second claim, note that any constructive co-proof of \( \seq{\phi} \psi \) must pass through a (sub-)sequent (modulo contraction and exchange) of either \( \seq \phi {\square \bot} \) or \( \seq \phi {\diamond x } \). But neither sequent is provable by soundness.
\end{proof}

On the other hand, the expanded cover modality does provide equivalent representations of the usual modalities:

\begin{proposition}
	\( \square x \VdashV \nabla(\emptyset,\setof x) \) and \( \diamond x \VdashV \nabla(\setof x , \setof{x,\top}) \).
\end{proposition}

The next observation confirms the claimed interpretation of classical co-proofs in the constructive fragment.

\begin{proposition}
	\label{kmo-kmoi}
	\( \kmo \vdash \seq {\Phi} \Psi \) iff \( \Seq {\Phi , \neg\Psi } {\bot} \).
\end{proposition}
\begin{proof}
%	Let \( \pi \) be a co-proof witnessing \( \kmu \vdash \seq \Phi \Psi \)
	The proof transformation described in Proposition~\ref{kmm-neg} preserves the co-proof acceptance condition on branches as it simply dualises traces and threads.
	So we can assume that \( \Psi \) is empty.
	Moreover, it is clear that \( \Seq \Phi \bot \) iff \( \kmoi \vdash \lseq {\Phi} \) by Proposition~\ref{kmoi-id} and the definition of \( \bot \).
	Proposition~\ref{kmm-kmmi} connects the two equivalences.
\end{proof}

The following result will be useful later.
\begin{proposition}
	\label{kmoi-contr}
	If \( \Seq \Phi \alpha \) then there is a constructive co-proof of \( \seq \Phi \alpha \) in which principal formulas of contractions are not disjunctive.
\end{proposition}
\begin{proof}
	By permuting contractions away from the root, we can assume that the principal formula any contraction is a conjunction.
	In particular, we can assume we are working with a contraction-free constructive co-proof using an expanded (\conjL) rule:
	\begin{prooftree*}
		\hypo{\seq {\Delta , \Phi}\Psi }
		\infer[left label={$\Delta \subseteq \Gamma$}]1[\conjsL]{\seq {\conj\Gamma, \Phi} \Psi }
	\end{prooftree*}
	Consider a vertex \( v \) labelled by the instance of (\conjsL) above where the principal formula \( \conj \Gamma \) is disjunctive.
	In particular, \( \Gamma = L \cup \diamond \Gamma' \cup \setof{ \square \disj \Gamma''} \) for \( L \subseteq \Lit \) and \( \Gamma' \subseteq \Gamma'' \subseteq \D \).
	We can assume that \( \card \Delta > 1 \) as otherwise an instance of either (\conjL) or (\weakL) would suffice.
	By permuting the rule with that above if possible, we may also assume that the child of \( v \) is an instance of either an identity or modal rule.
	In the former case, as \( L \) is consistent at most one element of \( L \) occurs in the premise, and the instance of (\conjsL) can be replaced by (\conjL).
	In the latter case, we can assume that \( \Delta \subseteq \diamond \Gamma' \cup \setof{ \square \disj \Gamma''} \).
	As \( \seq{\Delta , \Phi} \Psi \) is the conclusion of a modal rule, we have that \( \Delta = \setof { \diamond \gamma , \square \disj \Gamma'' } \) for some \( \gamma \in \Gamma' \) and that \( \diamond \gamma \) is principal in the premise.
	As \( \gamma \in \Gamma'' \), it suffices to choose \( \Delta = \setof{\diamond \gamma} \).
	Although this change has the potential to alter left threads by replacing subsequences \( \conj \Gamma \sf \square \disj \Gamma'' \sf \disj \Gamma'' \sf \gamma \) by \( \conj \Gamma \sf \diamond \gamma \sf \gamma \), it does not change whether the property of carrying a good thread.
\end{proof}

\begin{theorem}
	\label{kmoi-comp-conj}
	Let \( \alpha \) be a valid conjunctive formula. Then \( \Seq {\emptyset} \alpha \).
\end{theorem}
The theorem may appear at odds with our referring to \( \kmoi \) as an \emph{intuitionistic} sequent calculus given that it derives all classical validities of certain syntactic shape.
The answer lies in the restrictive notion of conjunctive formulas.
Recall that disjunctive formulas only permit occurrences of conjunction in the form the expanded cover modality. Whence, by definition, if \( \beta = \nabla_L(\Gamma,\Delta) \) is an instance of the expanded cover modality, then \( \beta^\bot \) is valid iff \( L = \emptyset \) (since \( L \) must be consistent) and \( \Gamma^\bot \) contains a valid formula.
 In particular, no instance of the classical law of excluded middle \( \beta \vee \beta^\bot \) qualifies as conjunctive.
\begin{proof}
We perform an exhaustive proof search starting with the sequent \( \rseq\alpha \) while staying within the constructive fragment.
	That is, we build a derivation tree in (an expansion of) \( \kmoi \) such each vertex is annotated by a intuitionistic sequent.
	At the case that the head formula is a disjunction, it will be an instance of the binary co-\!\( \nabla \) modality, in which case the proof-search branches to explore the various premises that could be used in some derivation of the sequent.
	We will show that if \( \alpha \) is valid then there is a pruning of the proof-search tree to a constructive co-proof. 
	In the following, we denote by \( \neg \nabla \) the modality dual to \( \nabla \):
	\[ 
		\neg\nabla_L(\Gamma,\Delta) = \disj\bigl( L \cup \square \Gamma \cup \setof{ \diamond \conj \Delta } \bigr)
	 \]
	subject to the same constraints as uses of \( \nabla \), namely \( L \) is consistent and \( \Gamma \subseteq \Delta \).
	Treatment of \( \neg\nabla \) formulas is given by the  branching rule:
	\begin{equation*}
%		\begin{prooftree}
%			%
%			\hypo{ \seq {\alpha , \Phi }{} }
%			\infer1[\diLs]{ \seq { \diamond \Psi , \square \Phi , \Pi } { } }
%			%
%		\end{prooftree}
%		\qquad
		\begin{prooftree}
			\hypo{ \seq {} { \gamma_1 }{} }
			\hypod
			\hypo{ \seq {} { \gamma_k }{} }
			\infer3[\( \neg\nabla \)]{ \seq {}{ \neg \nabla_L ( \setof{ \gamma_1 , \dotsc , \gamma_k } , \Delta) } }
		\end{prooftree}
%		\begin{prooftree}
%			\hypo{\vphantom{\Phi}}
%			\infer1[\( x \)]{\seq{x , \neg x, \Pi}{}}
%		\end{prooftree}
%		\qquad\text{for \( \Pi \subseteq \Lit \)}
%	  \begin{prooftree}
%	  	\hypo{ \Phi , \alpha }
%	  	\hypo{\text{for each } \alpha \in \Psi }
%	  	\infer2[$\square ^*$ where $ \Pi \subseteq \Lit $]{ \diamond\Phi , \square \Psi , \Pi }
%	  \end{prooftree}
%	  \label{eqn-supp*}
	\end{equation*}
	Note that given a co-proof of some premise, say \( \rseq{\gamma_i} \), a co-proof of the conclusion is obtained via the (\( \sqR \)) and (\( \disjR \)) rules.
	Also,  if \( \Gamma \) is empty then the rule (\( \neg\nabla \)) has no premises.
	
	Let \( d \) be a maximal derivation with endsequent \( \rseq \alpha \) in the specified rules.	
	The assumption that \( \alpha \) is guarded entails that every branch of \( d \) passes through infinitely many instances of (\( \neg\nabla \)).
	Letting \( \sequent_u \) denote the sequent labelling the vertex \( u \in d \), on every branch \( B \) of \( d \) there contains a pair of vertices \( u < v \in B \) such that both are conclusions of (\( \neg\nabla \)) and \( \sequent_u \) and \( \sequent_v \) are identical sequents. 
	Such a pair of vertices is called a \emph{repetition}.
	In fact, \( d \) will be a regular tree: for every repetition \( (u,v) \), \( d \res v = d \res u \), i.e., the subderivations of \( d \) with root \( u \) and \( v \) are identical.
	
	We claim that if \( \alpha \) is valid then a subtree of \( d \) is a \( \kmo \)-proof, where the subtree is determined by collapsing instances of (\( \neg\nabla \)) to a single premise and replacing the rule as described above.
	To that aim, suppose that \( d \) cannot be pruned in such a way to form a \( \kmo \)-proof.
	From the regularity assumption we derive a (regular) subtree \( d' \subseteq d \) such that for all \( v \in d' \),
	\begin{enumerate}
		\item if \( \rules_d( v ) = \neg\nabla \) then all children of \( v \) are present in \( d' \),
		\item if \( \rules_d( v ) = \conjR \) then \( d' \) contains exactly one child of \( v \), 
		\item The unique thread carried by each branch of \( d' \) is \( \mu \).
	\end{enumerate}

	From \( d' \) a model \( M = \tuple{ T , A } \) of \( \neg \alpha \) can be defined.
	Let \( \sim \) be the equivalence relation on \( d' \) given by \( u \sim v \) if \( u \le v \) are not separated by an instance of (\( \neg\nabla \)).
	The domain of \( M \) is the set of \( \sim \)-equivalence classes with \( [u] \le [v] \) iff \( u \le v \).
%	Let \( V \) be the set of free variables in \( \alpha \) closed under negation.
	We say that a variable \( x \) \emph{occurs} at \( w \in T \) if there exists \( v \in w \) such that \( \sequent_{d}(v) \) either contains \( x \) or a formula \( \neg\nabla_L(\Gamma,\Delta) \) with \( x \in L \).
	The label function \( A \colon T \to E \) is given by \( x \in A(w) \) iff \( x \) does not occur at \( w \).
	As \( \alpha \) is conjunctive, if \( v \in d' \) is the sequent \( \rseq {\neg\nabla_L(\Gamma,\Delta)} \) then \( L^\bot \) is, by assumption, consistent and thus \( x \in A([v]) \) iff \( x \not\in L \).
	If, furthermore, \( \Gamma = \emptyset \) then \( [v] \) is a bud of \( M \) and \( M \res {[v]} \not\in \oper{\neg\nabla_L(\Gamma,\Delta)} \) by design.

	We have that \( M \not \in \oper{\alpha} \), for if \( M \in \oper{\alpha} \) with \( V \) the witnessing verification then there will be branches of \( V \) and \( d' \) carrying identical threads, which contradicts condition 3 above.
	Thus, if \( \alpha \) is a valid conjunctive formula, then \( \Seq \emptyset \alpha \).
\end{proof}

Reformulating Theorem~\ref{kmoi-comp-conj} we have
%The first is immediate given the theorem and the second follows from Proposition~\ref{kmo-kmoi}.

%\begin{corollary}\label{kmo-comp-conj}
%	If \( \alpha \) is a valid conjunctive formula, then \( \kmo \vdash \rseq \alpha \).
%\end{corollary}%
\begin{corollary}\label{kmoi-comp-disj}
	If \( \alpha \) is disjunctive and \( \neg\alpha \) is valid, then \( \alpha \VdashV \bot \).
\end{corollary}%
\begin{proof}
	The hypothesis and Theorem~\ref{kmoi-comp-conj} yields \( \kmoi \vdash \rseq {\neg\alpha} \), whence \( \Seq \alpha \bot \) by Proposition~\ref{kmo-kmoi}. That also \( \Seq \bot \alpha \) is immediate.
\end{proof}

%Applying the previous observation to the completeness theorem for \( \kmm \) we obtain a partial completeness theorem for intuitionistic co-proofs:
%%
%\begin{ntheorem}
%	\label{kmmi-comp-Pi}
%	For every guarded \( \Pi \)-formula \( \alpha \), if \( \lseq \alpha \) is valid then \( \kmo \vdash \lseq{\alpha} \).
%\end{ntheorem}
%%
%\begin{proof}
%	Theorem~\ref{kmm-comp-Sigma}\rmv{ and Proposition~\ref{kmo-kmoi}}. 
%	Proposition~\ref{kmoi-id} ensures that the initial sequents of \( \kmm \) are derivable in \( \kmo \).
%\end{proof}

% --------------
%\section{Admissibility of Cut}
%\label{s-cut-admiss}
% --------------

Our main result concerning the constructive fragment is:
\begin{theorem}[Admissibility of Cut in the Constructive Fragment]
	\label{admiss-charac}
	If \( \Phi \Vdash \alpha \)  and \( \alpha , \Psi \Vdash \beta \), then \( \Phi , \Psi \Vdash \beta \).
\end{theorem}
Theorem~\ref{admiss-charac} is equivalent to statement that the intuitionistic rule of cut:
\[
%	\label{e-cut-rule}
	\begin{prooftree}
		\subproof[\pi]{ \seq \Phi \alpha   }
		\subproof[\rho]{ \seq {\alpha , \Psi} \beta  }
		\infer[separation=2em]2[\cut]{ \seq {\Phi , \Psi } {\beta}  }
	\end{prooftree}
\]
is admissible within the constructive fragment.
The proof of the theorem proceeds by using the derivation with cut as a first approximation to the constructive co-proof of \( \seq {\Phi,\Psi} \beta \) and performs a co-inductive process of cut elimination to ultimately eliminate this instance of cut entirely.

The sequence of proof transformation that leads to elimination of the cut rule is, in essence, nothing more than Gentzen's method of `reductive' cut elimination iterated ad infinitum, producing from co-proofs \( \pi \) and \( \rho \) above a constructive derivation \( d \) with endsequent \( \seq {\Phi,\Psi} \beta \).
Following Mints~\cite{Mints-cts-ce}, reductive cut elimination is  conceived as a continuous transformation on the space of maximal derivations (relative to a suitable topology) and the challenge is not in describing the derivation \( d \) but in confirming that it is an illfounded \emph{proof}.

Preservation of proofhood under cut elimination means identifying in each branch of \( d \) either a left \( \mu \)-thread or right \( \nu \)-thread.
We will demonstrate that the threads carried by any branch of \( d \) are all carried by particular branches of \( \pi \) or \( \rho \) and that the manner in which cut elimination transforms these two proofs into \( d \) necessarily preserves sufficiently many threads to confirm that \( d \) is a co-proof.

Continuous and co-inductive cut elimination
has been established for numerous systems of illfounded proof including fixed point logics and \( \mu \)-calculi, for example \cite{ForSan13,BDS16,Savateev:2017Cut-Elimination,Baelde:2022Bouncing,Sau23,Acclavio:2024Infinitary,AfsLei24-weyl}.
The approach presented here builds on Fortier and Santocanale's cut elimination argument~\cite{ForSan13}, adding modalities and lifting the restriction on the form of sequents.
Viewed as a cut-elimination method for an intuitionistic logic, our argument is closely related to Arai's well-founded cut-elimination for fixed point predicates over Heyting arithmetic~\cite{Arai-quick}.

%\ad{To be more precise about the process of co-inductive cut elimination, we present a method of cut admissibility for intuitionistic \emph{derivations} and obtain the admissibility result for \emph{proofs} by examining the preservation of threads through the proof transformation.
%%
%The basic argument follows the ideas of Mints' continuous cut-elimination~\cite{Mints-cts-ce} in which local transformations corresponding to reductive cut elimination are applied co-recursively to `push' the cut through the infinite derivations representing the cut premises.
%}

Rather than focusing on a binary cut rule it is convenient to utilise a multi-premise version encoding a sequence of cuts called the \emph{multicut} rule:
\[
  \begin{prooftree}
	\hypo{ \seq {\Phi_1} {\alpha_1} }
	\hypod
	\hypo{ \seq{\Phi_n} {\alpha_n} }
	\hypo{ \seq { \Psi_{0} , \alpha_1 , \Psi_1 , \dotsc , \alpha_n, \Psi_{n} } {\beta} }
	\infer4[\mcut]{ \seq{\Psi_0 , \Phi_1 , \Psi_1 , \dotsc , \Phi_n , \Psi_{n} } \beta }
	\end{prooftree}
\]
The formulas \( \alpha_1, \dotsc, \alpha_n \), namely the heads of all but the final premise, are referred to as  \emph{cut formulas}.
The rightmost premise of a multicut is called the \emph{major} premise; other premises are \emph{minor}.
A \emph{multicut derivation} is a derivation in the rules of \( \kmmi + \mcut \) such that every branch passes through exactly one instance of multicut.
If the multicut is at the root we refer to the derivation as \emph{multicut rooted}.
It is important to note that the multicut rule includes the \emph{single} premise instance:
\begin{prooftree*}
	\hypo{ \seq { \Psi } {\beta} }
	\infer1[\mcut]{ \seq{ \Psi  } \beta }
\end{prooftree*}
which provides a technical advantage in tracking the multicut reductions.

Although the above form of multicut is adequate for the present setting, to account for the reduction rules necessary for either classical or full intuitionistic logic (with implication) a more liberal conception of multi-cut would be required.
A version of multicut for intuitionistic \( \mu \)-calculus with implication is presented in~\cite{AfsLei24-weyl}, with cut elimination (not merely admissibility) established for a class of illfounded proofs. 
For continuous cut-elimination for classical logic, we refer the reader to~\cite{Sau23}.
The argument we present below can be seen as a special (and simpler) case of~\cite{AfsLei24-weyl} for the logic expanded by modal operators.%

% ------------------------------
\subsection{Multicut reductions}
% ------------------------------
The local transformations on derivations consist of a collection of `reductions' \( d \mapsto d' \) where \( d \) and \( d' \) are multicut derivations and \( d \) is multicut rooted.
Thus, we assume that \( d \) combines via a multicut the constructive derivations:
\begin{itemize}
	\item \( d_i \) with endsequent \( \seq{\Phi_i}{\alpha_i} \) for \( 1 \le i \le n \), constituting the minor premises, and
	\item \( e \) with endsequent \( \seq { \Psi_{0} , \alpha_1 , \Psi_1 , \dotsc , \alpha_n, \Psi_{n} } {\beta}  \), constituting the major premise.
\end{itemize}%
In addition, we impose the restriction that the minor premise derivations do not contain any right structural rules.
This constraint is technically motivated to simplify some cases of the multicut reductions and proof of cut admissibility.
In the proof of Theorem~\ref{admiss-charac} below we show that the restriction does not reduce the generality of the theorem.

The reduction \( d \mapsto d' \) described above is visualised as:
\begin{equation}\label{e-multicut}
  \begin{prooftree}
	\subproof*[d_1]{ \seq {\Phi_1}{\alpha_1} }
	\hypod
	\subproof*[d_n]{ \seq {\Phi_n}{\alpha_n} }
	\subproof*[e]{ \seq { \Psi_{0} , \alpha_1 , \Psi_1 , \dotsc , \alpha_n, \Psi_{n} } {\beta}}
	\infer4[\mcut]{ \seq{\Psi_0 , \Phi_1 , \Psi_1 , \dotsc , \Phi_n , \Psi_{n} } \beta }
	\end{prooftree}
	\quad\mapsto\quad d'
\end{equation}
The multicut reductions fall into three categories: 
\begin{description}
	\item [{External reductions}] These apply if the root inference of any premise is a non-modal rule and the cut formula is not principal. %That is, if the root inference in some \( d_i \) is a non-modal left rule or the root inference of \( e \) is a non-modal right rule.
	\item [Internal reductions] These require that the final rule of \( e \) is a left rule and that \( \alpha_1 \) is principal. Some internal reductions additionally require that the final rule of \( d_1 \) is a right rule.
	\item[Modal reductions] These apply if the final rule of every premise is a modal inference.
\end{description}
Internal reductions `consume' the final inference in the appropriate derivations and, possibly, restructure the multicut, whereas
external reductions `permute' the final inference from premise(s) with the multicut.
The modal reductions exhibit features of both reductions.

% ----------------------
\textbf{Modal reductions}
% ----------------------
We first treat the modal reductions, which apply if all premises end in an instance of one of the two modal rules.
There are two reductions, depending on whether the final inference of \( e \) is \( \diL \) or \( \sqR \).
The latter case induces the following reduction 
\begin{multline*}
\begin{prooftree}[tight,small]
	\subproof*[d_1^*]{ \seq { \Xi_1 } {\gamma_1} }
	\infer1[\sqR]{ \seq { \square \Xi_1 } {\square \gamma_1} }
	\hypod
	\subproof*[d_n^*]{ \seq {\Xi_n } {\gamma_n} }
	\infer1[\sqR]{ \seq { \square \Xi_n } {\square \gamma_n } }
	\subproof*[e^*]{ \seq { \gamma_1 , \Lambda_1 , \dotsc , \gamma_n, \Lambda_{n} } {\delta} }
	\infer1[\sqR]{ \seq { \square \gamma_1 , \square \Lambda_1, \dotsc \square \gamma_n, \square \Lambda_{n} } {\square \delta } }
	\infer4[\mcut]{ \seq{ \square \Xi_1 , \square \Lambda_1 , \dotsc , \square \Lambda_{n}  } {\square\delta} }
\end{prooftree}
\\[.25em]\mapsto\quad
\begin{prooftree}[tight,small]
	\subproof*[d_1^*]{ \seq {\Xi_1 } {\gamma_1} }
	\hypod
	\subproof*[d_n^*]{ \seq {\Xi_n } {\gamma_n} }
	\subproof*[e^*]{ \seq { \gamma_1 , \Lambda_1 , \dotsc , \gamma_n, \Lambda_{n} } {\delta} }
	\infer4[\mcut]{ \seq{ \Xi_1 , \Lambda_1 , \dotsc , \Lambda_{n}  } {\delta} }
	\infer1[\sqR]{ \seq{ \square \Xi_1 , \square \Lambda_1 , \dotsc , \square \Lambda_{n}  } {\square \delta} }
\end{prooftree}
\end{multline*}
Following the schema of \eqref{e-multicut} we have that \( \Psi_0 = \emptyset \), and \( \Phi_i = \square \Xi_i \) and \( \Psi_i = \square \Lambda_i \) for each \( 1 \le i \le n \). 
%Furthermore, the derivations \( d_1 \), \dots, \( d_n \), \( e \) comprise a single modal rule whose premise is an hypothesis.
%
The case of \( e \) ending in \( \diL \) is similar but structurally slightly more complex.
%\begin{multline*}
%\begin{prooftree}[tight,small]
%	\subproof*[d_1^*]{ \seq {\xi , \Xi_1 } {\gamma_1} }
%	\infer1[\diL]{ \seq {\diamond \xi , \square \Xi_1 } {\diamond \gamma_1} }
%	\subproof*[d_2^*]{ \seq {\Xi_2 } {\gamma_2} }
%	\infer1[\sqR]{ \seq { \square \Xi_2 } {\square\gamma_2} }
%	\hypod
%	\subproof*[e^*]{ \seq { \gamma_1 , \Lambda_1 , \dotsc , \gamma_n, \Lambda_{n} } {\delta} }
%	\infer1[\diL]{ \seq { \diamond \gamma_1 , \square \Lambda_1, \dotsc \square \gamma_n, \square \Lambda_{n} } {\diamond \delta } }
%	\infer4[\mcut]{ \seq{ \diamond \xi , \square \Xi_1 , \square \Lambda_1 , \dotsc , \square \Lambda_{n}  } {\diamond\delta} }
%\end{prooftree}
%\\[.5em]\mapsto\qquad
%\begin{prooftree}[tight,small]
%	\subproof*[d_1^*]{ \seq {\xi , \Xi_1 } {\gamma_1} }
%	\subproof*[d_2^*]{ \seq {\Xi_2 } {\gamma_2} }
%	\hypod
%	\subproof*[e^*]{ \seq { \gamma_1 , \Lambda_1 , \dotsc , \gamma_n, \Lambda_{n} } {\delta} }
%	\infer4[\mcut]{ \seq{ \xi , \Xi_1 , \Lambda_1 , \dotsc , \Lambda_{n}  } {\delta} }
%	\infer1[\diL]{ \seq{ \diamond \xi , \square \Xi_1 , \square \Lambda_1 , \dotsc , \square \Lambda_{n}  } {\diamond\delta} }
%\end{prooftree}
%\end{multline*}

% ----------------------
\textbf{External reductions}
% ----------------------
The external reductions all permute the multicut with a rule at one of the premises provided that the principal formula is not a cut formula.
One such example is that the major premise is a (non-modal) right rule:
\[
  \begin{prooftree}[tight,small]
	%\subproof*[d_1^*]{ \seq {\Phi_1}{\alpha_1} }
	\hypod
	\subproof*[d_n^*]{ \seq {\Phi_n}{\alpha_n} }
	\hypod
	\subproof*[e^*]{{ \seq { \Psi , \alpha_1 , \dotsc } {\gamma}}}
	\hypod
	\infer3[\starR]{ \seq { \Psi_0 , \alpha_1 , \dotsc } {\beta}}
	\infer3[\mcut]{ \seq{ \Psi_0 , \Phi_1 , \Psi_1 , \dotsc , \Phi_n , \Psi_{n} } \beta }
	\end{prooftree}
	\ \mapsto\ 
  \begin{prooftree}[tight,small]
	\hypod
%	\subproof*[d_1^*]{ \seq {\Phi_1}{\alpha_1} }
	\hypod
	\subproof*[d_n^*]{ \seq {\Phi_n}{\alpha_n} }
	\subproof*[e^*]{{ \seq { \Psi_0 , \alpha_1 , \dotsc } {\gamma}}}
	\infer3[\mcut]{ \seq{ \Psi_0 , \Phi_1 , \Psi_1 , \dotsc , \Phi_n , \Psi_{n} } {\gamma} }
	\hypod
	\infer3[\starR]{ \seq { \Psi_0 , \alpha_1 , \dotsc , \Phi_n,\Psi_n} {\beta}}
	\end{prooftree}
\]
The other two scenarios are that some \( d_i \) comprises a left rule or \( e \) a right rule. The reductions in these cases is clear.

% ----------------------
\textbf{Internal reductions}
% ----------------------
The internal reductions cover the case that the cut formula \( \alpha_1 \) is principal in the major premise. 
We begin with reductions related to structural rules.
Suppose, therefore, that \( e \) ends with in a left structural rule and \( \Psi_0 = \epsilon \). 
This may be (\excL), (\weakL) or (\conL).
We present the latter.
% -------------
%
The associated internal reduction duplicates the first minor premise of the rule and inserts a (left) contraction below the multicut unless \( \Phi_1 = \epsilon \):
\[
  \begin{prooftree}[tight,small]
	\subproof*[d_1]{ \seq {\Phi_1}{\alpha_1} }
	\hypod
%	\subproof*[d_n]{ \seq {\Phi_n}{\alpha_n} }
	\subproof*[e^*]{{ \seq { \alpha_1 , \alpha_1 , \Psi_1 , \dotsc } {\beta}}}
	\infer1[\conL]{ \seq { \alpha_1 , \Psi_1, \dotsc } {\beta}}
	\infer3[\mcut]{ \seq{ \Phi_1 , \Psi_1 , \dotsc , \Phi_n , \Psi_{n} } \beta }
	\end{prooftree}
	\quad\mapsto\quad
  \begin{prooftree}[tight,small]
	\subproof*[d_1]{ \seq {\Phi_1}{\alpha_1} }
	\subproof*[d_1]{ \seq {\Phi_1}{\alpha_1} }
	\hypod
%	\subproof*[d_n]{ \seq {\Phi_n}{\alpha_n} }
	\subproof*[e^*]{{ \seq { \alpha_1 , \alpha_1 , \Psi_1 , \dotsc } {\beta}}}
	\infer4[\mcut]{ \seq{ \Phi_1 , \Phi_1 , \Psi_1 , \dotsc , \Phi_n , \Psi_{n} } \beta }
	\infer1[(\con)]{ \seq{ \Phi_1 , \Psi_1 , \dotsc , \Phi_n , \Psi_{n} } \beta }
	\end{prooftree}
\]

If the final inference in \( e \) is a logical rule, then a reduction is applicable provided that \( d_1 \) also concludes with a logical rule.
The case of that \( e \) ends in (\idLR) induces the reduction:
\[
  \begin{prooftree}[]
	\axiom[\idLR]{{ \seq { x  } {x}}}
	\axiom[\idLR]{ \seq { x }{x} }
	\infer2[\mcut]{ \seq{ x } {x} }
	\end{prooftree}
\quad\mapsto \quad
 \begin{prooftree}[]
	\axiom[\idLR]{ \seq { x }{x} }
	\end{prooftree}
\]
The reduction is similar if \( e \) ends instead with (\idL), though in that case there could be two cut formulas, in which case we require that  \( d_1 \) ends in the identity rule also.

The remaining logical multicut reductions follow a common form where the first minor premise \( d_1 \) consists of a right rule and the major premise the corresponding left rule. 
For \( \star \) a logical operator other than a modality, the following reduction is available:
\begin{multline*}
  \begin{prooftree}[tight,small]
  	\hypod
	\subproof*[d_1^*]{ \seq {\Phi_1}{\alpha} }
	\hypod
	\infer3[\starR]{ \seq {\Phi_1}{\alpha_1}}
%	\subproof*[d_2]{ \seq {\Phi_2}{\alpha_2} }
	\hypod
	\hypod
	\subproof*[e^*]{ \seq { \alpha , \Psi_1 , \dotsc } {\beta} }
	\hypod
	\infer3[\starL]{ \seq { \alpha_1 , \Psi_1 , \dotsc , \alpha_n , \Psi_n } {\beta} }
	\infer3[\mcut]{ \seq{\Phi_1 , \Psi_1 , \dotsc , \Phi_n , \Psi_{n} } \beta }
	\end{prooftree}
\\\mapsto\quad
  \begin{prooftree}[tight,small]
	\subproof*[d_1^*]{ \seq {\Phi_1}{\alpha} }
%	\subproof*[d_2]{ \seq {\Phi_2}{\alpha_2} }
	\hypod
	\subproof*[e^*]{ \seq { \alpha , \Psi_1 , \dotsc , \alpha_n , \Psi_n } {\beta} }
	\infer3[\mcut]{ \seq{\Phi_1 , \Psi_1 , \dotsc , \Phi_n , \Psi_{n} } \beta }
	\end{prooftree}
\end{multline*}
The number of premises to \( \starR/\starL \) depends on the operator \( \star \). For \( \star = \disj \) the right rule \( \starR \) is unary and \( \starL \) is multi-premise whereas for \( \star \in \setof{\mu,\nu} \) both rules are unary. In any case, there will always be one premise of each rule with identical minor formula.

By design, every multicut with only non-trivial immediate subderivations is reducible:

\begin{proposition}
	\label{mcut-app}
	Let \( d \) be a multicut rooted derivation. 
	If no immediate subderivation of \( d \) is a bud and no minor premise ends in (\weakR), then there exists a multicut derivation \( d' \) such that \( d \mapsto d' \).
\end{proposition}%
\begin{proof}
	If no external or modal reduction is applicable then every minor premise ends in a right rule and at least one cut formula is principal in the major premise.
	If the final rule of the major subderivation is structural, then a reduction is applicable.
	Otherwise, an internal logical reduction is applicable.
\end{proof}

% ------------------------------
\subsection{Admissibility of cut}
% ------------------------------
The multicut reductions can be used to establish the admissibility of the cut rule in the constructive fragment of \( \kmo \).
As remarked, the main difficulty is showing that the maximal derivation obtained by co-recursively applying is indeed a co-proof.
The argument proceeds by relating the branches and threads of the resultant derivation to the initial co-proofs.
The desired connection is expressed by the next lemma.

Let \( {\sim} \subseteq T_0 \times T_1 \) be a relation between trees \( T_0 \) and \( T_1 \). 
For branches \( B_0 \) and \( B_1 \) through \( T_0 \) and \( T_1 \) respectively,
we write \( B_0 \sim B_1 \) iff for every \( u \in B_0 \) there exists \( v \in B_1 \) such that \( u \sim v \) and for every \( v \in B_1 \) there exists \( u \in B_0 \) such that\ \( u \sim v \).
For \( T \) a set of threads, let \( \neg{T} \coloneqq \setof{ \neg{\tau} }[\tau \in T] \).
A function \( h \colon T \to T' \) between trees is \emph{order-preserving} if \( h(u) \le h(v) \) whenever \( u \le v \in T \).
For an order-preserving function \( h \colon T \to T' \) and a branch \( B \) of \( T \), we write \( h(B) \) to denote the branch of \( T' \) traversing all \( h(B(i)) \) if this is unique, or else \( h(B) \coloneqq h(B(k)) \) if \( k \) is such that \( h(B(k)) = h(B(k+i)) \) for all \( i \).

%In order to prove Theorem~\ref{admiss-charac} we require a result relating branches in the result of multicut elimination.

%\begin{ndefinition}[Homomorphism]
%	A \emph{homomorphism} between derivations \( d \) and \( e \) is an order-preserving function \( h \colon d \to e \), i.e., \( h(u) \le h(v) \) whenever \( u \le v \in d \).
%\end{ndefinition}
%
\begin{lemma}
	\label{mcut-elim}
	Let \( \pi \vdash \seq \Phi \alpha \) and \( \rho \vdash \seq {\alpha,\Psi} \beta \) be constructive co-proofs and \( \pi \) not containing (\weakR).
	There exists a maximal constructive derivation \( d \) in \( \kmo \) with endsequent 
	\( \seq{ \Phi , \Psi }\beta  
	\), an {order-preserving function} \( h \colon d \to \rho \) and a relation \( {\sim} \subseteq \pi \times d \) satisfying, for every branch \( B \subseteq d \),
	\begin{enumerate}
		\item\( \Tr[\seq\Psi \beta]{\rho ,h(B)} = \Tr[\seq\Psi \beta]{ d ,B} \).
		\item 
		\( \bigcup \Setof{ \Tr[\seq{\Phi}{}]{ \pi , E }}[ B \sim E ] = \Tr[\lseq{\Phi}]{ d ,B} \).
		\item \( {\Tr[\lseq{\alpha}]{\rho,h(B)}} = \bigcup \Setof{ \nTr[\rseq{\alpha}]{ \pi , E }}[ E \sim B ] \).
%		\item \( e \le_{\Phi} d \).
	\end{enumerate}
\end{lemma}
%
%
%\ad{%
\begin{proof}
The derivation \( d \) is defined as the limit of a increasing sequence of finite \( \kmo \)-derivations \( (d_j)_{j<\omega} \).
It is technically convenient to assume that the combining rules (\weak) and (\con) are available as structural rules in the formation the \( d_j \).
Each derivation \( d_j \) will be associated a function \( m_j \) which assigns to the buds of \( d_j \) instances of the multicut rule whose conclusion matches the bud sequent.
This data will be termed a \emph{multicut-tipped} derivation, defined to be a pair \( (d,m) \) of a finite derivation \( d \) and a function \( m \) mapping each bud \( b \) of \( d \) to a sequence \( m(b) = ( u_1 , \dotsc , u_n , u ) \in \pi^n \times \rho \) such that
the sequents \( \sequent_\pi( u_1)  \), \dots, \( \sequent_\pi(u_n) \) enumerate the minor premises of a multicut, \( \sequent_\rho(u) \) is the major premise and the sequent at \( b \) is the conclusion of this multicut.
Formally, \( m \) should also specify which formula occurrences in \( \sequent_\rho(u) \) are cut formulas but we leave this information implicit.
Thus, the function \( m \) in a multicut-tipped derivation \( ( d,m) \) associates to each bud \( b \) of the multicut-rooted derivation
\begin{gather}\label{e-mcut-lem}
\begin{prooftree}[compact,small]
	\subproof*[\pi\res {u_1}]{ \seq {\Phi_1}{\alpha_1} }
	\hypod
	\subproof*[\pi\res{u_n}]{ \seq{\Phi_n}{\alpha_n} }
	\subproof*[\rho \res u]{ \seq { \Psi_{0} , \alpha_1 , \Psi_1 , \dotsc , \alpha_n, \Psi_{n} } {\beta} }
	\infer4[\mcut]{ \seq{\Psi_0 , \Phi_1 , \Psi_1 , \dotsc , \Phi_n , \Psi_{n} } \beta }
\end{prooftree}
\quad\text{where } m(b) = (u_1, \dotsc, u_n , u)
\end{gather}
The construction begins with the trivial multicut-tipped derivation
\( (d_0,m_0) \) where \( d_0 \) is a single bud assigned the initial two-premise multicut:
\[
	d_0 \coloneqq \left\{
	\begin{prooftree}
		\axiom[\hyp]{\seq {\Phi , \Psi } {\beta}}
	\end{prooftree}
	\right.
	\text{ and}  \thickspace
	m_0 \colon \epsilon \to (\epsilon , \epsilon) \text{ representing }
	\begin{prooftree}
		\subproof*[\pi]{ \seq \Phi \alpha   }
		\subproof*[\rho]{ \seq {\alpha , \Psi} \beta  }
		\infer2[\mcut]{ \seq {\Phi , \Psi } {\beta}  }
	\end{prooftree}
\]
%
%%
%\begin{prooftree*}
%	\subproof*[\rho_0]{ \seq {\Psi_0}{\alpha_0} }
%	\hypod
%	\subproof*[\rho_n]{ \seq{\Psi_n}{\alpha_n} }
%	\subproof*[\pi]{ \seq{\Theta_0 , \alpha_0 , \dotsc, \Theta_n , \alpha_n , \Theta}\beta }
%	\infer4[\mcut]{ \seq{\Theta_0 , \Psi_0 , \dotsc, \Theta_n , \Psi_n , \Theta}\beta }
%\end{prooftree*}
%%
%Now suppose that \( ( d_j , m_j ) \) is defined.
%We obtain \( d_{j+1} \) by replacing each bud in \( d_j \) by a multicut-tipped finite derivation given by applying a reduction rule to multicut identified by \( m_j \).
%Observe that \( d_{j+1} \) extends \( d_j \) by either a derivation without multicuts (via the internal right-weakening reduction) or a derivation of depth at most one,
%the latter holding because the multicut-tipped derivation inserted at \( b \) will comprise at most one rule instance whose premises are buds.
%%We give three examples.

Let \( b \) be a bud of \( d_j \) and suppose \( m_j(b) \) describes the multicut in \eqref{e-mcut-lem} above.
Proposition~\ref{mcut-app} implies that some multicut reduction is applicable.
Applying that reduction replaces the multicut at \( b \) by a finite derivation \( d \) in which each bud is assigned a multicut built from subderivations of \( \pi\res u_1 \), \dots, \( \pi\res u_n \), \( \rho\res{u} \).
%The derivation \( d_{j+1} \) extends \( b \) by \( d' \).
Some reductions replace the multicut in \eqref{e-mcut-lem} directly with another multicut while other reductions, such as the external reductions, insert a non-trivial derivation whose buds are multicuts.
%If \( d_{j+1} \) properly extends \( b \) then the partial function \( h_{j+1} \) and relations \( \sim_{j+1} \) are extended accordingly.
%We give three examples.
We describe the process for the three forms of multicut reduction.

\textbf{Internal reduction} We present the case of the non-modal logical rule.
Assume that the reduction applied to the multicut in \eqref{e-mcut-lem} for the connective \( \star \) is:
\begin{multline*}
  \begin{prooftree}[tight,small]
  	\hypod
	\subproof*[\pi\res u_1 i]{ \seq {\Phi_1}{\gamma} }
	\hypod
	\infer3[\starR]{ \seq {\Phi_1}{\alpha_1}}
%	\subproof*[\pi\res u_2]{ \seq {\Phi_2}{\alpha_2} }
	\hypod
	\hypod
	\subproof*[\rho \res u j]{ \seq { \gamma , \Psi_1 , \dotsc } {\beta} }
	\hypod
	\infer3[\starL]{ \seq { \alpha_1 , \Psi_1 , \dotsc , \alpha_n , \Psi_n } {\beta} }
	\infer3[\mcut]{ \seq{\Phi_1 , \Psi_1 , \dotsc , \Phi_n , \Psi_{n} } \beta }
	\end{prooftree}
\\\mapsto\quad
  \begin{prooftree}[tight,small]
	\subproof*[\pi\res u_1 i]{ \seq {\Phi_1}{\gamma} }
%	\subproof*[\pi\res u_2]{ \seq {\Phi_2}{\alpha_2} }
	\hypod
	\subproof*[\rho\res u j]{ \seq { \gamma , \Psi_1 , \dotsc , \alpha_n , \Psi_n } {\beta} }
	\infer3[\mcut]{ \seq{\Phi_1 , \Psi_1 , \dotsc , \Phi_n , \Psi_{n} } \beta }
	\end{prooftree}
\end{multline*}
In this case, \( b \) remains a bud in \( d_{j+1} \) but is assigned the multicut \( m_{j+1}(b) = ( u_1 i , u_2 , \dotsc, u_n , u j ) \).
The internal reduction for structural rules is analogous.

\textbf{External reduction} Suppose the rule at \( u \) (in \( \rho \)) is (\( \starR \)) and the following reduction is applied.
\begin{multline*}
  \begin{prooftree}[tight,small]
	\subproof*[\pi\res u_1]{ \seq {\Phi_1}{\alpha_1} }
	\hypod
	\subproof*[\pi\res u_n]{ \seq {\Phi_n}{\alpha_n} }
	\hypod
	\subproof*[\rho\res ui]{{ \seq { \Psi_0 , \alpha_1 , \dotsc } {\beta_i}}}
	\hypod
	\infer3[\starR]{ \seq { \Psi_0 , \alpha_1 , \dotsc } {\beta}}
	\infer4[\mcut]{ \seq{ \Psi_0 , \Phi_1 , \Psi_1 , \dotsc , \Phi_n , \Psi_{n} } \beta }
	\end{prooftree}
	\\\mapsto\quad 
  \begin{prooftree}[tight,small]
	\hypod
	\subproof*[\pi\res u_1]{ \seq {\Phi_1}{\alpha_1} }
	\hypod
	\subproof*[\pi\res u_n]{ \seq {\Phi_n}{\alpha_n} }
	\subproof*[\rho\res ui]{{ \seq { \delta_{i}, \Psi , \alpha_1 , \dotsc } {\beta_i}}}
	\infer4[\mcut]{ \seq{ \Psi_0 , \Phi_1 , \Psi_1 , \dotsc , \Phi_n , \Psi_{n} } {\beta_i} }
	\hypod
	\infer3[\starR]{ \seq { \Psi_0 , \alpha_1 , \dotsc } {\beta}}
	\end{prooftree}
\end{multline*}
Let \( (d,m) \) be the finite multicut-tipped derivation comprising only the rule (\( \starL \)) with the \( i \)-th premise assigned the multicut \( (u_1 , \dotsc, u_n , ui ) \).
Then \( d_{j+1} \) extends \( d_j \) by inserting \( d \) in place of the bud \( b \), and \( m_{j+1} \) maps the new buds to the corresponding multicut.

\textbf{Modal reduction}
The case of the modal reduction combines the two reduction forms above.
\medskip

With \( (d_j)_{j<\omega} \) defined, where \( d_{j+1} \) extends \( d_j \), set \( d \) to be the limit of the sequence, defined in the obvious way.
For \( (u,v) \in \pi \times d \), set \( u \sim v \) iff \( v \) is a bud of some \( d_j \) and \( u \) occurs as a minor premise of \( m_j(v) \).
The function \( h \colon d \to \rho \) is defined as tracking the major premise of the multicuts through the approximations of \( d \).
Every vertex in \( d \) is a bud of at least one of the derivations approximating \( d \).
For \( v \in d \) set \( h(v) \) to be the major premise of the multicut \( m_j(v) \) where \( j \) is minimal such that \( v \in d_j \),
%Otherwise, set \( h(v) = h(v') \) where \( v' \le v \)  is maximal such that \( v' \) is a bud of some \( d_j \).
%
This function is order-preserving.

Let \( B \) be any branch of \( d \), and let \( b_j \) be the bud of \( d_j \) along \( B \).
Tracing the sequence of multicut reductions that generate \( B \), it is clear that the second condition of the lemma holds.
To establish the first and third condition we must examine \( h(B) \).
We first consider the case that \( h(B) \) is not a branch, i.e., \( h(b_j) = h(b_{j+1}) \) for all \( j > k \).
Recall that if \( b_j < b_{j+1} \) then \( h(b_{j+1}) \) is the major premise of \( m_{j+1}(b_{j+1}) \) if this exists.
Hence, for \( h(B) \) not to be a branch, no subsequent reduction alters the major premise.
That is, no branch \( E \) of \( \pi \) for which \( E \sim B \) carries an \( \alpha \)-trace due to only finitely many right rules along \( E \), meaning that the three conditions of the lemma are met in this case.

Thus, we may assume that \( B \) is a branch of \( d \) and \( h(B) \) defines a branch of \( \rho \).
Tracing the multicut reductions \( B \), every left \( \alpha \)-thread carried by \( h(B) \) induces a branch \( E \sim B \) of \( \pi \) carrying the dual thread, and vice versa. Hence, the third condition is met.
Furthermore, each non-cut thread \( \tau \in \Tr[\seq \Psi\beta]{\rho,h(B)} \) is carried by \( B \) by virtue of the external reductions, whence the first condition.

All that remains is to show that \( d \) is maximal.
Suppose the contrary, namely that \( d \) has a bud \( b \). 
Let \( k \) be such that \( b \in d_k \).
In particular, \( b \) is a bud in \( d_{k+j} \) for all \( j \) meaning that the multicut reduction which transforms  \( m_{k+j}(b) \) to \( m_{k+j+1}(b) \) must be internal for all \( j \). 
But this scenario leads to a sequence of branches \( E_1 , \dotsc, E_n \) of \( \pi \) and a branch \( F \) of \( \rho \)  such that the \( E_i \) carry only right threads, \( F \) carries left \( \alpha \)-threads only, and
\(
	\bigcup_{i\le n} \Tr{\pi,E_i} = \nTr{\rho,F}
\), contradicting that \( \pi \) and \( \rho \) are co-proofs.
\end{proof}
%}

Theorem~\ref{admiss-charac} follows directly from Lemma~\ref{mcut-elim}:
\paragraph{Proof of Theorem~\ref{admiss-charac}:}
Suppose \( \Seq \Phi \alpha \) and \( \Seq {\alpha,\Psi} \beta \). Let \( \hat \alpha \) be the result of augmenting each (derived) subformula of \( \alpha \) by disjunction with \( \bot \). Then \( \Seq \Phi {\hat \alpha} \) by replacing in the co-proof of \( \Seq \Phi \alpha \) every head formula \( \gamma \) by \( \hat\gamma \) and inserting an instance of (\disjR) after every right logical rule. We can, furthermore, assume that the co-proof of \( \seq \Phi {\hat\alpha} \) does contain instance of (\weakR) since the principal in this case will be \( \hat\gamma \vee \bot \) for some \( \gamma \in \subform \alpha \), and a minor formula of \( \bot \) can be chosen instead.
Also, \( \Seq {\hat\alpha,\Psi} \beta \) is easily established.
Thus, we may assume constructive co-proofs \( \pi \vdash \seq \Phi {\hat\alpha} \) and \( \rho \vdash \seq {\hat\alpha,\Psi} \beta \) with \( \pi \) not containing any right structural rules.

Applying Lemma~\ref{mcut-elim}, we obtain a maximal derivation \( d \) satisfying properties 1--3.
Let \( B \) be a branch of \( d \) and consider the branch \( h(B) \) of \( \rho \) (if \( h(B) \) is not a branch a similar argument applies).
Towards a contradiction, suppose \( B \) is bad. 
As \( h(B) \) is good, there must be a left \( \mu \)-thread \( \tau \in \Tr[\lseq \alpha]{\rho,h(B)} \) initiating from the cut formula \( \alpha \). By condition 3, \( \tau \in \nTr[\rseq \alpha]{\pi,E} \) for some \( E \sim B \). 
As every path in a constructive co-proof carries at most one right trace (Proposition~\ref{kmui-basic}) and \( \tau \) is \( \mu \), there is a branch \( E \sim B \) carrying a left \( \mu \)-thread. 
Condition 2 implies that \( \Tr{d,B} \) contains the same thread, whence \( B \) is good. 
\hfill\( \square \)

\bigskip
Theorem~\ref{admiss-charac} can be generalised in two ways.
The method of multicut admissibility can be lifted to a (partial) cut elimination result to the effect that any co-proof with finitely many distinct cut formulas (but, potentially, infinitely many cut \emph{occurrences}) can be transformed into a cut-free co-proof.
The proof, described in~\cite{AfsLei24-weyl}, requires a more complex conception of multi-cut than presented above and a more careful analysis of the limit co-derivation.
The second generalisation is to a cut-admissibility theorem for the constructive modal \( \mu \)-calculus including implication.
The two generalisations are essentially equivalent: by projecting cuts into the root sequent a co-proof of \( \seq \Phi \alpha \) with finitely many distinct cut-formulas \( \beta_0 \), \dots, \( \beta_n \) can be directly transformed into a cut-free co-proof of \( \seq {\hat\beta, \Phi} \alpha \) where \( \hat \beta = \nu x\, ( \square x \wedge \conj\setof{ \beta_i \to \beta_i }[i \le n] )\).

% ------------------------------
\subsection{Guardedness}
% ------------------------------

Building on the closure properties at the start of this section, we can show that every formula is constructively equivalent to a guarded formula.
The argument follows the usual transformation (see, e.g., \cite{NiwWal:96}) of partially unfolding quantifiers, permuting propositional connectives and trivialising the remaining unguarded variable occurrences, though with a reliance on constructive co-proofs.
The next proposition, proved by syntactic induction on formulas, presents the main reductions.

\begin{proposition}
	\label{guard-reds}
	For all predicates \( \phi \), \( \psi \), formulas \( \alpha \), sets \( \Delta_0 , \dotsc, \Delta_n \), and operators \( \sigma  \in \setof{ \mu ,\nu } \) and \( \bigcirc \in \setof{ \conj,\disj} \):
	\begin{enumerate}
		\item \( \sigma x .\, \phi x \VdashV \sigma  x.\, \sigma  \phi \) where \( x \) is permitted to occur freely in \( \phi \).
		\item \( \sigma ( \phi \cdot \psi ) \VdashV \phi( \sigma  ( \psi \cdot \phi ) ) \) where \( \phi \cdot \psi  = \lambda  x.\, \phi(\psi x)\) for \( x \) not occurring in \( \phi, \psi \).
		\item \( \bigcirc \setof{\alpha} \VdashV \alpha \).
		\item \( \bigcirc \bigl( \Delta_0 \cup \setof{ \bigcirc \Delta_1 } \bigr) \VdashV \bigcirc \bigl( \Delta_0 \cup \Delta_1 \bigr) \).
		\item \( \bigcirc \Setof{ \neg\bigcirc \Delta_i }[i \le n ] \VdashV \neg\bigcirc \Setof{ \bigcirc \Delta }[ \card{\Delta \cap \Delta_i } = 1 \text{ for all }i\le n] \).
	\end{enumerate}
\end{proposition}
%\begin{proof}
%	Straightforward.
%\end{proof}

\begin{proposition}
	\label{guard-reds-f}
	Suppose \( \alpha \VdashV \hat{\alpha} \). If \( \beta \) contains \( \alpha \) as a {literal} subformula and \( \hat{\beta} \) is the result of replacing one of these occurrence of \( \alpha \) by \( \hat{\alpha} \), then \( \beta \VdashV \hat{\beta} \).
\end{proposition}
\begin{proof}
	Induction on the syntactic complexity of \( \beta \) using Proposition~\ref{kmoi-mono}.
\end{proof}

\begin{theorem}[Expressive Adequacy of the Guarded Fragment]\label{guard-equi}
	For every \( \alpha \) there exists a guarded formula \( \gamma \VdashV \alpha \).
\end{theorem}
The proof of the theorem makes use of transitivity of \( \Vdash \) to iterate the `single-step' reductions above.
\begin{proof}
	Call a predicate \( \phi = \lambda x.\, \alpha \) \emph{guarded} if every occurrence of \( x \) in \( \alpha \) is under the scope of a modal operator,
	and \emph{strongly unguarded} if \emph{no} occurrence of \( x \) in \( \alpha \) is under the scope of a modal operator.
	Via the first equivalence of Proposition~\ref{guard-reds} and Proposition~\ref{guard-reds-f}, every formula \( \alpha \) can be converted to a formula \( \hat{\alpha} \VdashV \alpha \) in which every subpredicate \( \phi \) in \( \hat\alpha \) is either {guarded} or {strongly unguarded}.
	To complete the reduction to guarded form it suffices to observe the following equivalences and apply Proposition~\ref{guard-reds-f} iteratively.
	\begin{itemize}
		\item If \( \phi \) is strongly unguarded, then \( \mu  \phi \VdashV \phi \bot \) and \( \nu  \phi \VdashV \phi \top \).
	\end{itemize}
	To establish the claim, first apply proposition~\ref{guard-reds} to reduce \( \phi \) to the forms \( \lambda x.\, \disj \setof{ \conj \Gamma_i }[i\le m] \) and \( \lambda x.\, \conj \setof{ \disj \Delta_i }[i\le n] \) in which all occurrences of \( x \) are as elements of the \( \Gamma_i \)s and \( \Delta_j \)s.
	Propositional logic confirms that \( \phi(\phi x) \VdashV \phi x \), whence the claim obtains.
\end{proof}

% --------------
%\clearpage
\section{Cyclic Proofs}
\label{s-cyclic-proofs}
% --------------

We now introduce a calculus in which proofs are finite representations of co-inductive proofs.
That is, proofs in the calculus are finite derivations for which each bud can be connected to an earlier occurring vertex, called the companion, labelled by, essentially, the same sequent whose unravelling to a derivation induces a co-proof.
Our notion of \emph{cyclic} proof is more constrained than merely asserting that the tree unravelling of the cyclic derivation is a co-proof.
Assessing whether a cyclic derivation is a cyclic \emph{proof} in our sense is a purely syntactic condition on the local cycle between companion and bud.
This is in contrast to requiring that the tree unravelling is its self a cyclic proof, for which one would have to confirm the thread condition on every induced cycle, not merely the simple ones.

In the context of the proof of completeness of \( \kmu \), the cyclic calculus \( \cmu \) introduced below plays two roles. 
One is as an intermediate system between the illfounded and finitary calculi. 
The second is as a framework for the proof-search method on the path to the disjunctive normal theorem (cf.\ Section~\ref{s-DNF}). In the latter role, we will only utilise the concept of a derivation in \( \cmu \) and not cyclic \emph{proofs} per se.
The system \( \cmu \) is essentially the calculus \( \system{Clo} \) of cyclic proofs of~\cite{AL17lics} extended with the rule of cut, though the presentation below simplifies some details compared to earlier works.
It is, however, straightforward to show that the cut-free fragment of \( \cmu \) derives the same sequents as \( \system{Clo} \).

It was claimed in~\cite{AL17lics} that the cut-free fragment of \( \cmu \) is complete for the modal \( \mu \)-calculus.
A counter-example was recently discovered by Johanness Kloibhofer~\cite{Kloib23-counter-ex} that confirms the contrary.
As a consequence of the present results, the cut-free fragment of \( \cmu \) is complete for sequents \( \seq \Phi \alpha \) where \( \Phi \) consists entirely of \( \PD \) formulas.
In Section~\ref{s-completeness} we prove that in the presence of cut, partial completeness can be lifted to full completeness.

As mentioned, the objects of deduction in \( \cmu \) are \emph{annotated} sequents.
These are ordinary sequents in which quantified subformulas are annotated by symbols from a designated set and the sequent as a whole carries a global `control' on the annotations contained within.
Fix an infinite set \( \nms \); elements of \( \nms \) are called \emph{names}. 
An \emph{annotation} is a finite non-repeating sequence of names.
The set of annotations is denoted \( \ann \).
Each annotation \( a \) is associated variants of the quantifiers, marked as \( \nf[a] \) and \( \mf[a] \).
The plain `unannotated' quantifiers \( \nu \) and \( \mu \) are identified with the quantifiers associated to the empty annotation \( \epsilon \).
Annotated formulas are generated by the grammar that admits these new quantifiers. 
To reduce confusion, in the following we often refer to formulas without annotations as \emph{plain} formulas.

\begin{definition}%[Annotated formula]
	The annotated formulas are formed by expanding the syntax of plain formulas with quantifiers \( \nu^a \) and \( \mu^a \) for each annotation \( a \):
	\[\begin{aligned}
	  \alpha &\Coloneqq 
		x \mid 
%		\neg x \mid 
		\square  \alpha \mid 
		\diamond \alpha \mid 
		\conj \Gamma \mid
		\disj \Gamma \mid 
		\mf[a] \phi \mid 
		\nf[a] \phi 
		\quad ( a \in \ann )
	  \\
	  \Gamma &\Coloneqq \emptyset \mid \Gamma \cup \setof{\alpha}
	\\
	  \phi &\Coloneqq \lambda x.\, \alpha, \text{ for \( x \)-positive \( \alpha \).}
	\end{aligned}\]
	If the quantifier \( \sigma^a \) occurs in \( \alpha \) we say that \( a \), and \( a(i) \) for each \( i < \lh a \), \emph{occurs} in \( \alpha \).
%	An \emph{a-predicate} is a predicate in the sense of the grammar above.
	For \( \sigma \in \setof{\mu,\nu} \), the \( \sigma \)-annotated formulas are the annotated formulas in which only the quantifier \( \sigma \) has non-trivial annotations.
\end{definition}

Negation is lifted to annotated formulas by setting \( \neg{(\sigma^a \phi)} \coloneqq \bar\sigma^a \neg {\phi} \) where \( \bar\sigma = \neg \sigma \).
The (plain) formula obtained by eliminating all annotations from \( \alpha \) is denoted \( \alpha^- \).
We write that \( \alpha^- \) \emph{underlies} \( \alpha \), and \( \alpha \) \emph{annotates} \( \alpha^- \).
%Let \( \alpha^- \coloneqq \alpha^{\bot\bot} \) be the plain formula underlying \( \alpha \), obtained by removing all annotations in \( \alpha \).
%Observe that \( \alpha = \alpha^{-} \) iff \( \alpha \) is plain.
%and any annotated formula \( \beta \) such that \( \alpha = \beta^{-} \) is called an \emph{annotation of \( \alpha \)}.

\begin{definition}%[Annotated sequent]
	An \emph{annotated sequent} is an expression \( \aseq [c] \Phi \Psi \) where \( c \) is an annotation, called the \emph{control}, and \( \Phi \) and \( \Psi \) are finite sequences of \( \mu \)-annotated and, respectively, \( \nu \)-annotated formulas 
	 such that every annotation occurring in \( \seq \Phi\Psi \) is a subsequence of \( c \).
%	The set of annotated sequents is denoted \( \asequents \).
	Plain sequents are identified with annotated sequents with empty control.
\end{definition}
Given an annotated sequent \( \sequent \) and a sequence or set of names \( c \), let \( \sequent^{\upharpoonright c} \) be the sequent resulting from removing all names in \( \sequent \) not present in \( c \).
Note, if \( \sequent \) is an annotated sequent, then so is \( \sequent^{\upharpoonright c} \).

We now proceed with the definition of the cyclic sequent calculus \( \cmu \).
The logical rules of \( \cmu \) are presented in fig.~\ref{f-ann-rules-struc} and \ref{f-ann-rules} and comprise the natural liftings of \( \kmo + \cut \) to annotated sequents with the addition of structural rules for manipulating annotations.
Implicit in all rules is that the premises and conclusions are annotated sequents.
In particular, in instances of (\nuR[n]) the name \( n \) acts as an eigenvariable as it cannot occur in the endsequent (since \( cn \) is an annotation), and cut formulas are always plain (as they are both \( \mu \)- and \( \nu \)-annotated).
%In addition to the structural rules of \( \kmu \), there are two structural rules for manipulating annotations.
The rules (\awR) and (\awL) allows annotations within subformulas to be appended by names present in the control. 
We denote by (\aw) the rule incorporated finite iterations of the annotation weakening and exchange rules.
The rule (\cw) expands the control by a fresh name, though in the following we assume this rule also incorporates finite iterations of the operation.
The definition of traces lifts to derivations in \( \cmu \) in the obvious way.

\begin{definition}[Cyclic proofs]\label{d-cmu}%
	The system \( \cmu \) of \emph{cyclic} proofs has the rules listed in fig.~\ref{f-ann-rules-struc} and \ref{f-ann-rules}, with the former comprising the \emph{structural} rules.
	A proof in \( \cmu \), called a \emph{cyclic proof}, is a finite derivation \( d \) whose root is an unannotated sequent and is such that for every path \( (v_i)_{i\le k} \) from the root of \( d \) to a bud, there exists \( l < k \) and \( n \in \nms \) satisfying three conditions:
	\begin{enumerate}
		\item The rule at \( v_l \) is either (\nuR[n]) or (\muL[n]),
		\item the name \( n \) occurs in the control of \( v_i \) for all \( l < i \le k  \), and 
		\item \( \sequent_d(v_k)^{\upharpoonright c_l} = \sequent_d(v_l) \) where \( c_l \) is the control of \( v_l \).
	\end{enumerate}
	The vertex \( v_l \) above is referred to as a \emph{companion} of the bud \( v_k \).
	
	The \emph{constructive fragment} of \( \cmu \) is the system \( \cmui \) restricting the rules of \( \cmu \) to intuitionistic sequents, namely sequents with at most one head formula.
\end{definition}
Inserting instances of structural rules as necessary (namely (\aw), (\cw) and (\exc)), it can be assumed that  bud--companion pairs are unique and  have the form represented in fig.~\ref{f-bud-schema}.
%In particular, it can be assumed that each bud has exactly one companion.
%
\begin{figure}
	\centering
	\ebproofset{small}
	\begin{tabular}{c@{\quad}c@{\quad}c@{\quad}c}
		\begin{prooftree}
			\hypo{ \seq[c] {\Phi , \alpha , \Lambda} \Psi }
			\infer1[\excL]{ \seq[c] {\alpha , \Phi , \Lambda} \Psi }
		\end{prooftree}
		&
		\begin{prooftree}
			\hypo{ \seq[c]{\alpha , \alpha ,  \Phi} \Psi }
			\infer1[\conL]{ \seq[c]{\alpha , \Phi} \Psi }
		\end{prooftree}
		&
		\begin{prooftree}
			\hypo{ \seq[c] \Phi \Psi }
			\infer1[\weakL]{ \seq[c] {\Lambda,\Phi} {\Psi} }
		\end{prooftree}
		&
		\begin{prooftree}
			\hypo{ \seq [c] { \phi ( \mf [a] \psi ), \Phi } \Psi }
			\infer1[\awL]{ \seq [c] { \phi ( \mf [an] \psi ) , \Phi } \Psi }
		\end{prooftree}
		\\[1.5em]
		\begin{prooftree}
			\hypo{ \seq[c] {\Phi} {\Psi, \alpha , \Lambda} }
			\infer1[\excR]{ \seq[c] {\Phi} {\Psi,\Lambda,\alpha} }
		\end{prooftree}
		&
		\begin{prooftree}
			\hypo{ \seq[c] \Phi {\Psi , \alpha , \alpha}  }
			\infer1[\conR]{ \seq[c] \Phi {\Psi , \alpha} }
		\end{prooftree}
		&
		\begin{prooftree}
			\hypo{ \seq[c] \Phi \Psi }
			\infer1[\weakR]{ \seq[c] {\Phi} {\Psi,\Pi} }
		\end{prooftree}
		&
		\begin{prooftree}
		\hypo{ \seq [c] \Phi { \Psi , \phi ( \nf [a] \psi )} }
		\infer1[\awR]{ \seq [c] \Phi { \Psi , \phi ( \nf [an] \psi )}  }
		\end{prooftree}
		\\[1.5em]
		\begin{prooftree}
			\axiom[\hyp]{\aseq \Phi \Psi}
		\end{prooftree}
		&
		\multicolumn{2}{c}{
			\begin{prooftree}
			\hypo{ \seq[c] \Phi {\Psi , \alpha} }
			\hypo{ \seq[c] {\alpha , \Phi} \Psi }
			\infer2[\cut]{ \seq[c] \Phi \Psi }
		\end{prooftree}
		}
		&
		\begin{prooftree}
			\hypo{ \seq[c d] \Phi \Psi }
			\infer1[\cw]{ \seq [c a d] \Phi \Psi }
		\end{prooftree}
	\end{tabular}
	\caption{Structural rules of \( \cmu \).}
	\label{f-ann-rules-struc}
\end{figure}
\begin{figure}
	\centering
	\ebproofset{small}
	\begin{tabular}{c@{\quad}c@{\quad}c@{\quad}c}
		\begin{prooftree}
			\hypo{ \aseq {\gamma , \Phi} \Psi }
			\infer[left label={$\gamma \in \Gamma$}]1[\conjL]{ \aseq {\conj \Gamma , \Phi} \Psi }
		\end{prooftree}
		&
		\begin{prooftree}
			\hypo{ \Setof{ \aseq {\gamma , \Phi} \Psi }[\gamma \in \Gamma] }
			\infer1[\disjL]{ \aseq {\disj \Gamma , \Phi} \Psi }
		\end{prooftree}
		&
		\begin{prooftree}
			\hypo{ \aseq {\alpha , \Phi} \Psi }
			\infer1[\diL]{ \aseq {\diamond \alpha , \square \Phi} {\diamond \Psi} }
		\end{prooftree}
		&
		\begin{prooftree}
			\hypo{ \aseq {\phi ( \sigma  \phi ) , \Phi} \Psi }
			\infer1[\nuL]{ \aseq {\sigma \phi , \Phi} \Psi }
		\end{prooftree}
		\\[1.5em]
		\begin{prooftree}
			\hypo{ \setof{ \aseq \Phi {\Psi , \gamma} }[ \gamma \in \Gamma ] }
			\infer1[\conjR]{ \aseq \Phi {\Psi , \conj \Gamma} }
		\end{prooftree}
		&
		\begin{prooftree}
			\hypo{ \aseq \Phi {\Psi , \gamma} }
			\infer[left label={$\gamma \in \Gamma$}]1[\disjR]{ \aseq \Phi {\Psi , \disj \Gamma } }
		\end{prooftree}
		&
		\begin{prooftree}
			\hypo{ \aseq \Phi {\Psi , \alpha} }
			\infer1[\sqR]{ \aseq {\square \Phi} {\diamond \Psi , \square  \alpha} }
		\end{prooftree}
		&
		\begin{prooftree}
			\hypo{ \aseq \Phi {\Psi , \phi ( \sigma  \phi ) }}
			\infer1[\muR]{ \aseq \Phi {\Psi , \sigma \phi }}
		\end{prooftree}
		\\[1.5em]
		\begin{prooftree}
			\hypo{ \aseq[cn] {\phi ( \mf[an]  \phi ) , \Phi} \Psi }
			\infer1[\amuL n]{ \aseq[c] {\mf[a] \phi , \Phi} \Psi }
		\end{prooftree}
		&
		\begin{prooftree}
			\hypo{ \aseq[cn] \Phi {\Psi , \phi ( \nf[an]  \phi ) }}
			\infer1[\anuR n]{ \aseq[c] \Phi {\Psi , \nf[a] \phi }}
		\end{prooftree}
		&
		\multicolumn{2}{l}{
		\begin{prooftree}
			\hypo{}
			\infer[left label={$\neg\Phi\cup\Psi=\setof{x,\neg x}$}]1[\idLR]{\aseq[\epsilon] { \Phi } \Psi }
		\end{prooftree}
		}
	\end{tabular}
	\caption{Logical rules of \( \cmu \).}
	\label{f-ann-rules}
\end{figure}
\begin{figure}
	\centering
	\ebproofset{small}
	\begin{prooftree}
		\bud{ \aseq[c n] \Phi {\Psi , \nf[ a n ] \phi } }
		\ellipsis{}{ \aseq [c n] \Phi {\Psi , \phi (\nf[ a n ] \phi )} }
		\infer1[\anuR n]{ \aseq [c] \Phi {\Psi , \nf[a] \phi} }
	\end{prooftree}
	\qquad
	\begin{prooftree}
		\bud{ \seq[c n] {\mf[ a n ] \phi , \Phi} \Psi  }
		\ellipsis{}{ \seq [c n] {\phi (\mf[ a n ] \phi ) , \Phi } \Psi }
		\infer1[\amuL{n}]{ \seq[c] {\mf[a] \phi , \Phi} \Psi }
	\end{prooftree}
	\caption{Schematic form of bud--companion pairs (leaf and conclusion vertex respectively).} 
	\label{f-bud-schema}
\end{figure}

The motivation behind the term `cyclic proof' is that identifying buds with their companions gives rise to a finite presentation of a certain illfounded proof. Since co-proofs do not utilise the cut rule, this reduction of \( \cmu \) to \( \kmo \) applies to cut-free proofs only. 
The next theorem illustrates the construction of a co-proof from a cyclic proof.

\begin{theorem}
	\label{cmu-in-kmo}
	If \( \cmu \vdash \seq \Phi \Psi \) via a cut-free proof, then \( \kmo \vdash \seq \Phi \Psi \).
	Likewise, for systems \( \cmui \) and \( \kmoi \) respectively.
\end{theorem}
\begin{proof}
	Consider a cut-free cyclic proof \( d \) which, without loss of generality, we assume has (unique) bud--companion pairs matching the scheme of fig.~\ref{f-bud-schema}.
	Inserting a further instance of either (\nuR) or (\muL), we can view the bud--companion pair in as a literal repetition of the sequent labelling the child of the companion. 
	Identifying the new bud with its repeat induces a maximal derivation in \( \cmu \) which, if all annotations are removed, describes a maximal derivation in \( \kmo \).
	Let \( B \) be a branch of this derivation and let \( c \) be the longest annotation that is a prefix of infinitely many controls in \( B \). Necessarily, \( c \) has the form \( d n \) where \( n \) is the name associated to the bud--companion pair whose companion is closest to the root among those that appear infinitely often along \( B \). As every sequent in the suffix of \( B \) contains a formula with \( n \) in its annotation, this branch must be good.
	The same interpretation embeds cut-free \( \cmui \) in \( \kmoi \).
\end{proof}

The admissibility of cut in \( \kmoi \) is not sufficient to describe an interpretation of arbitrary \( \cmui \)-proofs in \( \kmoi \) because cuts in a cyclic proof which straddle the bud--companion pairing induce infinitely many cuts in tree unravelling.
A more general argument establishing cut \emph{elimination}, based on for instance~\cite{AfsLei24-weyl}, would be necessary for this result.

The converse to theorem~\ref{cmu-in-kmo}, that every \( \kmo \)-derivable sequent admits a cut-free proof in \( \cmu \), was claimed in \cite{AL17lics} but refuted in~\cite{DKMV23-det,Kloib23-counter-ex}.
In the presence of cut however, \( \cmu \) is a complete calculus (cf.~Theorems~\ref{cmui-in-kmui} and \ref{completeness} below).
A direct reduction of co-proofs to cyclic proofs is blocked by the `eigenvariable' condition implicit in admissible bud--companion pairs: the constraint on the `\( n \)' in fig.~\ref{f-bud-schema} requires that exactly one formula in the sequent is labelled by the name.
Simply annotating the traces in a co-proof with the goal that every branch has a finite prefix satisfying the cyclic path condition, cannot work unless augmented with a method to separate conflicting annotations into distinct branches~\cite{Kloib23-counter-ex}.
The cut rule offers this possibility but at an obvious cost.

More liberal notions of the bud--companion pair can be given for which the corresponding definition of cyclic proof is cut-free complete.
The first complete system of cyclic proofs for the modal \( \mu \)-calculus in the style of \( \cmu \) is the Jungteerapanich--Stirling `tablueax' system of~\cite{jungteerapanich:article,stirling2014tableau}, realised as a sequent calculus in~\cite{AL17lics,AL16-mfo}.
That calculus operates on a slightly different notion of annotated sequent but can be adapted to the present form without difficulty.
A version of the Jungteerapanich--Stirling calculus is utilised in the next section to establish the disjunctive normal form theorem.

An alternative cyclic proof system, closer in conception to \( \cmu \), was introduced by the authors in~\cite{AL16-pamm}.
That calculus is also cut-free complete~\cite{AL16-mfo,AL23normal} and has been used to establish Lyndon interpolation for the modal \( \mu \)-calculus~\cite{AL22lyndonInterp}.
Generic methods to derive cyclic proof systems from illfounded calculi are explored in~\cite{DKMV23-det,LeiWehr23GTCtoReset}.

% --------------
\subsection{From cyclic to finitary proofs}
\label{s-cyclic-wproof}
% --------------
Whereas the connection between cyclic and illfounded proofs is complex, the relation to finitary proofs is (relatively) simple: the  systems \( \cmu \) and \( \kmu \) prove the same (plain) sequents, and every \( \cmui \)-provable sequent is provable in \( \kmui \).
Furthermore, these result are witnessed by direct translations between cyclic and finitary proofs.
Proving these reductions will be our first task.
Some observations will be useful in this regard.
The first concerns the derivability via cyclic proofs of arbitrary identity sequents, closure under negation and the monotonicity rule. 
The argument is identical to that of Proposition~\ref{kmoi-id}, but employs annotations to discharge the buds.

\begin{proposition}
	\label{cmu-mono}\ 
	\begin{enumerate}
		\item \( \cmui \vdash \seq {\alpha} {\alpha} \) for every plain formula \( \alpha \).
		\item \( \cmu \vdash \seq \Phi \Psi  \) iff \( \cmu \vdash \lseq {\Phi , \neg \Psi } \).
		\item If \( \cmui \vdash \alpha \Rightarrow \beta \) and \( \cmui \vdash \phi x \Rightarrow \psi x \) then \( \cmui \vdash \phi\alpha \Rightarrow \psi\beta \).
	\end{enumerate}

\end{proposition}

As in the illfounded system \( \kmo \), there is no explicit induction rule in \( \cmu \).
The rule is not derivable in \( \kmo \) but certain forms admit straightforward co-proofs in the presence of cut which can be annotated to obtain cyclic proofs:

\begin{figure}
	\centering
	\begin{prooftree}[small]
		\hypo{ \seq [\epsilon] \alpha {\phi \alpha } }
		\infer1[\cw]{ \seq [n] \alpha {\phi \alpha } }
		
		\bud[*]{ \seq[n] \alpha {\nf[n] \phi} }
		\ellipsis{\footnotesize Prop.~\ref{cmu-mono}}{ \seq [n] {\phi \alpha } {\phi ( {\nf[n] \phi } )} }

		\infer2[\cut]{ \seq [n] \alpha { \phi ( {\nf[n] \phi } )} } 
		\infer1[\nuR[n]]{ \seq [\epsilon] \alpha { \nf \phi } \thickspace (*) }
	\end{prooftree}
	\qquad
	\begin{prooftree}[small]
		\bud[*]{ \seq[n] {\mf[n] \phi } {\alpha} }
		\ellipsis{\footnotesize Prop.~\ref{cmu-mono}}{ \seq [n] {\phi ( {\mf[n] \phi } )} {\phi\alpha} }

		\hypo{ \seq [\epsilon] {\phi \alpha } \alpha }
		\infer1[\cw]{ \seq [n] {\phi \alpha } \alpha }
		
		\infer2[\cut]{ \seq [n] { \phi ( {\mf[n] \phi } )  } \alpha } 
		\infer1[\nuR[n]]{ \seq [\epsilon] { \mf \phi } \alpha \thickspace (*) }
	\end{prooftree}
	\caption{The induction rules (\indR) and (\indR) with a single side formula expressed as constructive cyclic proofs. The companion of the bud marked \( * \) is the root.}
	\label{f-cmu-ind}
\end{figure}
\begin{proposition}\label{simul-ind}\ 
	\begin{enumerate}
		\item If \( \cmui \vdash \seq \alpha {\phi\alpha} \) then \( \cmui \vdash \seq \alpha { \nu \phi} \).
		\item If \( \cmui \vdash \seq {\phi\alpha}\alpha  \) then \( \cmui \vdash \seq { \mu \phi} \alpha  \).
		\item The induction rules (\indL) and (\indR) are derivable in \( \cmu \).
	\end{enumerate}
\end{proposition}
\begin{proof}
	Figure~\ref{f-cmu-ind} presents constructive derivations of the first two claims.
	The final claim is a consequence of the others. 
	For example, for \( \delta = \conj( \Phi \cup \neg \Psi ) \), if \( \cmu \vdash \seq {\Phi}{\Psi , \phi \delta } \), then \( \cmu \vdash \seq \delta {\phi \delta} \) by the previous proposition, whereby \( \cmu \vdash \seq \delta {\nf \phi} \) by 1. Since \( \cmu \vdash \seq {\Phi} {\Psi , \delta} \) we are done.
\end{proof}

With the admissibility of the induction rule, it is straightforward to show that every finitary proof induces a cyclic proof of the same sequent.
The constraints imposed on the bud--companion pairs opens the converse by allowing a syntactic translation of annotated sequents into plain sequents provable in \( \kmu \).
An obvious hurdle in the reduction is the transformation of a co-inductive conception of proof into a purely inductive form.
The cyclic structure itself encodes a `global' (co-)inductive argument only loosely connected to the predicates occuring within the proof.
It is precisely the annotations of \( \cmu \) which make the \emph{implicit} induction of cycles \emph{explicit} through the identification buds to companions and and relative order between nested bud--companion pairs. 

The embedding of cyclic into finitary proofs is detailed in \cite{AL17lics} for cut-free cyclic proofs in classical logic employing a slightly different notion of annotations and sequents.
Following~\cite{AL17lics}, the reduction proceeds by associating a syntactic translation to each vertex of a given cyclic proof that translates the annotated sequents into provable plain sequents. In most cases, the translation associated to a vertex is identical to the translation of its immediate successors.
The rules which manipulate annotations and, notably, the annotated (\muL) and (\nuR) rules provide the cases in which the translation associated to the premise and conclusion differ.

We present the embedding in three stages. First, we isolate the general structure of the syntactic embeddings and identify conditions under which the rules of \( \cmui \) are admissible in \( \kmui \).
These observations bootstrap the second step of a global embedding of constructive cyclic proofs in \( \kmui \).
Finally, we lift the embedding to the full, classical, logics.

Let \( * \colon \nms \to \form  \) be a partial function assigning names to (plain) formulas. We call such functions \emph{assignments}.
An assignment induces an interpretation of annotated formulas as plain formulas in the following way.
%For an annotation \( a \), let \( *a = \setof{ *n }[n\text{ occurs in }a] \).
Given a plain predicate \( \mnf \phi \) and annotation \( a \) we introduce a plain formula \( \mnf[a*] \phi \) defined by \( \mnf[\epsilon *] \phi = \mnf \phi \) and %recursion on the length of \( a \). Set \( \mnf[\epsilon *] \phi = \mnf \phi \) and
\[
\begin{aligned}
%\mnf[\epsilon *] \phi &\coloneqq \mnf \phi
%\\
\mnf[an*] \phi &\coloneqq
\begin{cases}
	\mnf[ a* ] \phi , &\text{if \( n \not\in \mathsf{dom} \,* \),}
	\\
	*n \circ \mnf[ a*] x.\, \phi ( *n \circ x), &\text{if \( n \in \mathsf{dom} \,* \),}
\end{cases}
%\\
%\mf[a*] \phi &\coloneqq
%	\conj \Setof{ {*}n_1 , \dotsc, * n_k , \mnf x.\, \phi \bigl( \conj \setof{ *n_1 , \dotsc, *n_k, x } \bigr)}
%	\conj \bigl( {*}a \cup \Setof{ \mnf x.\, \phi \bigl( \conj ( *a \cup \setof x ) \bigr)} \bigr)
%	\\
%\nf[a*] \phi &\coloneqq
%	\disj \Setof{ {*}n_1 , \dotsc, * n_k , \mnf x.\, \phi \bigl( \disj \setof{ *n_1 , \dotsc, *n_k, x } \bigr)}
%	\disj \bigl( {*}a \cup \Setof{ \mnf x.\, \phi \bigl( \disj ( *a \cup \setof x ) \bigr)} \bigr)
\end{aligned}
\quad \text{where } {\circ} =
\begin{cases}
	\wedge, &\text{if }\sigma = \nu, \\
	\vee, &\text{if }\sigma = \mu. 
\end{cases}
\]
Given an annotated formula \( \alpha \), the plain formula \( \alpha^* \) is defined by structural recursion compositionally with \( ( \mnf[ a ] \phi )^* = \mnf[ a* ] {\phi^*} \) and \( (\lambda x.\,\alpha)^* = \lambda  x.\, \alpha^* \).
In particular, for \( a = n_1 \dotsm n_k \),
\begin{align*}
	\phi ( { \nf[ a ] \phi }) ^* 
%	&= \phi^*\bigl( \disj \Setof{ {*}n , \dotsc, * n_k , \mnf x.\, \phi \bigl( \disj \setof{ *n_1 , \dotsc, *n_k, x } \bigr)} \bigr)
%	\\
	&= \psi(\nf \psi)
	\text{ where } \psi = \lambda x.\, \phi^* ( *n_k \vee \dotsm \vee *n_1 \vee x )
	\\
	\phi ( { \mf[ a ] \phi }) ^* 
	&= \psi( \mu \psi )
	\text{ where }
	\psi = \lambda x.\, \phi^* ( *n_k \wedge \dotsm \wedge *n_1 \wedge x )
\end{align*}
Compositionality of the  \( * \)-translation implies
\begin{lemma}
	For any assignment \( * \), the \( * \)-translation of each \( \cmui \)-rule except the annotated \( \nu \)- and \( \mu \)-rules is admissible in \( \kmui \).
\end{lemma}
The $*$-translation of the annotated quantifier rules (\nuR[n]) and (\muL[n]) is derivable in \( \kmui \) if \( n \in \mathsf{dom} \,* \) and \( *n \) carries the side formulas:
\begin{lemma}\label{cmu-in-kmu-l1}\ 
\begin{enumerate}
	\item If \( \kmui \vdash \seq \Phi {\phi( \nf[an] \phi) ^* } \) and \( *n = \conj \Phi \), then \( \kmui \vdash \seq \Phi {(\nf[a] \phi)^*}\).
	\item If \( \kmui \vdash \seq {\phi( \mf[an] \phi) ^* , \Phi } \alpha \) and \( *n = \disj (\Phi \cup \setof{\alpha} ) \), then \( \kmui \vdash \seq {(\mf[a] \phi)^* , \Phi } \alpha \).
\end{enumerate}
\end{lemma}
\begin{proof}
	We treat the second claim.
	From \( \kmui \vdash \seq {\phi( \mf[an] \phi) ^* , \Phi } \alpha \) we obtain
	\( \kmui \vdash \seq {\mf[ a* ] x.\, \phi^* ( *n \wedge x ) , \Phi } {\alpha } \) by an application of (\muL) and weakening on the left.
	An application of strong induction (Proposition~\ref{kmu-sind}) yields 
	\( \kmu \vdash \seq {\mf[ a* ] \phi^* , \Phi } {\alpha } \).
%	The following derivations witness \( \kmui \)-provability of the desired sequents.
%	\[
%%		\label{e-nu-a-in-koz}
%		%
%		\begin{prooftree}
%			\hypo{ \seq {\Phi} {\phi (\nf[ a n ] \phi ) ^*} }
%			\infer1[\nuR]{ \seq {\Phi} {\nf[ a* ] x.\, \phi^* ( *n \vee x )  } }
%			\infer1[\sindR]{ \seq {\Phi} {( \nf[ a ] \phi )^*} }
%		\end{prooftree}
%		\qquad
%		\begin{prooftree}
%			\hypo{ \seq {\phi (\mf[ a m ] \phi ) ^* , \Phi } {\alpha } }
%			\infer1[\muL]{ \seq {\mf[ a* ] x.\, \phi^* ( *m \wedge x ) , \Phi } {\alpha } }
%			\infer1[\sindL]{ \seq {\mf[ a* ] \phi^* , \Phi } {\alpha } }
%		\end{prooftree}
%	\]
\end{proof}

We require a final observation before presenting the embedding of \( \cmui \) in \( \kmui \).
Let \( \Phi \le \Psi \) express that \( \Phi = \Psi^{\upharpoonright c} \) for some \( c \in \nms^{<\omega} \).
\begin{lemma}
	\label{cmu-in-kmu-l2}
	If \( \Phi \le \Psi \) are \( \mu \)-annotated and \( \alpha \le \beta \) are \( \nu \)-annotated, then for every assignment \( * \), (plain) predicate \( \phi \) and annotation \( a \),
	\begin{enumerate}
		\item if \( \kmui \vdash \seq {\Phi^*} {\alpha^*} \), then \( \kmui \vdash \seq {\Psi^*} {\beta^*} \).
		\item if \( *n = \conj \Phi^* \) for some \( n \) in \( a \), then \( \kmui \vdash \seq{\Psi^*} {\nf[ a*] \phi} \).
		\item if \( *n = \disj ( \neg{\Phi^*} \cup \setof {\alpha^*} )  \) for some \( n \) in \( a \), then \( \kmui \vdash \seq{\mf[ a*] \phi , \Psi^*} { \beta^* } \).
	\end{enumerate}
\end{lemma}
\begin{proof}
	For the first claim, it suffices to show derivability of \( \seq {\nf \phi} {\nf x \, { \phi ( {\gamma \vee \var x} ) }} \) and \( \seq {\mf x\, \phi( \gamma \wedge x ) } {\mf \phi } \) for arbitrary \( \phi \) and \( \gamma \), which is straightforward.
%	The derivation employs the simple induction rule. For the case of \( \nu \):
%	\begin{prooftree*}
%		\axiom[$\nf \phi$]{ \seq {\nf \phi } {\nf \phi} }
%		\infer1[\disjR]{ \seq {\nf \phi } { \gamma \vee \nf \phi} }
%		\ellipsis{}{ \seq {\phi ( {\nf \phi } )} {\phi ( { \gamma \vee \nf \phi } ) }}
%		\infer1[\nuL]{ \seq {\nf \phi} {\phi ( {\gamma \vee \nf \phi } )} }
%%		\LeftLabel{\( \ind \)}
%		\infer1[\indR]{ \nu  \phi \Rightarrow \nf x. \, \phi ( \gamma \vee x ) }
%	\end{prooftree*}
	%
	For the second claim we show \( \kmui \vdash \seq {\Phi^*} {\nf[ a* ]\phi} \) given the assumption on \( * \) and invoke part 1. 
	That derivation is trivial as 
	\(
		\nf[a*] \phi = \gamma_1 \vee \dotsm \vee \gamma_k \vee \conj \Phi^* \vee \gamma 
	\)
	for some \( \gamma_1 \), \dots, \( \gamma_k \), \( \gamma \).
	The final claim is analogous.
\end{proof}

From the above, we can complete the reduction of cyclic to finitary proofs.

% environment 'restatable' allows to repeat statement.
\begin{theorem}
\label{cmui-in-kmui}%
	\( \cmui \vdash \seq \Phi \alpha \) implies \( \kmui \vdash \seq \Phi \alpha \).
\end{theorem}

\begin{proof}
	Let \( \pi \) be a constructive \( \cmu \)-proof of \( \seq \Phi \alpha \).
	Let \( \seq [c_v] {\Phi_v} {\alpha_v} \) be the annotated sequent labelling \( v \in \pi \).
	To each \( u \in \pi \) we associate an assignment \( [u] \colon \nms \to \form \).
	The domain of \( [u] \) is the set of names in \( c_u = n_0 \dotsm n_{k-1} \).
	The formula \( [u] n_{i} \) is defined by recursion on \( \lh u \) and \( i < k \) in the following way.

	Let \( c n \le c_u \) and suppose \( [u] n_i \) is defined for all \( i < \lh c \). There is a unique application of the rule (\nuR[n]) or (\muL[n]) on the path from \( u \) to the root such that \( c n \) is a prefix of the control of all sequents on the path from \( u \) to the premise of the rule.
	Let \( v \) name this vertex. 
	Suppose the rule in question is (\nuR[n]), meaning that the sequent at \( v \) has the form \( \seq[c_v] {\Phi_v} {\nf[ a n ] \phi } \) and its premise is labelled by \( \seq[c_v n] {\Phi_v} {\phi (\nf[ a n ] \phi )} \).
	As \( v < u \) we can assume that \( [v] \) is defined on the names in \( c_v \).
	In particular, \( [v] \) is defined on the names in \( c \) because \( c \) is a subword of \( c_v \). 
	We then define \( [u] n = (\conj \Phi_v)^{[v]} \).
	
	It remains to show that \( \kmu \vdash \seq { (\Phi_u)^{[u]}} {(\alpha_u)^{[u]}} \)  for every \( u \in \pi \).
	As the control of a sequent is only modified via the quantifier rules  (\cw), if no applications of either rules occurs on a path from \( u < v \) to \( v \), then the assignments \( [u] \) and \( [v] \) are identical.
	Thus, by the observations above, it suffices to consider the case that \( u \) is the conclusion of those rules or is a bud. 
	Suppose that \( u \) is the conclusion of an instance of (\nuR[n]) with \( \beta = \nf[a] \phi \) the principal formula.
	Let \( v \) be the premise and \( \Psi = \Phi_u = \Phi_v \).
	As \( \Psi \) does not contain \( n \) and \( [v] m  = [u] m \) for all \( m \neq n \),  the induction hypothesis yields
	\[ \kmui \vdash \seq {\Psi^{[v]}} {\phi(\nf[an] \phi )^{[v]}}. \]
	Furthermore, \( [v] n = \conj \Psi^{[v]} \), whereby Lemma~\ref{cmu-in-kmu-l1} implies
	\[ \kmui \vdash \seq {\Psi^{[v]}} {\beta^{[v]}}. \]
	However, this sequent is identical to the desired sequent \( \seq {\Psi^{[u]}} {\beta^{[u]}} \).
	The case of (\cw) is immediate, and that of buds is covered by Lemma~\ref{cmu-in-kmu-l2}.
\end{proof}
Starting from a cyclic proof in \( \cmu \), the same translation returns a finitary proof in \( \kmu \).
Combining the argument with Proposition~\ref{simul-ind} yields
\begin{corollary}
	\label{cmu-is-kmu}
	\( \cmu \vdash \seq \Phi \Psi \) iff \( \kmu \vdash \seq \Phi \Psi \).
\end{corollary}
%
%\begin{proof}
%	Theorem~\ref{cmu-in-kmo}.
%\end{proof}

% --------------
\subsection{From co-proofs to cyclic proofs}
\label{s-cyclic-coproof}
% --------------

The remainder of this section addresses a partial interpretation of co-proofs as cyclic proofs which, in turn, establishes a partial completeness theorem for \( \cmu \):

\begin{theorem}
	\label{quasi-completeness}%
	Suppose \( \Phi \Vdash \alpha \) and \( \Phi \subseteq \PD \).
	Then \( \cmui \vdash \Phi \Rightarrow \alpha \). 
	Moreover, the cyclic proof is cut-free.
\end{theorem}
The proof of the theorem appeals to the special properties of constructive co-proofs and \( \PD \)-formulas. By definition, every branch of a co-proof witnessing \( \Seq \Phi \alpha \) carries at most one right-thread.
If \( \Phi \subseteq \D \) comprised only disjunctive formulas, then we can assume that every branch carries at most one left-thread for each formula in the context (Proposition~\ref{kmoi-contr}).
This constraint on threads clearly fails for a context with arbitrary \( \PD \)-formulas. 
For the reduction to cyclic proofs, however, it suffices to observe that for every branch of the co-proof there is at most one \emph{good} thread per formula in the conclusion.
Proposition~\ref{SC-is-WC} guarantees this property for a context comprising \( \PD \)-formulas.

\begin{proof}
	Let \( \Phi \) be a finite sequence of \( \PD \) formulas and \( \pi \vdash \seq \Phi \alpha \) be a constructive co-proof. By Proposition~\ref{kmoi-contr}, we may assume that no contraction in \( \pi \) has a principal disjunctive formula.
	We work with an extension of \( \cmui \) by two new quantifier rules which preserve the principal annotation without extending it:
	\[
		\begin{prooftree}
		\hypo{ \seq [c] {\phi ( \mf[ a ] \phi ) , \Phi } \alpha }
		\infer1[\muhL[a]]{ \seq[c] {\mf[a] \phi , \Phi } \alpha }
		\end{prooftree}
		\qquad\quad
		\begin{prooftree}
		\hypo{ \seq [c] \Phi {\phi ( \nf[ a ] \phi )} }
		\infer1[\nuhR[a]]{ \seq[c] \Phi {\nf[a] \phi} }
		\end{prooftree}
	\]
	The above inferences are derivable in \( \cmui \) via a combination of the annotated quantifier rules and the structural rules \( (\aw) \) and \( (\cw) \).
	For the argument below it is, however, convenient to assume the above rules are part of the calculus.
	Via these rules we can consider \( \pi \) as a \( \cmu \)-derivation in which all controls and annotations are empty.
	We can, therefore, assume that \( \pi \) contains at least one branch.
	
	The first step towards a \( \cmu \)-proof is to annotate \( \pi \) by names in a way that annotations have length at most \( 1 \) and the rules (\nuR[n]), (\muL[n]) and (\cw) are applied whenever possible.
	For this purpose, we fix two disjoint infinite sets of names, referred to as the \emph{primary} and \emph{secondary} names respectively, and initially annotate \( \pi \) by primary names only, yielding a derivation \( \pi_* \) satisfying the following conditions:
	\begin{enumerate}
		\item for every instance of the rule (\nuR[n]) or (\muL[n]) the name \( n \) is primary, the principal formula is annotated as \( \nf [\epsilon] \phi \) or \( \mf[\epsilon] \phi \) and the minor formula as \( \phi ({\nf [n] \phi}) \) or \( \phi(\mf[n]\phi ) \),
		\item all other instances of the \( \nu  \)R- and \( \mu \)L-rule are \emph{preserving}, namely instances of (\nuhR[n]) or (\muhL[n]) for a primary name \( n \) already present in the control,
		\item every vertex labelled by a sequent \( \seq [c] \Phi \alpha \) for which a name occurs in \( c \) but not in \( \seq \Phi \alpha \) is the conclusion of an instance of \( (\cw) \),
		\item the result of removing all annotations from \( \pi_* \) and merging the premise and conclusion of each instance of \( (\cw) \) is the derivation \( \pi \).
	\end{enumerate}
	Since comprises only \( \PD \) formulas and has no contractions on disjunctive formulas, Proposition~\ref{SC-is-WC} implies that every primary name occurs in at most one formula of each sequent. 
	That assumption alone does not preclude that sequents may grow unboundedly in \( \pi_* \) however.
	Let \( B \) be a branch of \( \pi_* \). Since \( B \) is good, there is a primary name \( n \) which 
	occurs in all but finitely many formulas on this thread and 
	for which the rule (\nuhR[n]) or (\muhL[n]) is applied infinitely often along \( B \).
	As \( n \) can occur in an annotation of at most one (sub)formula in each sequent, the thread witnessing that \( B \) is good passes through all the principal formulas of instances of the inferences (\nuhR[n]) or (\muhL[n]) along some suffix of \( B \).
	Let \( \c B = c n \) be the longest annotation which is a prefix of all controls on some suffix of \( B \). This is well-defined as new names are only appended to the control.
	Given an annotated sequent \( \sequent \) along \( B \) and annotation \( c \), set \( \sequent^{\upharpoonright c} \) to be the annotated sequent that results after removing all names in \( \sequent \) not present in \( c \).
	
	Let \( \sequent_u = (\seq [c_u] {\Phi_u} {\alpha_u}) \) be the annotated sequent labelling \( u \in \pi \).
	We call a pair \( ( u , v ) \in B ^2 \) a \emph{repetition} if 
	\begin{itemize}
		\item \( u < v \) and there is a modal premise on the path from \( u \) to \( v \),
		\item \( \c B \) is a prefix of \( c_w \) for every \( w \) such that \( u \le w \in B \),
		\item \( \alpha_u^{\upharpoonright \c B} = \alpha_v^{\upharpoonright \c B} \) and \( \Phi_u^{\upharpoonright \c B} \) is a subsequence of \( \Phi_v^{\upharpoonright \c B} \) up to permutation,
		\item \( u \) and \( v \) are both conclusion of an instance of (\nuhR[n]) or (\muhL[n]).
	\end{itemize}
	We refer to \( \c B \) as the \emph{invariant} of the repetition \( (u,v) \), and write it as \( \c {v} \).
	Observe that if \( (u,v) \) and \( (u',v') \) are repetitions associated to branches \( B \) and \( B' \), then \( v= v' \) implies \( \c B = \c {B'} \).
	
	As \( \pi \) is a co-proof, there is an infinite subsequence \( P \colon \omega \rightarrow B \) such that for every \( i < j \), the pair \( ( P(i) , P(j) ) \) is a repetition with invariant \( \c B \).
	We thus obtain a finite \( \cmu \)-derivation \( \rho \subseteq \pi_* \) satisfying
	\begin{enumerate}
		\item \( \rho \) is an initial segment of \( \pi_* \),
		\item every bud \( l \in \rho \) is associated a vertex \( l^c < l \), called the \emph{companion} to \( l \), such that \( ( l^c , l ) \) is a repetition in \( \pi_* \),
		\item if \( l^c \le o^c < l \) for buds \( l , o \in \rho \), then \( \c {l} \le \c {o} \).
	\end{enumerate}
	Inserting instances of weakening and exchange at buds we may assume, additionally, that \( \Phi_u^{\upharpoonright c_u } = \Phi_{u^c}^{\upharpoonright c_u } \) for each bud \( u \).
	The derivation \( \rho \) can be obtained by recursion on length of invariants for paths in \( \pi_* \). König's lemma ensures that the process of pruning \( \pi_* \) terminates in a finite derivation.
	Since \( \c l \) is a prefix of every control on the path from \( l^c \) to \( l \), if \( l^c \le o^c < l \) then \( \c l \) and \( \c o \) are comparable: indeed, \( \c l \le \c o \) iff \( \lh {\c l} \le \lh { \c o } \).

	What remains is to expand certain annotations in \( \rho \) by secondary names and replace applications of the \nuhR- and \muhL-rules by appropriately chosen annotated versions to transform \( \rho \) into a \( \cmu \)-proof with no buds.
	This operation is carried out by recursion through \( \rho \).
	First, extend the definition of the invariant to non-buds: 
	For \( u \in \rho \) let \( \c u \) be the longest invariant \( \c v \) for a bud \( v \) such that \( v^c < u \le v \).
	In particular, \( \c u \) is the empty annotation if all leaves above \( u \) are axiomatic, and \( \c {v^c} = \c v \) iff there is a companion \( u^c < v^c < u \) with \( \c u = \c v \).
	
	We define a \( \cmu \)-derivation \( \rho_u \) for each \( u \in \rho \). A bud \( v > u \) is said to be \emph{open} if \( v^c < u \).
	The buds of \( \rho_u \) are the sequents labelling the buds of \( \rho \) that are open above \( u \), with each restricted to the invariant; the endsequent of \( \rho_u \) is the sequent at \( u \) similarly restricted:
	\begin{prooftree*}
		\hypo{\Setof{ \seq [\c v] {\Phi_v^{\upharpoonright \c v}} {\alpha_v^{\upharpoonright \c v}} }[ v^c < u \le v ]}
		\ellipsis{$\rho_u$}{ \seq[\c u] {\Phi_u^{\upharpoonright \c u} } {\alpha_u^{\upharpoonright \c u}} }
	\end{prooftree*}
	For \( u \in \rho \) a bud \( \rho_u \) is simply the trivial derivation \( \seq[\c u] {\Phi_u^{\upharpoonright \c u} } {\alpha_u^{\upharpoonright \c u}} \).
	Suppose \( \rho_v \) is defined for each \( v \in \Child {\rho} u \).
	We define \( \rho_u \) by case distinction on the inference with conclusion \( u \). In the case of unary rules, \( u^+ \) names the unique child of \( u \) (in \( \rho \)).
\begin{pcases}
	\item [Case (\id).] In this case \( \Phi_u\cup \setof{\neg{\alpha_u}} \) contains an inconsistent pair of variables. As \( \c u = \epsilon \) is the empty annotation, \( \rho_u \) is the \( \cmu \)-axiom 
	\[
		\begin{prooftree}
		\axiom[\idLR]{ \seq[\epsilon] {\Phi_u^{\epsilon} } {\alpha_u^{\epsilon}}}
		\end{prooftree}
	\]
	\item [Case (\( \boldsymbol{\conjL} \)).] We have \( c_{u^+} = c_u \), \( \alpha_{u^+} = \alpha_u \), \( \Phi_{u^+} = {\Gamma , \Psi } \) for some \( \Psi \) and \( \Phi_u = \conj \Gamma , \Psi \). In particular, \( \c u = \c {u^+} \). Define \( \rho_u \) as
		\begin{prooftree*}
			\subproof*[\rho_{u^+}]{ \seq[\c u] {\Gamma^{\upharpoonright \c u} , \Psi^{\upharpoonright \c u} } {\alpha_u^{\upharpoonright \c u}} }
			\infer1[\disjL]{ \seq[\c u] {\Phi_u^{\upharpoonright \c u} } {\alpha_u^{\upharpoonright \c u}} }
		\end{prooftree*}
	\item [Case (\( \boldsymbol{\conjR} \)) or (\( \boldsymbol{\disjL} \)).] For each \( v \in \Child{\rho} u \) we have a derivation \( \rho_v \) of the sequent \( \lambda _\rho(v) \) restricted to \( \c v \). Form \( \rho_u \) by an instance of (\conjR) or (\disjL) with same principal formula and applications of (\aw) and (\cw) above the premises, as needed.
	
	\item [Case (\( \boldsymbol{ \nuR[n]} \)) or (\( \boldsymbol{\muL[n]} \)).] As \( n \) is a primary name, \( u \) is not a companion and \( \c u = \c {u^+} \). \( \rho_u \) is formed by attaching an instance of (\nuR[n]) or (\muL[n]) at the root of \( \rho_{u^+} \) with no associated bud.
	
	\item [Case (\( \boldsymbol{\nuhR[a] } \)) or (\( \boldsymbol{\muhL[a] } \)).] Without loss of generality, we treat the rule (\muhL). There are three subcases to contend:
	\begin{enumerate}
		\item \( u \) is not the companion of a bud. In this case, \( \c u = \c {u^+} \). \( \rho_u \) concludes \( \rho_{u^+} \) by an application of (\muL) as appropriate.
		
		\item \( u = v^c \) is the companion of a bud \( v \) in \( \rho \) and \( \c u < \c {u^+} \). In particular, \( \c {u^+} = \c u d n \) for some \( d \) and \( \Phi_u = \nf [n]\phi , \Psi \).
	Let \( \mathscr O \) be buds of \( \rho \) which are open above \( u \).
	By construction, \( \rho_{u^+} \) is a proof in \( \cmui \) of the sequent 
	\[ 
		\seq [\c u d n] 
			{\phi(\mf[n] \phi)^{\upharpoonright \c u d n } , \Psi^{\upharpoonright \c u d }} 
			{\alpha_u^{\upharpoonright \c u d }}  
	\]
	from the assumption \( \seq [\c ud n] {\nf \phi^{\upharpoonright \c u d n } , \Psi^{\upharpoonright \c u dn }} {\alpha_u^{\upharpoonright \c u dn}}  \) and the sequents labelling the buds in \( \mathscr O \):
	\[
		\rho_{u^+} =\left\{ \begin{prooftree}
			\hypo{ \seq[\c u d n] {( \mf[n] \phi )^{\upharpoonright \c u d n} , \Phi^{\upharpoonright \c u d } } {\alpha_u^{\upharpoonright \c u d}}}
			\hypo{\setof{(\sequent_o)^{\upharpoonright \c o}}[o \in \mathscr O]}
			\infer[dashed]2{ }
			\ellipsis{}{ \seq [\c u d n] {\phi(\mf[n] \phi)^{\upharpoonright \c u d n } , \Psi^{\upharpoonright \c u d }} {\alpha_u^{\upharpoonright \c u d}}  }
		\end{prooftree}
		\right.
	\]
	The derivation \( \rho_u \) is formed by removing the names in \( d \) from \( \rho_{u^+} \) before inserting an instance of (\nuR[n]) at the root:
	\[
		\rho_u =\left\{ \begin{prooftree}
			\hypo{ \seq[\c u n] {( \mf[n] \phi )^{\upharpoonright \c u n} , \Phi^{\upharpoonright \c u }} {\alpha_u^{\upharpoonright \c u }} }
			\hypo{ \Setof{ ( \sequent _o )^{\upharpoonright \c o} }[ o \in \mathscr O ]}
			\infer[dashed]2{ }
			\ellipsis{$(\rho_{u^+})^{\upharpoonright \c u n}$}{ \seq [\c u n] {\phi(\mf[n] \phi)^{\upharpoonright \c u n } , \Psi^{\upharpoonright \c u }} {\alpha_u^{\upharpoonright \c u}}  }
			\infer1[\muL[n]]{ \seq [\c u] {\mf \phi^{\upharpoonright \c u } , \Psi^{\upharpoonright \c u }} {\alpha_u^{\upharpoonright \c u}}  }
		\end{prooftree}
		\right.
	\]
		
		\item \( u = v^c \) is the companion of a bud in \( \rho \) and \( \c u = \c {u^+} \). In particular, \( \c u = \c v \). 
	Let \( \c u = c n \) and \( \Phi_u = \nf [n]\phi , \Psi \).
	In this case there are no names in the control that need removal but it is necessary to maintain the name \( n \) in the conclusion of \( \rho_u \). This is achieved by introducing a secondary name \( m \) that does not occur in \( \rho_u \) and first forming the derivation in which all occurrences of the name \( n \) are replaced by \( n m \) (working from the root and stopping once a vertex is reached which does not contain \( n \)). 
	Let \( c^{(m)} \) express the result of effecting this change to \( c \); likewise \( \alpha^{(m)} \) and \( \Phi^{(m)} \).
	Let \( \mathscr O \) be the open buds of \( \rho \) above \( u \). 
	The derivation \( \rho_u \) is formed by using \( m \) as the principal name for discharging the associated bud:
	\[
		\rho_u =\left\{ \begin{prooftree}
			\hypo{ \seq [\c {u} m] { (\mf[n m] \phi )^{\upharpoonright \c u m } , \Phi^{\upharpoonright \c u } } {\alpha^{\upharpoonright \c u}} }
			\hypo{ ( \sequent _o)^{\upharpoonright \c o } }
			\infer1[$\aw+\cw$]{ (\sequent _o)^{(m)}}
			\delims{\left\{}{\;\middle|\; o \in \mathscr O \right\}}
			\infer[dashed]2{ &\qquad\qquad\qquad\qquad\qquad\qquad\qquad\qquad}
			\ellipsis{}{ \seq [\c {u} m] { \phi(\mf[n m] \phi )^{\upharpoonright \c u m } , \Phi^{\upharpoonright \c u } } {\alpha^{\upharpoonright \c u}} }
			\infer1[\muL[m]]{ \seq [\c {u}] { (\mf[n] \phi )^{\upharpoonright \c u } , \Phi^{\upharpoonright \c u } } {\alpha^{\upharpoonright \c u}} }
		\end{prooftree}
		\right.
	\]
	\end{enumerate}
	\item [Other cases.] The other rules are straightforward.
\end{pcases}
	By construction, \( \rho_\epsilon \) is a \( \cmu \)-derivation with endsequent \( \seq [\epsilon] \Phi \alpha \) and no buds, that is, \( \rho_\epsilon \) is a \( \cmui \)-proof.
\end{proof}

Combining theorem~\ref{quasi-completeness} with our partial completeness result for \( \kmoi \), yields:
\begin{corollary}
	\label{cmu-comp-sc}
	If \( \alpha \) is valid and conjunctive, then \( \cmui \vdash \rseq \alpha \).
\end{corollary}
%
%\begin{proof}
%	Theorems~\ref{kmoi-comp-conj} and \ref{quasi-completeness}.
%\end{proof}

% --------------
\section{Disjunctive Normal Form}
\label{s-DNF}
% --------------

We now tend to the task of showing that the disjunctive fragment of \( \form \) is expressively adequate for constructive modal \( \mu \)-calculus, namely that every formula can be associated an equivalent disjunctive formula and the equivalence holds in \( \kmoi \).
As already mentioned, the proof of this result uses a proof-search procedure for annotated sequents.
Due to the restrictive form of proofs in \( \cmu \), rarely can we expect the proof-search to actually isolate a cyclic proof. 
Rather, annotations assist in bounding the search space relative to potential \emph{illfounded} proofs.
\begin{theorem}[Disjunctive normal form theorem]
	\label{DNF-thm}
	For every formula \( \alpha \) there exists a guarded disjunctive formula \( \gamma \VdashV \alpha \). In particular, \( \denote{\gamma} = \denote{\alpha} \).
\end{theorem}
The associated disjunctive formula \( \gamma \) is designed to encode all information about how the formula \( \alpha \) can be utilised in a co-proof of a sequent \( \seq {\alpha } \beta \) independent of the choice of \( \beta \).
The construction of \( \gamma \) proceeds in two steps.
First, the role of \( \alpha \) in unspecified co-proof is expressed as a canonically chosen derivation in a modification of the cyclic calculus \( \cmu \) that merges the proof-search calculus of Theorem~\ref{kmoi-comp-conj} and the annotation method of Theorem~\ref{quasi-completeness}.
This derivation will be called the \emph{template} for \( \alpha \).
Second, we show how the template can be encoded as a (disjunctive) formula of \( \mu  \)-calculus which is provably equivalent to the initial formula \( \alpha \).
Essentially the same two-step construction is employed in~\cite{ALM21uniformInter} for establishing uniform Lyndon interpolants for the modal \( \mu  \)-calculus.
Indeed, we are able to recapture that result as a corollary of the construction.
%From the perspective of completeness, however, the additional complexity arising from adding a choice of vocabulary as a parameter serves only to obfuscate the construction, whence this task is left to the interested reader.

The present section is dedicated to the proof of theorem~\ref{DNF-thm}.
With the theorem in place we obtain as a corollary completeness of the illfounded proof system:
\begin{corollary}
	\label{kmo-complete}
	If \( \alpha \) is valid then \( \neg{\alpha} \Vdash \bot \) and \( \kmo \vdash {}\Rightarrow \alpha \).
\end{corollary}
\begin{proof}
	Let \( \alpha \) be valid and let \( \gamma \) be a disjunctive formula \( \gamma \VdashV \neg\alpha \), entailed by theorem~\ref{DNF-thm}. 
	Since \( \denote{\alpha} \subseteq \denote{\neg\gamma} \) (Theorem~\ref{kmo-snd}), \( \neg\gamma \) is a valid conjunctive formula, whence \( \gamma \Vdash \bot \) by Theorem~\ref{kmoi-comp-conj}.
	Transitivity (Theorem~\ref{admiss-charac}) implies \( \neg \alpha \Vdash \bot \) and, therefore, \( \kmo \vdash {}\Rightarrow \alpha \).
\end{proof}

The proof of theorem~\ref{DNF-thm} is split into three parts.
First, we present the construction of the disjunctive formula.
Following that, we establish each of the two provability claims.

% --------------
\subsection{Construction}
% --------------
Without loss of generality (Theorem~\ref{guard-equi}) we may assume that \( \alpha \) is guarded.
The construction utilises the notion of annotated formulas and sequents from the previous section. To that base, we add a few more concepts.

\begin{definition}
	A name \( n \) is \emph{covered} in an annotated sequent \( \seq[c] \Phi \Psi \) if it occurs in \( \Phi \) or \( \Psi \) and for every \( \sigma^a \phi \in \subform{\Phi,\Psi} \), if \( a_i = n \) for some \( i \) then \( \lh a > i + 1 \).
\end{definition}

Fix \( c \in \nms^{<\omega} \).
The subsequences of \( c \) form a finite tree \( S(c) \) relative to the  prefix relation. 
This tree is equipped with an ordering on non-comparable vertices, defined as \( a \) is \emph{left of} \( b \) if there exists \( k < \lh a \) such that
\begin{enumerate}
	\item \( a(i) = b(i) \) for all \( i < k \), and
	\item \( a(k) = c(k_0) \) and \( b(k) = c(k_1)\) for some \( k_0 < k_1 < \lh c \).
\end{enumerate}
The \emph{Kleene--Brouwer ordering} on \( S(c) \) is the relation \( \le_c \) defined as \( a \le_c b \) iff \( b \) is a prefix of \( a \), or \( a \) is to the left of \( b \).
Observe that \( \epsilon \) is the maximal element of \( \le_c \) for every \( c \in \nms^{<\omega} \).
\begin{proposition}\label{KB-wo}
	If \( c \) is an annotation then \( \le_c \) well-orders the subsequences of \( c \).
\end{proposition}
\begin{proposition}
	\label{KB-mono}
	If \( c \le d \) then \( {\le_c} \subseteq {\le_d} \).
\end{proposition}

The orderings \( \le_c \) will be shortly lifted to an ordering on annotated formulas that will be used to constrain the size of sequents in the template.
For this purpose it is convenient to focus attention on the particular annotated formulas that can occur in derivations. 
These are formulas for which the annotated subformulas are linearly ordered by substitution. We call such formulas \emph{linearly annotated}:
\begin{definition}\label{d-lin-ann}
	An annotated formula \( \alpha \) is {linearly \( \sigma \)-annotated} if there exist a plain formula \( \beta \) and predicates \( \phi_1 \), \dots, \( \phi_n \) and non-trivial annotations \( a_1 , \dotsc, a_n \) such that
	\begin{enumerate}
		\item \( \alpha = (\lambda x_n \dotsm ( \lambda x_1\,\beta) ({\mnf[a_1] \phi_1}) \dotsm ) ({\mnf[a_n] \phi_n}) = \beta \subs[x_1]{\mnf[a_1] \phi_1} \dotsm \subs[x_n]{\mnf[a_n] \phi_n}\),
		\item \( x_1 \) is free in \( \beta \) and \( x_{i+1} \) is free in \( \sigma  \phi_i \) for \( i \in [ 1, n] \),
		\item a name occurs in both \( a_i \) and \( a_j \) iff \( i = j \).
	\end{enumerate}
	Given a plain formula \( \alpha \) and annotation \( c \), let \( L_\sigma(\alpha,c) \) be the set of linear \( \sigma \)-annotations of \( \alpha \) in which all annotations are subsequences of \( c \).
\end{definition}

The second condition of the definition implies that a decomposition of \( \alpha \) in the form of 1 is unique, if it exists. 
The third condition states that each name occurs in a unique subformula.
This assumption assists with keeping track of names during the construction of the template.
By the first two conditions, it follows that \( \neg{x_i} \) does not occur in \( \beta , \phi_1 , \dotsc, \phi_{i-1} \).

Given linearly \( \sigma \)-annotated formulas \( \alpha \) and \( \beta \) with respective decompositions 
\begin{align*} 
	\alpha &= \alpha_0 \subs[x_1]{\mnf[a_1] \phi_1} \dotsm \subs[x_n]{\mnf[a_n] \phi_n} 
	\\
	\beta &= \beta_0 \subs[x_1]{\mnf[b_1] \psi_1} \dotsm \subs[x_m]{\mnf[b_m] \psi_m}
\end{align*} 
we set \( \alpha <_c \beta \) if 
\begin{enumerate}
	\item \( \alpha_0 = \beta_0 \), \( m = n \) and \( \phi_i = \psi_i \) for each \( i \in [1,m] \), and
	\item there exists \( k \in [1,m] \) such that \( a_k <_c b_k \) and \( a_i = b_i \) for all \( i \in [ k+1, m] \).
\end{enumerate}
The weak order is defined as \( \alpha \le_c \beta \) iff \( \alpha <_c \beta \) or \( \alpha = \beta \).
In particular, if \( \alpha \le_c \beta \) then \( \alpha^- = \beta^- \).

\begin{example}
	\( \nf[a] x( \diamond x \vee \nf[b] y.\, \square  y ) \) is linearly \( \nu \)-annotated for all \( a,b \in \ann \).
	The formula \( \mf[a] x( \diamond x \vee \mf[b] y.\, \square ( x \wedge y ) ) \) is linearly \( \mu \)-annotated iff \( b = \epsilon \)
	because this is the only case in which the formula can be expressed as a linear decomposition following Definition~\ref{d-lin-ann}.
	Similarly, the formula \( \conj \setof{ \nf[a] x.\, \square x , \nf[b] y \diamond y } \) is linearly annotated iff at least one of \( a \) and \( b \) is trivial.
	
	The ordering \( <_c \) compares linearly ordered formulas with the comparable decompositions by a lexicographic ordering on the annotations.
	Supposing that \( \phi , \psi \) are plain predicates and \( \phi \) contains \( x \) free, \( (\nf[a] \phi) \subs {\nf[b] \psi} <_c (\nf[a'] \phi) \subs {\nf[b'] \psi}  \) iff \( b <_c b' \) or \( b = b' \) and \( a <_c a' \).
\end{example}

As \( \le_c \) well-orders the subsequences of \( c \) its lifting to \( L_\sigma(\alpha,c) \) is both a partial order and a well-quasi order.
Indeed, a bound on the number of pairwise \( \le_c \)-incomparable elements set in \( L_\sigma(\alpha,c) \) depends only on \( \alpha \) (and not \( c \)).
\begin{proposition}\label{anno-wo}
	Let \( \alpha \) be a plain formula.
	There exists \( N < \omega \) such that for 
	all \( c \in \ann \) every set of \( \le_c \)-incomparable elements of \( L(\alpha,c) \) has cardinality bounded by \( N \).
\end{proposition}

We now proceed with the definition of the template.
The template of \( \alpha \) is a certain finite derivation in a variant of \( \cmu \) adjusted for proof-search of the sequent \( \lseq \alpha {} \).
As the calculus employed is cut-free, all formulas will be linear \( \mu \)-annotations of subformulas of \( \alpha \).
The additional rules of the calculus are presented in fig.~\ref{f-template-rules}.
We use \( (\cover) \) to denote the union of \( (\cover^n) \) for all \( n \in \nms \).
The structural rules \( (\dup) \) and \( (\cover) \) provide special instances of formula weakening and annotation weakening respectively that can be permitted in the context of proof-search.

\begin{figure}
	\centering
	\begin{prooftree}
		\hypo{ \tseq c {\Phi , \alpha} {} }
		\hypo{ \alpha \le_c \alpha'}
		\infer2[$\dup$]{ \tseq c {\Phi , \alpha , \alpha'}{} }
	\end{prooftree}
	\qquad
	\begin{prooftree}
		\hypo{ \tseq c {\Phi^{\downarrow n}}{} \vphantom\phi }
		\hypo{ n \text{ \small covered in } \Phi }
		\infer2[$\cover^n$]{ \tseq c \Phi {} }
	\end{prooftree}
	\\[2ex]
	\begin{prooftree}
		\hypo{\aseq {\Gamma , \Phi}\Psi }
		\infer1[\conjsL]{\aseq {\conj\Gamma, \Phi} \Psi }	
	\end{prooftree}
	\qquad
	\begin{prooftree}
		\hypo{ \setof{ \tseq c {\alpha , \Phi } }[ \alpha \in \Psi]  }
		\hypo{ \tseq c \Phi {}  }
		\infer[left label={$L \subseteq \Lit $}]2[$ \diLs $]{ \tseq c {\diamond \Psi , \square \Phi , L } }
	\end{prooftree}
	\caption{Structural and logical rules specific to the template. \( \Phi^{\downarrow n} \) is the restriction of each annotation in \( \Phi \) which contains \( n \) to the shortest prefix containing \( n \).}
	\label{f-template-rules}
\end{figure}
\begin{definition}\label{d-pre-template}
	Fix an enumeration of \( \subform{\alpha} \) and an enumeration \( n_0 , n_1 , \dotsc \) of \( \nms \).
	The \emph{pre-template} of \( \alpha \) is the maximal cut-free derivation with root \( \epsilon : \alpha \) in the rules of \( \cmu \) where the rules (\aw), (\cut), (\diL) and (\conjL) are replaced by the four rules in fig.~\ref{f-template-rules} which satisfies the following for every vertex \( u \):
	\begin{enumerate}
		\item If the sequent at \( u \) is the conclusion of an instance of (\cw), (\dup) or (\cover) then one of these rules labels \( u \) with that priority (so the rule is (\cw) if possible, otherwise (\dup), or (\cover) if neither other rule is applicable).
%		\item If a name is covered in the sequent at \( u \), then the rule at \( u \) is either \( (\cover) \) or \( (\dup) \);
		\item If the rule at \( u \) is a non-modal logical rule, then the principal formula is the first formula in the sequent according to the enumeration of \( \subform{\alpha} \).
		\item If the rule at \( u \) is (\muL[n]) then all names of lower index occur in the sequent at \( u \), i.e., \( n = n_i \) and \( n_j \) occurs in \( u \) for all \( j < i \).
	\end{enumerate}
\end{definition}

Existence of a pre-template is immediate, and the three constraints of Definition~\ref{d-pre-template} ensure its uniqueness.
As the pre-template is required to be maximal, leaves are instances of either (\idL) or (\disjL) with principal formula \( \disj \emptyset \).
All other sequents are conclusions of a rule in the calculus.
Also of note is that a sequent of the form \( \aseq[\epsilon] \emptyset {} \) may occur in the pre-template (possibly obtained as a premise of (\diLs) where \( \Phi = \emptyset \)) in which case this vertex lies on a single branch whose suffix is the sequence \( ( \aseq[\epsilon] \emptyset {} , (\diLs))^\omega \).
The next observation makes use of the fact that only linearly annotated formulas occur in a pre-template.
\begin{proposition}\label{pre-temp-reg}
	The pre-template is a regular tree. That is, there are only finitely many maximal subtrees of the pre-template of a given formula \( \alpha \).
\end{proposition}
\begin{proof}
	Let \( P \) be the pre-template of \( \alpha \). It suffices to show that every branch of \( P \) contains a repeated annotated sequent as regularity then follows from the uniqueness of the pre-template.
	By induction through \( P \) we see that every formula is linearly \( \mu \)-annotated (note that \( \phi(\mf[a] \phi) \) is linearly annotated if \( \mf[a] \phi \) is).
	Proposition~\ref{anno-wo} and the priority in applications of (\dup) in \( P \) ensure that each annotated sequent in the branch has bounded length. 
	Coupled with the priority on applications of (\cover) there is, therefore, a bound on the length of all controls, whence a bound on the number of distinct annotated sequents in the branch. Hence, every branch contains a repeated annotated sequent.
\end{proof}
A consequence of regularity is that the pre-template for \( \alpha \) can be identified with a finite initial segment of itself.
For our purposes it is convenient to isolate a particular initial segment of \( P \), henceforth called the \emph{template}, in which information about the eliminated parts of the tree is directly expressed in the buds.
Identifying the template is similar to isolating the basis of the cyclic proof in Theorem~\ref{quasi-completeness}.

A \emph{repetition} is a pair \( (u ,v) \in P \times P \) of vertices labelled by the same annotated sequent such that \( u < v \).
As \( \alpha \) is guarded, if \( ( u , v ) \) is a repetition then there is a vertex \( w \in [ u , v) \) such that \( \rules_P (w) = (\diLs ) \). 
As observed, every branch of the pre-template contains infinitely many repetitions.
For the present construction, the \emph{invariant} of a repetition \( ( u , v ) \) will be a pair \( ( c , k ) \in \ann \times \setof{ 0,1 } \) where \( c \) is a non-trivial prefix of all controls on the path interval \( (u,v) \) and \( k \) records whether the final name in \( c \) is covered in one of these sequents (in which case \( k = 0 \)).
If there is no such control, then \( c \) is the longest prefix of all controls and \( k = 1 \).
Thus, the invariant of a repetition \( (u , v) \) is the pair \( \inv {u,v} = ( c , k ) \in \ann \times \setof{ 0,1 } \) where:
\begin{itemize}
	\item \( c \) is the longest annotation that is a prefix of every vertex in the path segment \( [ u , v ] \) and is such that if \( \rules_P(w) = (\cover^{c(i)}) \)  for some \( w \in [u , v ) \) then \( i = \lh c - 1 \),
	\item \( k = 0 \) iff \( \rules_P(w) = (\cover^{c(i)}) \) for some \( w \in [ u , v ) \) and \( i < \lh{c} \).
\end{itemize}

Invariants are compared lexicographically:
\( ( c , k ) \sqsubseteq ( d , l ) \) iff \( c < d \) or \( c = d \) and \( k \le l \).
In particular, \( ( c , 1 ) \sqsubseteq ( d , 0 ) \) iff \( c \) is a proper prefix of \( d \).
The relation \( \sqsubseteq \) forms a partial order on invariants.

Assume a companion function on \( P \) is available, i.e., a partial function \( u \mapsto u^c \) such that if \( u^c \) is defined then \( ( u^c , u ) \) is a repetition.
If \( u^c \) is defined we write \( \inv u \) for \( \inv {u^c , u} \).
The companion function induces a more liberal `reachability' relation \( \rightsquigarrow \) on \( P \) by \( u \rightsquigarrow v \) iff \( u^c \) and \( v^c \) are both defined and \( u^c \le v^c < u \).
Note, \( \rightsquigarrow \) is not transitive in general.
However, if \( u \rightsquigarrow v \) then \( \inv u \) and \( \inv v \) are comparable: \( \inv { u } \sqsubseteq \inv { v } \) or \( \inv { v } \sqsubseteq \inv { u } \).
The template is given by a choice of companion function such that the two orderings are in agreement:

\begin{definition}%[template]
	A \emph{template} of \( \alpha \) is a finite initial subderivation \( T \subseteq P \) and a companion function \( u \mapsto u^c \) defined on the buds of \( T \) such that if \( u \rightsquigarrow v \) for buds \( u , v \) then \( \inv u \sqsubseteq \inv v \).
\end{definition}
Existence of a template \( T \) of \( \alpha \) is straightforward and can be obtained by isolating the buds of \( T \) by recursion through the finitely many invariants in \( P \).
Any template implicitly witnesses regularity of the pre-template in the sense that for every bud \( u \in T \) the maximal sub-derivations of \( P \) rooted at \( u \) and \( u^c \) are identical.

In the following we assume a fixed template of \( \alpha \), denoted \( T \).
A bud \( u \) is \emph{successful} if its invariant \( \inv u = ( c , 0 ) \) for some \( c \); otherwise \( u \) is \emph{unsuccessful}.
If \( u \) and \( v \) are buds with the same companion, then \( \inv u = \inv v \), so \( u \) is successful iff \( v \) is successful.
Thus we can unambiguously refer to a companion \( u^c \) as \emph{(un)successful} if \( u \) is (un)successful.

We now proceed with the definition of the disjjunctive formula.
As discussed, this formula simply encodes the pre-template of \( \alpha \) via the chosen template.
We assume a distinct variable \( x_u \) is available for each vertex \( u \in T \), let \( \rules_P(u) \) be the rule labelling \( u \) and let \( \tseq {c_u} {\Phi_u} \) be the sequent labelling \( u \in T \). 
The construction of \( \hat\alpha \) proceeds by associating to each \( u \in T \) disjunctive formulas \( \alpha_u^+ \) and \( \alpha_u^- \) in the following way.
An example of the association is given in fig.~\ref{f-template-ex}.

\begin{figure}
	\centering
	\begin{prooftree}[center=false]
	
		\hypo{ \tseq n {\square \beta , \diamond  \mf[n] \phi} }
		\infer1[$\dup$]{ \tseq n {\square \beta , \diamond\beta , \diamond  \mf[n] \phi } }
		\infer1[$\conj$]{ \tseq n {\square\beta , \phi( \mf[n] \phi )} }
		\infer1[$\cover$]{ \tseq {n m} {\square\beta , \phi( \mf[nm] \phi ) } }
		\infer1[$ \mu ^m$]{ \tseq n {\square \beta , \mf[n] \phi} }
		\infer1[$\nu $]{ \tseq n {\beta , \mf[n] \phi} }
		
		\hypo{ \tseq \epsilon \beta }
			\hypo{ \tseq \epsilon \emptyset }
			\infer1[$\square ^*$]{ \tseq \epsilon \emptyset \vphantom{\psi} }
		\infer[separation=2ex]2[$\square ^*$]{ \tseq \epsilon {\square  \beta \vphantom{\psi}}}
		\infer1[$\nu $]{ \tseq \epsilon {\beta \vphantom{\psi}}}
		\infer1[$\cw$]{ \tseq n {\beta \vphantom{\psi}} }
			
		\infer2[$\square ^*$]{ \tseq n {\square \beta , \diamond  \mf[n] \phi \vphantom{(\psi)}} }
		\infer1[$\conj$]{ \tseq n {\square \beta \wedge \diamond  \mf[n] \phi \vphantom{(\psi)}} }
		\infer1[$\nu ^n$]{ \tseq \epsilon {\mu  \phi }\vphantom\psi }
	\end{prooftree}
	\qquad
	\begin{prooftree}[center=false]
		\hypo  { z \vphantom{\nf[n]\phi} }
		\infer1{ z \vphantom{\nf[n]\phi} }
		\infer1{ z \vphantom{\nf[n]\phi} }
		\infer1{ z \vphantom{\nf[n]\phi} }
		\infer1{ z \vphantom{\nf[n]\phi} }
		\infer1{ z \vphantom{\nf[n]\phi} }
		
			\hypo{ y }
			\hypo   { x }
			\infer1 { \nabla(\emptyset,\setof x) }
			\infer[separation=2em]2
				   { \nabla( \setof y , \setof{ y , \gamma }) }
			\infer1{ \nu  \psi }
			\infer1{ \nu  \psi  }
			
		\infer[separation=3em]2{ \mu  z\, \nabla( \setof z , \setof{z , \nu  \psi} ) }
		\infer1{ \mu  z\, \nabla( \setof z , \setof{z , \nu  \psi} ) }
		\infer1{ \mu  z\, \nabla( \setof z , \setof{z , \nu  \psi} ) }
	\end{prooftree}

	\caption{An example of a template (left) for the formula \( \mu  \phi \) where \( \phi = \lambda  x.\, \square(\nu  x\, \diamond  x) \wedge \diamond  x \) and \( \beta = \nu  x\, \diamond  x \). On the right is the corresponding disjunctive formula associated to each vertex of the template, with \( \gamma \) and \( \psi \) abbreviating \( \mu x\, \nabla(\emptyset,\setof x) \) and \( \lambda  y.\, \nabla( \setof y , \setof{ y , \gamma }) \) respectively.}
	\label{f-template-ex}
\end{figure}

\begin{definition}\label{d-conj-comp-prelim}
The formulas \( \alpha_u^+ \) and \( \alpha_u^- \) are defined by recursion through \( u \in T \).
Assuming that \( \alpha_u^- \) is defined, set
\[ \alpha_u^+ \coloneqq 
\begin{cases}
	\mf x_u \, \alpha_u^-, &\text{if \( u \) is successful,}
	\\
	\nf x_u \, \alpha_u^-, &\text{otherwise.}
\end{cases} \]
The formula \( \alpha_u^- \) is built from the formulas \( \setof{\alpha_v^+}[v \in \Suc [T] u] \) according to the rule at \( u \in T \):
%Assume that \( \hat\alpha_v \) has been defined for each \( v \in \Suc[T]u \).
\begin{itemize}
	\item If \( u \) is a bud,  \( \alpha_u^- = x_{u^c} \). 
	\item If \( \rules_P(u) = (\disjL) \), then \( \alpha_u^- = \disj \Setof{ \alpha_v^+ }[v\in \Suc[t]{u}] \).
	\item If \( \rules_P(u) = (\diLs) \) and \( \Phi_u \) contains both \( x , \neg x \) for some variable \( x \) then \( \alpha_u^- = \bot \).
	\item If \( \rules_P(u) = (\diLs) \) and \( \Phi_u \) is not of the form above. % then \( \alpha_u = \nabla_L(\Gamma , \Delta ) \) where \( \Gamma \), \( \Delta \) and \( L \) are given as follows.
	Let \( \Phi_u = \diamond \delta_0 , \dotsc, \diamond \delta_{n-1} , \square\Psi , L \) where \( L \subseteq \Lit \) and let \( \Suc[T] u = \setof{ u_1 , \dotsc , u_{n} } \) be such that \( \Phi_{u_i} = \delta_i , \Psi \) for each \( i < n \) and \( \Phi_{u_{n}} = \Psi \).
	Set 
	\[ 
		\alpha_u^- = \nabla_L( \setof{ \alpha_{u_{0}}^+ , \dotsc , \alpha_{u_{n-1}}^+ } , \setof{ \alpha_{u_{0}}^+ , \dotsc , \alpha_{u_n}^+ } ) . 
	\]
	\item In all other cases \( u \) has exactly one child \( u^+ \), and \( \alpha_u^- = \alpha_v^+ \).
\end{itemize}
\end{definition}
\begin{definition}%[conjunctive companion]
	The \emph{disjunctive companion} to a guarded formula \( \alpha \) is the formula \( \hat\alpha \coloneqq \alpha_\epsilon^+ \) associated to the root of the template as per the previous definition. 
	More generally, the disjunctive companion to a vertex \( u \in T \) is the formula
	\[ \hat\alpha_u \coloneqq \alpha_u^+ \subs[x_{u_n}]{ \alpha_{u_n}^+ } \dotsm \subs[x_{u_1}]{ \alpha_{u_1}^+ } 
	\]
	where \( \epsilon = u_1 < \dotsm < u_n < u \) enumerate the ancestors of \( u \).
	The disjunctive companion to a non-guarded formula is the disjunctive companion to the canonical guarded equivalent formula determined by proposition~\ref{guard-equi}.
\end{definition}

We begin with a few useful observations.
\begin{proposition} For \( \hat\alpha_u \) defined as above,
	\begin{enumerate}
		\item \( \hat\alpha_u \) is a guarded disjunctive formula.
		\item If \( u \in T \) is axiomatic then \( \kmm \vdash \hat\alpha_u \).
		\item If \( u \rightsquigarrow v \) then \( \hat\alpha_u \sft \hat\alpha_v \) for a thread in which \( {\hat\alpha_v} \) is the shortest occurring formula.
	\end{enumerate}
\end{proposition}

% --------------------
\subsection{Validation}
% --------------------

With \( \hat\alpha \) as the disjunctive companion to \( \alpha \), we now establish the two properties of Theorem~\ref{DNF-thm}: \( \alpha \Vdash \hat\alpha \) and \( \hat\alpha \Vdash \alpha \).
The two witnessing co-proofs are similar in construction. 
As the former is a more involved argument, we present it in detail and sketch the verification of the other claim.
Both claims are implicit in the main result of \cite{ALM21uniformInter}.

Both constructive co-proofs are obtained by modifying the pre-template for \( \alpha \) and inserting the corresponding disjunctive formula \( \hat\alpha_u \) at each vertex.
To confirm that each branch is good it is necessary to project branches of the pre-template \( P \) into the template \( T \).
Regularity of the pre-template means that every \( u \in P \setminus T \) has the form \( u = v w \) where \( v \) is a bud of \( T \) and \( v^c w \in P \).
Thus, the companion function of \( T \) induces a canonical projection \( p \colon P \to T \colon u \mapsto u^p \) of \( P \) into \( T \) defined as: if \( u = v w \) for a bud \( v \in T \) then \( u^p = ( v^c w )^p \); otherwise \( u^p = u \).
Note that \( u^p \) is never a bud of \( T \).

We appeal to the following observation relating invariants to the existence of threads in the pre-template.
Recall that a non-bud vertex is successful iff it is the companion to a successful bud.
\begin{proposition}\label{conj-good-thread}
	Let \( B \) be a branch of the pre-template of \( \alpha \) and let \( u_B \in T \) be the \( \le \)-minimal vertex such that \( u_B = B(i)^p \) for infinitely many \( i \).
	Then \( B \) is good (in the sense of the unannotated calculus where \( (\dup) \) is interpreted as weakening) iff \( u_B \) is a successful companion.
\end{proposition}
For the proof of the proposition it is helpful to fix invariants to companions.
Since buds with the same companion are associated identical invariants we can, without confusion, refer to the invariant of a companion.
\begin{proof}
	Let \( B^p \) be the vertices in \( T \) that are projection of infinitely many vertices in \( B \).
	Let \( C_B \) be the set of companions in \( B^p \).
	There is a unique \( \le \)-minimal element of \( B^p \) and this is also an element of \( C_B \).
	Let \( u_B \) be this vertex and \( ( c , k ) \) the associated invariant.
	By the choice of \( T \), for every \( v \in C_B \) we have \( \inv {u_B} \sqsubseteq \inv v \).
	Let \( c_B \) be the longest prefix of all controls for vertices in \( B^p \).
	In particular, \( c \le c_B \).
	Henceforth, we treat only the template and pre-template rooted at \( u_B \) and assume \( B(0) = u_B \).

	Suppose \( B \) is good and, seeking a contradiction, that \( k = 1 \).
	So \( c_B = c \).
	Consider the plain formula \( \mf \phi \) of lowest complexity for which annotations occur infinitely often in the thread and let \( \tau \) be the suffix of the thread in which all formulas have complexity at least \( \rk{\mf \phi} \).
	Let \( ( \mf[a_i] \phi_i )_i \) enumerate the annotations of \( \mf \phi \) through which \( \tau \) passes.
	Linearity of the annotations means that \( \phi_i = \phi_j \) for all \( i,j \). Let \( \psi = \phi_0 \).
	Let \( a \) be the longest common prefix among the \( a_i \).
	Consider the vertex at which \( \mf[a_i] \psi \) is principal. The minor formula of this rule is the annotated formula \( \psi(\mf[a_i n] \psi) \) for some name \( n \) that is not in the control of the conclusion.
	As \( \mf[a_{i+1}] \psi \) is a descendent of \( \psi(\mf[a_i n] \psi ) \) in the pre-template \( P \), we have that either \( a_i < a_{i+1} \) or there is a sequent on the path between the two formulas at which a name in \( a_i \) is covered.
	Given that there are finitely many annotations in the template, there is a shortest annotation \( a \) such that \( a = a_i \) for infinitely many \( i \) and the last name in \( a \), say \( m \), must be covered in some sequent in \( B^p \).
	Since \( a \) is the prefix of an annotation that occurs in every sequent in \( B^p \) it follows that \( a \) is a subsequence of \( c_B \).
	But then the invariant of the repetition \( ( v^c, v) \) that surrounds the instance of \( (\cover^m) \) must have invariant \( ( c , 0 ) \), contradicting that \( \inv {u_B} \sqsubseteq \inv v \).

	For the converse direction, suppose \( \inv {u_B} = ( c , 0 ) \).	
	Let \( n \) be the final name in \( c \) and
	consider the tree of annotated formula traces through \( B \) that follow only formulas in which \( n \) occurs.
	By the choice of \( u_B \), \( n \) occurs in all controls of \( B^p \), whence  this tree of traces is infinite.
	By weak König's lemma there is therefore an infinite such annotated formula trace. Call this \( t \).
	From the assumption that \( u_B \) is successful, some instance of \( (\cover^n) \) is applied infinitely often in \( B \).
	For this to be the case, \( t \) must infinitely often pass through a formula in which \( n \) is covered, whence \( t \) must pass through a formula \( \mf[an] \phi \) followed by \( \phi(\mf[anm] \phi) \) infinitely often for some \( a \) and \( m \). 
	Linearity of the annotations enforces that it is specifically the formula \( \mf[an] \phi \) which witnesses that \( t \) supports a left \( \mu \)-thread.
\end{proof}

\begin{proposition}
	\label{conj-entail-1}
	\( \alpha \Vdash \hat\alpha \).
\end{proposition}

\begin{proof}
The desired co-proof is obtained from the template of \( \alpha \) by inserting \( \hat\alpha_u \) as the consequent of the sequent at \( u \in T \) and inserting new inferences where appropriate. The definition of \( \hat\alpha_u \) entails that buds continue to have the same sequent as their companion, and identifying buds with companions yields a maximal \( \kmo \)-derivation \( \pi \).
That every branch of \( \pi  \) is good is guaranteed by the choice of \( \hat\alpha \): If a branch of \( \pi \) does not carry a left \( \mu \)-thread then, by virtue of proposition~\ref{conj-good-thread}, the unique right thread carried by the branch will be \( \nu \).
\end{proof}

The converse provability claim, presented next, employs a similar argument but is complicated by the disjunctive interpretation of modalities.
For \( u \in T \), let \( \alpha_u^+ \), \( \alpha_u^- \) be as in definition~\ref{d-conj-comp-prelim}.
For each \( u \in T \), let \( \hat{\alpha}_u^- \) be the subformula of \( \hat\alpha_u \) corresponding to the choice of \( \alpha_u^- \), namely the following formula where \( \epsilon = u_1 < \dotsm < u_n < u \) is an enumeration of the ancestors of \( u \):
\[
\begin{aligned}
	\hat{\alpha}_u^- &\coloneqq \alpha_u^- \subs[x_u]{ \alpha_u^+ } \subs[x_{u_n}]{ \alpha_{u_n}^+ } \dotsm \subs[x_{u_1}]{ \alpha_{u_1}^+ }
	\\
	&\thinspace = \alpha_u^- \subs[x_{u_1}]{ \hat\alpha_{u_1} } \dotsm \subs[x_{u_n}]{ \hat\alpha_{u_n} }  \subs[x_u]{ \hat\alpha_u } 
	&&\text{by definition of the $\hat\alpha_v$.}
	\\
	\hat\alpha_u &\thinspace = \sigma_u x_u\, \hat\alpha_u^-
\end{aligned}
\]
\begin{proposition}
	\label{conj-entail-2}
	\( \hat\alpha \Vdash \alpha \).
\end{proposition}
\begin{proof}
	Let \( P \) be the pre-template of \( \alpha \) with \( c_u : \lseq {\Phi_u} \) being the sequent at \( u \in P \).
	The witnessing co-proof is constructed in two steps.
	First, we co-recursively construct a co-proof of \( \seq \alpha \alpha \) maintaining the invariant that each bud of the partially determined derivation is a sequent \( \seq {(\Phi_u)^-} \gamma \) for an associated \( u \in P \) satisfying \( \gamma \in (\Phi_{u})^- \). 
	In place of the usual right modal rule \( (\sqR) \), however, is an amalgamation of (\( \sqR \)) and the branching rule (\( \diLs \)):
	\begin{prooftree*}

	\hypo{ \seq { \beta_1 , \Phi  } {\delta} }

	\hypod
	
	\hypo{ \seq { \beta_n , \Phi  } {\delta} }
	
	\hypo{ \seq { \Phi  } {\delta} }
	
	\infer4[\sqRs]{ \seq{ \diamond \beta_1 , \dotsc , \diamond \beta_n , \square \Phi , L  }{\square \delta} }
	\end{prooftree*}
	The rule is sound because each premise is a weakening of the final one.
	The motivation for the rule comes from the second step in the process, wherein each antecedent \( (\Phi_{u})^- \) is replaced by the corresponding subformula of \( \hat\alpha \), with the modal case necessitating the new rule above.
	
	The co-proof in the first part is easily to construct.
	At any bud whose sequent satisfies the desired invariant, the derivation is continued by copying the left rules from the pre-template \( P \) until either the rule (\( \diLs \)) is reached or the next rule in \( P \) will destroy the invariant.
	We first treat the case of a non-modal rule in \( P \) which destroys the invariant.
	This means that the construction of the desired co-proof has reached a sequent \( \seq {(\Phi_u)^-} \gamma \) where \( \gamma \) is the principle formula at \( u \).
	In all scenarios we can continue by applying the left rule from \( P \) followed by the corresponding right rule.
	For example, if \( \gamma = \disj \setof{\gamma_1, \dotsc, \gamma_n} \) and \( (\Phi_u)^- = \gamma , \Psi \), the following derivation is inserted.
	\begin{prooftree*}

	\hypo{\seq {\gamma_1, \Psi } {\gamma_1} }
	\infer1[\disjR]{ \seq {\gamma_1, \Psi } {\gamma} }

	\hypod
	
	\hypo{\seq {\gamma_n, \Psi } {\gamma_n} }
	\infer1[\disjR]{ \seq {\gamma_n, \Psi } {\gamma} }
	
	\infer3[\disjL]{ \seq{ \gamma , \Psi }{\gamma} }
	\end{prooftree*}
%	The created buds correspond to successors of \( u \) in \( P \).
	
	Now suppose \( \rules_P(u) = \diLs \) and that \( (\Phi_u)^- = \diamond \Psi , \square \Phi , L \) where \( L \) is a set of literals.
	If \( L \) is inconsistent or \( \gamma \in L \), the sequent is initial, and if \( \gamma \in \diamond \Psi \) the rule \( (\diL) \) is inserted with the corresponding successor of \( u \).
	For \( \gamma = \square \delta \in \square \Phi \) we insert the special modal rule (\( \sqRs \)) presented above with \( \Psi = \beta_1 , \dotsc, \beta_n \).
	The buds of the extended derivation each satisfy the invariant and correspond to the immediate successors of \( u \) in \( P \).
	
	The second step of the argument is now simply the matter of replacing each antecedent \( (\Phi_u)^- \) by the formula \( \hat \alpha_{p(u)} \) where \( p : P \to T \) is the projection of the pre-template \( P \) into the template from which \( \hat\alpha \) is constructed.
	This step requires inserting instances of the left rules (\( \muL \)) and (\( \nuL \)) to accommodate the quantifiers in \( \hat\alpha \). 
	But the logical steps transfer directly, with instances of the new rule (\( \sqRs \)) corresponding to
	\begin{prooftree*}

	\subproof{ \seq { \hat \alpha_{p(u_0)} } {\delta} }

	\hypod
	
	\subproof{ \seq { \hat \alpha_{p(u_{n-1})}  } {\delta} }
	
	\subproof{ \seq { \hat \alpha_{p(u_n)}  } {\delta} }
	
	\infer4[\disjL]{ \seq{ \disj \setof{\hat \alpha_{p(u_0)} , \dotsc, \hat \alpha_{p(u_n)} }  }{ \delta} }
	\infer1[\sqR,\conjL]{ \seq {\hat\alpha_{p(u)}^-} \delta }
	\infer1[\sigL]{ \seq {\hat\alpha_{p(u)}} \delta }
	\end{prooftree*}
	That the constructed derivation is a proof is a consequence of Proposition~\ref{conj-good-thread}.
\end{proof}

Let \( \Lit(\alpha) \) denote the set of literals occurring free in \( \alpha \), i.e., \( \Lit(\alpha) = \setof{ x \in \Lit }[\alpha \sft x] \).
As a corollary of the proof of Theorem~\ref{DNF-thm}, we obtain
\begin{corollary}[Uniform Lyndon interpolation for constructive co-proofs]\label{uni-int}
	For every formula \( \alpha \) and set of literals \( L \), there exists a formula \( \iota \) satisfying
	\begin{enumerate}
		\item \( \Lit(\iota) \subseteq \Lit(\alpha) \cap L \).
		\item For every \( \beta \) with \( \Lit(\beta) \subseteq L \),
		\(
			\Seq \alpha \beta \) iff \( \Seq \iota \beta .
		\)
	\end{enumerate}
\end{corollary}
\begin{proof}
	Let \( \gamma \) be the disjunctive formula constructed by the proof of Theorem~\ref{DNF-thm} and define \( \iota \) by replacing each subformula \( \nabla_K(\Gamma,\Delta) \) of \( \gamma \) by \( \nabla_{K \cap L}(\Gamma,\Delta) \).
	Then \( \Lit(\iota) \subseteq \Lit(\gamma) \cap L \subseteq \Lit(\alpha) \cap L \).
	Also, \( \Seq \iota \beta \) iff \( \Seq \gamma \beta \) provided \( \Lit(\beta) \subseteq L \). 
	Since \( \alpha \VdashV \gamma \), we are done.
\end{proof}

% --------------
\section{Completeness}
\label{s-completeness}
% --------------

We can now 
establish completeness of the finitary proof system \( \kmu \).
The result arises as a corollary of the stronger result that the constructive calculus \( \kmu \) is complete for `negative' sequents, i.e., \( \kmui \vdash \seq {\neg \alpha} \bot \) if \( \alpha \) is valid.
This partial completeness result for the constructive \( \mu \)-calculus arises from the partial interpretation of constructive co-proofs as cyclic proofs (Theorem~\ref{quasi-completeness}) by applying the result to a valid sequent \( \seq{\delta} \bot \) where \( \delta \in \PD \) is such that \( \cmui \vdash \seq {\neg\alpha} \delta \).
As the cyclic and finitary calculi coincide, (full) completeness for \( \cmu \) and \( \kmu \) obtains.

A natural candidate for the \( \PD \)-formula \( \delta \) is the disjunctive normal form given by Theorem~\ref{DNF-thm}.
While this choice is sufficient (and is the choice in \cite{Waluk00comp-lc}), the argument behind \( \cmu \vdash \seq {\neg \alpha} {\delta}  \) is simplified by a more careful definition.

\begin{definition}%[Companion formula]
	\label{d-conj-comp}
	Each formula \( \alpha \) is associated a \emph{companion} formula \( \alpha_d \in \PD \) as follows.
%	The \emph{companion formula} \( \alpha_d \) by induction on \( \alpha \). 
	Let \( \hat\alpha \) be the disjunctive companion of \( \alpha \).
	\begin{itemize}
		\item If \( \alpha \in \Lit \) atomic, \( \alpha_d \coloneqq \alpha \).
		
		\item If \( \alpha = \bigcirc\Gamma \) for \( \bigcirc \in \setof{ \disj, \conj } \), then \( \alpha_d \coloneqq \bigcirc\setof{ \gamma_d }[\gamma \in \Gamma] \).
		
		\item If \( \alpha = \triangle \beta \) for \( \triangle \in \setof{ \square   , \diamond  } \), then \( \alpha_d \coloneqq \triangle  \beta_d \).
		
		\item If \( \alpha = \nf \phi \), then \( \alpha_d \coloneqq \nf {\phi_d} \) where \( \phi_d = \lambda  x.\, ({\phi x})_d \).

		\item If \( \alpha = \mf \phi \), then \( \alpha_d \coloneqq \hat\alpha \).

	\end{itemize}
\end{definition}

The definition of \( \alpha_d \) applies the disjunctive normal form theorem to at most on formula on each \( \alpha \)-thread.
%That is, for \( \alpha \in \PF F \) we have \( \alpha_d \in \PF {F_d} \) where \( F_d = \setof{ \gamma_d }[\gamma \in F] \).
%Indeed, \( \alpha_d \in \PD \) for every \( \alpha \).
%
An induction on the complexity of \( \alpha \), applying Proposition~\ref{guard-reds-f}, transitivity of \( \Vdash \) and the disjunctive normal form theorem, yields
\begin{proposition}
	\label{compl-lemma-1}
	\( \alpha_d \VdashV \alpha \).
\end{proposition}
%
%\begin{proof}
%	By definition and Theorem~\ref{DNF-thm}
%\end{proof}

Inserting Proposition~\ref{compl-lemma-1} into Theorem~\ref{quasi-completeness} gives rise to

\begin{proposition}
	\label{compl-a-hat}
	\( \cmui\vdash \seq{\alpha_d}{\alpha} \).
\end{proposition}
%\begin{proof}
%	%
%	Theorem~\ref{quasi-completeness} on account that \( \alpha_d \Vdash \alpha \) and the former formula is \( \PD \).
%	%
%\end{proof}

Provability of the other implication is by induction on formulas:

\begin{proposition}
	\label{compl-hat-a}
	\( \cmui \vdash \seq\alpha {\alpha_d} \).
\end{proposition}
\begin{proof}
	By induction on \( \alpha \). We treat the quantifier cases as the others are straightforward. 
	Let \( \vdash \) denote provability in \( \cmui \).
	Suppose \( \alpha = \nf \phi \), so \( \alpha_d = \nf \phi_d \). 
	The induction hypothesis yields \( \vdash \seq {\phi x}{\phi_d x}  \). 
	So \( \vdash \seq \alpha {\phi_d \alpha}  \) by substitution and \( \vdash \seq \alpha {\alpha_d} \) by derived \( \nu \)-induction (Proposition~\ref{simul-ind}).
	
	Now suppose \( \alpha = \mf \phi \). 
	Proposition~\ref{compl-lemma-1} implies \( \Seq {\phi_d x} {\phi x} \),  Proposition~\ref{kmoi-mono} gives rise to
	\[ 
		\phi_d \alpha_d
%		\Vdash \phi \alpha_d
		\Vdash \phi \alpha
		\Vdash \mf \phi
		\Vdash \alpha_d
		. 
	\] 
	Therefore, \( \vdash \seq {\phi _d \alpha_d }{\alpha_d} \) by Theorem~\ref{quasi-completeness} as \( \phi_d \alpha_d \in \PD \).
	Also, \( \vdash \seq {\phi \alpha_d } { \phi_d \alpha_d } \) via the induction hypothesis, so \( \vdash \seq {\alpha}{\phi_d \alpha_d} \) by derived \( \mu \)-induction.
	An application of cut yields \( \vdash \seq \alpha {\alpha_d} \).
\end{proof}

The preceding propositions are the final ingredient to the reduction of co-proofs to cyclic proofs.

\begin{theorem}
	\label{kmoi-in-cmui}
	If \( \Phi \Vdash \alpha \) then \( \cmui \vdash \Phi \Rightarrow \alpha \).
\end{theorem}
\begin{proof}
	Let \( \Phi_d = \setof{\beta_d}[\beta \in \Phi] \). Transitivity of \( \Vdash \) implies \( \Phi_d \Vdash \alpha \), whence Theorem~\ref{quasi-completeness} gives rise to \( \cmui \vdash \Phi_d \Rightarrow \alpha \). 
	Proposition~\ref{compl-hat-a} yields \( \cmui \vdash \seq \Phi \alpha \).
\end{proof}

Combining Corollary~\ref{kmo-complete} and Theorems~\ref{kmoi-in-cmui} and \ref{cmui-in-kmui}, we deduce

\begin{corollary}
	If \( \alpha \) is valid, then \( \seq {\neg\alpha} \bot \) is provable in \( \cmui \) and \( \kmui \).\label{completeness-i}
\end{corollary}
%\begin{proof}
%	%
%	Corollary~\ref{kmo-complete} and Theorems~\ref{kmoi-in-cmui} and \ref{cmui-in-kmui}.
%	%
%\end{proof}

The closure of the classical calculi under negation yields completeness these logics:
\begin{corollary}
	\( \kmu \) and \( \cmu \) are complete: If \( \alpha \) is valid then \( \kmu \vdash \rseq \alpha \) and \( \cmu \vdash \rseq \alpha \).\label{completeness}
\end{corollary}
%
%\begin{proof}
%	The previous corollary and equivalence of \( \cmu \) with \( \kmu \) (Theorem~\ref{cmui-in-kmui}) complete the proof.
%\end{proof}

% --------------
\bibliography{}
%\bibliographystyle{fundam}
%\bibliography{paper/bibliography}
% --------------
% ======================
\end{document}